\documentclass[11pt,a4paper]{amsart}

\usepackage{amssymb,amscd,amsmath}
\usepackage{stmaryrd,tikz,hyperref}
\usepackage{mathrsfs,mathdots,mathtools}
\usepackage[mathcal]{eucal}
\usepackage{float,mathabx}
\usepackage{soul,caption}
\usepackage[amsstyle,leqno]{cases}
\usepackage[normalem]{ulem}
\usepackage{hyperref}
\usepackage{cleveref}
\usepackage[UKenglish]{babel}
\usepackage{array}

\usepackage{bm}

\theoremstyle{plain}
\newtheorem{thm}{Theorem}[section]
\newtheorem{lem}[thm]{Lemma}
\newtheorem{pro}[thm]{Proposition}
\newtheorem{cor}[thm]{Corollary}
\newtheorem{con}[thm]{Conjecture}

\theoremstyle{remark}
\newtheorem{rem}[thm]{Remark}
\newtheorem{exm}[thm]{Example}
\newtheorem{dfn}[thm]{Definition}
\newtheorem*{acknowledgements}{Acknowledgements}

\numberwithin{equation}{section}
\numberwithin{table}{section}
\numberwithin{figure}{section}

\newcommand{\N}{\mathbb{N}}
\newcommand{\Z}{\mathbb{Z}}
\newcommand{\Q}{\mathbb{Q}}
\newcommand{\R}{\mathbb{R}}
\newcommand{\mff}{\mathfrak{f}}

\newcommand{\mfh}{\mathfrak{h}}
\newcommand{\mfp}{\mathfrak{p}}

\newcommand{\lri}{\mathfrak{o}}

\renewcommand{\epsilon}{\varepsilon}
\renewcommand{\phi}{\varphi}
\renewcommand{\theta}{\vartheta}
\newcommand{\mcA}{\mathcal{A}}
\newcommand{\mcF}{\mathcal{F}}

\newcommand{\Perm}{S}
\newcommand{\Spec}{\mcS}

\newcommand{\atup}{\boldsymbol{\alpha}}
\newcommand{\rtup}{\boldsymbol{r}}
\newcommand{\stup}{\boldsymbol{s}}
\newcommand{\gtup}{\boldsymbol{\gamma}}
\newcommand{\Xtup}{\bfX}
\newcommand{\Ytup}{\bfY}
\newcommand{\Ztup}{\bfZ}
\newcommand{\bfa}{\boldsymbol{a}}
\newcommand{\bfb}{\boldsymbol{b}}

\DeclareMathOperator{\GMC}{GMC}
\DeclareMathOperator{\MC}{MC}
\DeclareMathOperator{\red}{red}
\DeclareMathOperator{\supp}{supp}
\DeclareMathOperator{\Des}{Des}
\DeclareMathOperator{\Asc}{Asc}
\DeclareMathOperator{\rk}{rk}

\def \no {\textup{n.o.}}  
\def \parti {\mathcal{P}} 
\def \topo {\textup{top}}
\def \bfz {\mathsf{0}}
\def \bfo {\mathsf{1}}
\def \bfX {{\bf X}}
\def \bfY {{\bf Y}}
\def \bfZ {{\bf Z}}
\def \mcC {\ensuremath{\mathcal{C}}}
\def \mcD {\ensuremath{\mathcal{D}}}
\def \mcK {\ensuremath{\mathcal{K}}}
\def \mcL {\ensuremath{\mathcal{L}}}
\def \mcS {\ensuremath{\mathcal{S}}}
\def \mcW {\ensuremath{\mathcal{W}}}
\def \Fq {\ensuremath{\mathbb{F}_q}}
\def \Zp  {\mathbb{Z}_p}

\def \Mat {\mathrm{Mat}}
\def \varq {{q}}
\def \lowBracketLeft {\raisebox{-8.5pt}{\Bigg(}}
\def \lowBracketRight {\raisebox{-8.5pt}{\Bigg)}}
\def \cardres { q_{\mathfrak{o}} }

\author{Viola Siconolfi, Marlies Vantomme, Christopher Voll}
\address{Fakult\"at f\"ur Mathematik, Universit\"at Bielefeld, D-33501
  Bielefeld, Germany} 
\email{\begin{tabular}{l}
		viola.siconolfi@poliba.it, mvantomme@math.uni-bielefeld.de, \\
		C.Voll.98@cantab.net
	\end{tabular}}


\begin{document}

\title[Subgroup growth in free class-$2$-nilpotent groups]{Subgroup
  growth in free class-$2$-nilpotent groups}

\date{\today}
\begin{abstract}
 We describe an effective procedure to compute the local subgroup zeta
 functions of the free class-$2$-nilpotent groups on $d$ generators,
 for all~$d\geq 2$. For $d=4$, this yields a new, explicit
 formula. For $d\in\{4,5\}$, we compute the topological subgroup zeta
 functions. We also obtain general results about the reduced and
 topological subalgebra zeta functions. For the former, we determine
 the behaviour at one; for the latter, the degree and behaviours at
 zero and infinity.  Some of these results confirm, in the relevant
 special cases, general conjectures by Rossmann.
\end{abstract}

\maketitle

\thispagestyle{empty}

\setcounter{tocdepth}{1} \tableofcontents{}

\thispagestyle{empty}

\section{Introduction}\label{sec:intro}

\subsection{Setup}
We study the subgroup growth of finitely generated free nilpotent groups of
nilpotency class two. The free nilpotent group $F_{2,d}$ of nilpotency class
two on $d$ generators has presentation
\begin{equation*} \label{eq:F2d}
  F_{2,d}=\langle g_1,\dots, g_d \mid \forall i,j,k \in \{1,\dots,d\}: [[g_i,g_j],g_k]=1 \rangle.
\end{equation*}
The \emph{subgroup zeta
  function of $F_{2,d}$} is the Dirichlet generating series
\begin{equation*} \label{eq:zetaF2d}
	\zeta_{F_{2,d}}(s)=\sum_{H\leq F_{2,d}} \lvert
        F_{2,d}:H \rvert^{-s},
\end{equation*}
where $s$ is a complex variable (and $\infty^{-s}=0$, so the sum extends in
effect only over the subgroups of finite index). It is well-known that
$\zeta_{F_{2,d}}(s)$ has an Euler product
\begin{equation} \label{equ:euler.alg1}
	\zeta_{F_{2,d}}(s)=\prod_{p\text{ prime}} \zeta_{F_{2,d},p}(s),
\end{equation}
where, for each prime $p$, the Euler factor $\zeta_{F_{2,d},p}(s)$
enumerates the subgroups of $F_{2,d}$ of $p$-power index;
cf.\ \cite[Prop.~4]{GSS/88}. By a result of Grunewald, Segal, and
Smith \textcolor{black}{(\cite[Thm.~2]{GSS/88}) these \emph{local} zeta
  functions are all rational in the parameters $p$ and~$t=p^{-s}$.}
Computing these---and other groups'---local zeta functions explicitly,
however, is difficult. Previously, explicit formulas were only known
for $d\in\{2,3\}$; see~\Cref{subsec:dleq3}.

The free class-$2$-nilpotent Lie ring $\mff_{2,d}$ on $d$ generators has
pre\-sen\-ta\-tion
\begin{equation*} \label{eq:f2d} \mff_{2,d}=\langle x_1,\dots, x_d \mid
  \forall i,j,k \in\{1,\dots,d\}: [[x_i,x_j],x_k]=0 \rangle.
\end{equation*}
The \emph{subalgebra zeta function of $\mff_{2,d}$} is
\begin{equation*} \label{eq:zetaf2d}
\zeta_{\mff_{2,d}}(s)=\sum_{H\leq \mff_{2,d}} \lvert \mff_{2,d}:H \rvert^{-s}.
\end{equation*} 
It is well known that the problem of counting subgroups of the group $F_{2,d}$
is the same as the problem of counting subalgebras of~$\mff_{2,d}$.  Indeed,
the fact that
\begin{equation}\label{fact}
  \zeta_{F_{2,d}}(s) = \zeta_{\mff_{2,d}}(s)
\end{equation}
is not hard to verify; see~\cite[Chap.~4]{GSS/88}.  It justifies that
we concentrate on subalgebra zeta functions in the following.

\label{sec:PAdiceZetaFunctions}
For any (commutative) ring $R$, we set $\mff_{2,d}(R) =
\mff_{2,d}\otimes_{\Z}R$ and consider the subalgebra zeta function
\begin{equation} \label{eq:zetafR}
	\zeta_{\mff_{2,d}(R)}(s) = \sum_{H \leq
		\mff_{2,d}(R)}|\mff_{2,d}(R):H|^{-s},
\end{equation}
enumerating $R$-subalgebras of $\mff_{2,d}(R)$ of finite index. In practice,
we will focus on rings $R$ which are compact discrete valuation rings (cDVRs),
viz.\ rings of integers of finite extensions of the field
  $\mathbb{Q}_p$ of $p$-adic numbers (in characteristic zero) or rings
$\Fq\llbracket T\rrbracket$ of formal power series over finite fields (in
positive characteristic). We write $\mfp$ for the unique maximal ideal of
$\lri$, with residue field cardinality $|\lri/\mfp|=:\cardres$, a prime power.
The Euler product decomposition \eqref{equ:euler.alg1} is
mirrored by the factorization
\begin{equation} \label{eq:eulerproductf2do}
	\zeta_{\mff_{2,d}}(s) = \prod_{p \textup{ prime}} \zeta_{\mff_{2,d}(\Zp)}(s).
\end{equation}

\subsection{Main results}\label{subsec:main.res}
In \cite[Thm.~2]{GSS/88} Grunewald, Segal, and Smith established that
there exists a bivariate rational function
$\zeta_{\mff_{2,d}}(\varq,t)$ such that $\zeta_{\mff_{2,d}}(p,p^{-s})
= \zeta_{\mff_{2,d}(\Zp)}(s)$ for all primes~$p$. Their arguments
extend easily to general cDVRs~$\lri$, that is
$\zeta_{\mff_{2,d}}(\cardres,\cardres^{-s}) =
\zeta_{\mff_{2,d}(\lri)}(s)$. Indeed, the proof of
\cite[Thm.~2]{GSS/88} is based on the observation that whether or not
a submodule $\Lambda$ of the $\Z$-module $\mff_{2,d}(\Z)$ is a
subalgebra may be decided by a combinatorial criterion involving the
elementary divisor types of the $\Z$-modules $\Lambda +
[\mff_{2,d}(\Z),\mff_{2,d}(\Z)]$ and $\Lambda\cap
[\mff_{2,d}(\Z),\mff_{2,d}(\Z)]$. This criterion carries over verbatim
to $\lri$-modules $\mff_{2,d}(\lri)$ for any cDVR~$\lri$. We spell
this out in~\Cref{prop:GSS}.

In the present paper, we present an effective procedure to compute
these bivariate rational functions. The formula for
$\zeta_{\mff_{2,d}}(q,t)$ we present in \Cref{thm:main} is given in
terms of Gaussian $\varq$-multinomials and finitely many generating
functions enumerating the integral points of rational polyhedral
cones.

We apply this formula in different ways, both theoretically and practically.
First, it inspires a notion of \emph{no-overlap subalgebra zeta function} of
$\mff_{2,d}(\lri)$ enumerating, loosely speaking, ``most of'' the finite-index
subalgebras of $\mff_{2,d}(\lri)$, see~\Cref{subsec:novlp}.  We prove that
this summand satisfies the same local functional equation as
$\zeta_{\mff_{2,d}(\lri)}(s)$; see~\Cref{thm:funeq.novlp}.  Second, we derive
from the formula that, for all $d$ and all cDVRs~$\lri$ whose
  residue field cardinalities avoid finitely many values, the $\mfp$-adic
subalgebra zeta function $\zeta_{\mff_{2,d}(\lri)}(s)$ has a simple pole
at~$s=0$, establishing a conjecture of Rossmann for the relevant algebras,
see~\Cref{thm:simple.pole}.  Third, we compute the subalgebra zeta functions
$\zeta_{\mff_{2,4}(\lri)}(s)$ explicitly by implementing the formula in
\texttt{SageMath} \cite{SageMath93} using \texttt{LattE} \cite{LattE} and
\texttt{Zeta} \cite{Zeta041}, see~\Cref{thm:d=4.pad} for a paraphrase and
\href{https://doi.org/10.5281/zenodo.7966735}{10.5281/zenodo.7966735} for full
details.

To consider Euler products such as \eqref{eq:eulerproductf2do} is just one way
to capture information about ``many'' $\mfp$-adic zeta functions
uniformly. Others include the reduced zeta functions pioneered by Evseev
(\cite{Evseev/09}) and the topological zeta functions developed by Rossmann
(\cite{Rossmann/15}). Both may be paraphrased as results of ``setting
$\cardres=1$'', in subtly different ways. Crudely speaking, the reduced zeta
function $\zeta^{\red}_{\mff_{2,d}}(t)$ is the univariate rational function in
$t$ defined as $\zeta_{\mff_{2,d}}(1,t)$.  Equally informally, the topological
zeta function $\zeta^{\topo}_{\mff_{2,d}}(s)$ is the univariate rational
function in $s$ obtained as the first non-zero coefficient of
$\zeta_{\mff_{2,d}}(\varq_\lri,\varq_\lri^{-s})$, expanded in~$\varq_\lri-1$.

In~\Cref{thm:pole_reduced} we show that the reduced subalgebra zeta
function $\zeta^{\red}_{\mff_{2,d}}(t)$ has a pole at~$t=1$ of order
$D:=d+\binom{d}{2}$, which is the $\Z$-rank of $\mff_{2,d}$.  In
\Cref{thm:deg_topo} we establish that $D$ is also the degree of the
topological subalgebra zeta
function~$\zeta^{\topo}_{\mff_{2,d}}(s)$. In~\Cref{thm:top_pole.simple}
we show that the topological zeta function has a simple pole at $s=0$
and compute its residue there. This confirms, in the relevant special
cases, general conjectures by Rossmann.  The topological and reduced
subalgebra zeta functions feature together in \Cref{thm:topo_infi}; it
links the topological zeta function's behaviour at infinity and the
reduced zeta function's residue at~$t=1$.  We also compute the
topological subalgebra zeta functions $\zeta^{\topo}_{\mff_{2,4}}(s)$
and $\zeta^{\topo}_{\mff_{2,5}}(s)$ explicitly using our
implementation of~\Cref{thm:main}; see Theorems~\ref{thm:d=4.top}
and~\ref{thm:d=5.top}.

\subsection{Related work}\label{subsec:rel.work}
For $d\leq 3$, the $\mfp$-adic subalgebra zeta functions
$\zeta_{\mff_{2,d}(\lri)}(s)$---and, as corollaries, their topological
and reduced analogues---have been known for some time; see
\Cref{subsec:dleq3} for explicit formulas and references. For $c>2$,
the subalgebra zeta functions of the free class-$c$-nilpotent Lie rings
on $d$ generators $\mff_{c,d}$ are largely unknown. To our knowledge,
explicit formulas are only known for $(c,d)=(3,2)$, by work of
Woodward; cf.~\cite[Thm.~2.35]{duSWoodward/08}.

The ideal zeta functions $\zeta^{\triangleleft}_{\mff_{2,d}(\lri)}(s)$,
enumerating ideals of finite index, have been computed, for all $d$,
in~\cite{Voll/05a}. This yields, in particular, the (global) ideal zeta
function
$\zeta^{\triangleleft}_{\mff_{2,d}}(s) =
\prod_{p}\zeta^{\triangleleft}_{\mff_{2,d}(\Zp)}(s)$. In analogy
with~\eqref{fact} we have
$\zeta^{\triangleleft}_{\mff_{2,d}}(s) = \zeta^{\triangleleft}_{F_{2,d}}(s)$,
the normal zeta function of the free class-$2$-nilpotent group
$F_{2,d}$, enumerating normal subgroups of finite index.

\subsection{Organization and notation}

In~\Cref{sec:prelim}, we recall some well-known nomenclature and
results.  We consider Gaussian binomial and multinomial coefficients
in~\Cref{subsec:GaussianBinomialMultinomial}; the enumeration of
submodules of $\lri$-modules of finite rank in~\Cref{subsec:mod},
convex polyhedral cones in $\Q^m$ in~\Cref{subsec:cones}; various
monoids in $\N_0^m$, in particular, solution sets of systems of linear
homogeneous Diophantine equations, in \Cref{sec:cones.LHDE}; and
generating functions of subsets of $\N_0^m$, in particular monoids, in
\Cref{sec:GeneratingSeries}.  In~\Cref{sec:FEAIEAC,sec:IEAC}, we
define some notation and prove preliminary results for certain subsets
of a monoid in $\N_0^m$.  In~\Cref{sec:SubsetsUsedIn}, we define the
specific monoids and subsets in $\N_0^m$ that are used in the later
sections to write down formulas for the subalgebra zeta functions
considered.

\Cref{sec:DerivationFormula} culminates in~\Cref{thm:main}, an
explicit formula for $\zeta_{\mff_{2,d}(\lri)}(s)$ as a finite sum,
whose summands are products of Gaussian $\varq$-multinomials and
generating functions of subsets of $\N_0^m$ as discussed in
\Cref{sec:GeneratingSeries}. \textcolor{black}{We give an informal
  overview of the proof of \Cref{thm:main} in
  \Cref{subsec:informal.main}.}
In~\Cref{sec:ResultsPAdicZetaFunction}, we use this explicit formula
to obtain several general results on the $\mfp$-adic subalgebra zeta
functions $\zeta_{\mff_{2,d}(\lri)}(s)$. Notably, we introduce the
no-overlap subalgebra zeta function of $\mff_{2,d}(\lri)$
in~\Cref{subsec:novlp} and show, in~\Cref{sec:simplepole}, that for
all $d$ and all cDVRs $\lri$ whose residue field cardinalities avoid
finitely many values, the $\mfp$-adic subalgebra zeta function
$\zeta_{\mff_{2,d}(\lri)}(s)$ has a simple pole at~$s=0$.

In~\Cref{sec:gen.res} we obtain results on the reduced and topological
zeta functions mentioned in~\Cref{subsec:main.res}. For the former, we
determine the behaviour at $t=1$ and for the latter, the degree and
behaviours at zero and infinity.  In \Cref{sec:expl-form} we record
explicit formulas for $\mfp$-adic, reduced, and topological subalgebra
zeta functions associated with $\mff_{2,d}$ for small values of $d$,
both known and new.

\Cref{tab:notation} gives a partial list of the notation used. 

\begin{table}
	\begin{center}	
		\begin{tabular}{ l p{9.1cm} l}
                  Notation & Meaning & Location \\ \hline
                  $\lri$ &  compact discrete valuation ring &~\Cref{sec:PAdiceZetaFunctions} \\
                  $\cardres$ & cardinality of the residue field $\lri/\mathfrak{p}$, a prime power &~\Cref{sec:PAdiceZetaFunctions} \\
                  $\mff_{2,d}(\lri)$ & tensor product $\mff_{2,d} \otimes_{\mathbb{Z}} \mathfrak{o}$ &~\Cref{sec:PAdiceZetaFunctions} \\
                  $\zeta_{\mff_{2,d}(\lri)}(s)$ & subalgebra zeta function of $\mff_{2,d}(\lri)$ & \eqref{eq:zetafR} \\
                  $\zeta_{\mff_{2,d}}(\varq,t)$ & 
                                                  rational function with $\zeta_{\mff_{2,d}}(\cardres,\cardres^{-s})=\zeta_{\mff_{2,d}(\lri)}(s)$ for all~$\lri$ 
                                     &~\Cref{subsec:main.res} \\
                  $D$ & $d+\binom{d}{2}=\binom{d+1}{2}$, the $\Z$-rank of $\mff_{2,d}$ &~\Cref{subsec:main.res} \\
                  $d'$ & $\binom{d}{2}$ &~\Cref{subsec:GaussianBinomialMultinomial} \\
                  $\binom{n}{J}_{\varq}$ & Gaussian multinomial coefficient &~\Cref{dfn:GaussianMultinomial} \\
                  $\parti_n$ & set of integer partitions of at most $n$ parts &~\Cref{subsubsec:SubgroupsFiniteAbelianPGroups}\\
                  $\nu\leq \mu$ & $\nu_i\leq \mu_i$ for $i\in [n]$ for partitions $\mu,\nu\in \parti_n$ &~\Cref{subsubsec:SubgroupsFiniteAbelianPGroups} \\
                  $\lambda_1^{(n)}$ & $(\lambda_1)_{j\in [n]} \in \parti_n$ &~\Cref{subsubsec:SubgroupsFiniteAbelianPGroups} \\
                  $\alpha(\lambda,\mu;\lri)$ & number of subgroups of isomorphism type $\mu$ of a finite $\lri$-module of isomorphism type $\lambda$ &~\Cref{subsubsec:SubgroupsFiniteAbelianPGroups} \\
                  $\lvert \lambda \rvert$ & $\sum_{i=1}^n \lambda_i$ &~\Cref{subsubsec:torsionfree.o-mod} \\
                  $\overline{F}$ & interior of a monoid $F$ in $\N_0^m$ &~\Cref{sec:cones.LHDE} \\
                  $L(E)$ & lattice of supports of a monoid $E$ &~\Cref{sec:cones.LHDE}\\
                  $X(\Ztup)$ & generating function of $X\subseteq \N_0^m$ &~\Cref{sec:GeneratingSeries} \\
                  $F_{E,A},\overline{F}_{E,A}$ & submonoid resp.\ subset of a monoid $E$ for $A\in L(E)$&~\Cref{sec:FEAIEAC}\\ 
                  $D_{\overline{F}}$ & specific finite subset of a monoid $F$& \eqref{def:D.circ} \\
                  $\Spec_{2d'}$ & subset of the symmetric group $S_{2d'}$ &~\Cref{dfn:mcs}\\
                  $\mcW_d$ & set of relevant pairs $(I,\sigma)$ with $I\subseteq [d-1]$ and $\sigma\in \Spec_{2d'}$ &~\Cref{dfn:mcW} \\ 
                  $G_{I,\sigma}$ & specific monoid of $\N_0^{d+d'}$ for $(I,\sigma)\in \mathcal{W}_d$ &~\Cref{dfn:GISigma} \\
                  $\Xtup$ & tuple of indeterminates $(X_i)_{i\in [d]}$ &~\Cref{subsec:CISigma} \\
                  $\Ytup$ & tuple of indeterminates $(Y_j)_{j\in [d']}$ &~\Cref{subsec:CISigma} \\
                                                  $H_{I,J}$ & specific subset of $\N_0^{d+d'}$ for $I\subseteq [d-1]$ and $J\subseteq [d'-1]$ &~\Cref{dfn:HIJ} \\
$\mcD_{2d'}$ & set of Dyck words of length $2d'$ &~\Cref{subsec:dyck}\\
                  $w_{\sigma}$ & Dyck word associated with $\sigma\in \Perm_{2d'}$ &~\Cref{dfn:DyckWordAssociated} \\      
                  $\mu_{\lambda}$ & partition $(\mu_j)_{j\in [d']}\in \parti_{d'}$ 
                                    such that the multisets $\{ \mu_j \mid j\in [d']\}$ and
                                    $\left\{ \lambda_{i}+\lambda_{i'} \mid i< i'\in [d] \right\}$ coincide, for $\lambda\in\parti_d$
                                     &~\Cref{def:mu_i} \\
                  $\GMC_{I,\sigma}$ & product of Gaussian multinomial coefficients associated with $(I,\sigma)\in \mathcal{W}_d$ &~\Cref{dfn:GMCISigma} \\
                  $\chi_\sigma$ & numerical data map &~\Cref{dfn:chi_sigma} \\
                  $\zeta^w_{\mff_{2,d}}(\varq,t)$ &                                                     rational function with  $\zeta^w_{\mff_{2,d}}(\cardres,\cardres^{-s}) = \zeta^w_{\mff_{2,d}(\lri)}(s)$ for all~$\lri$
                                     &~\Cref{subsec:novlp} \\
                  $\zeta^{\no}_{\mff_{2,d}}(\varq,t)$ &                                                        rational function with $\zeta^{\no}_{\mff_{2,d}}(\cardres,\cardres^{-s}) = \zeta^{\no}_{\mff_{2,d}(\lri)}(s)$ for all~$\lri$
                                     &~\Cref{subsec:novlp} \\
                  $\chi_{\no}$ & no-overlap numerical data map &~\Cref{subsec:novlpalt} \\
                  $\MC_{I,\sigma}$ & product of multinomial coefficients associated to $(I,\sigma)$ &~\Cref{dfn:MC} \\
			\begin{tabular}[t]{@{}l@{}}
				$a_{\sigma}(\alpha), b_{\sigma}(\alpha)$
			\end{tabular}
			& non-negative integers for $\sigma\in \Spec_{2d'}$ and $\alpha\in \N_0^{m_{\sigma}}$ that are closely related to the numerical data map~$\chi_{\sigma}$ &~\Cref{dfn:absigmaalpha} \\
			$U_{I,\sigma,\max}$ & set of $u\in U_{I,\sigma}$ such that $\dim K_u = D$ &~\Cref{dfn:UISigmaMax_and_cd} \\
			$c_d$ & specific positive rational number &~\Cref{dfn:UISigmaMax_and_cd} \\
			$\zeta_{\mff_{2,d}}^{\red}(t)$ & reduced subalgebra zeta function of $\mff_{2,d}$ &~\Cref{sec:red} \\
			$\chi_{\red}$ & reduced numerical data map &~\Cref{dfn:chi_red} \\
			$\zeta_{\mff_{2,d}}^{\topo}(s)$ & topological subalgebra zeta function of $\mff_{2,d}$ &~\Cref{subsec:topo_zeta_funct} \\
		\end{tabular}
	\end{center}
	\caption{Notation.}
	\label{tab:notation}
\end{table}

\section{Preliminaries}\label{sec:prelim}

\subsection{Gaussian binomial and multinomial coefficients}
\label{subsec:GaussianBinomialMultinomial}
We start by recalling Gaussian binomial and multinomial 
coefficients.

\begin{dfn} \label{dfn:GaussianBinomialCoefficient2}
	Let $k,n\in \mathbb{N}_0$ with $k\leq n$.  The \emph{Gaussian
        binomial coefficient} or \emph{$\varq$-binomial coefficient}
        $\binom{n}{k}_{\varq}$ is the following polynomial in $\varq$:
	\begin{equation*} \label{eq:GaussianBinomialCoefficient}
		\binom{n}{k}_{\varq}:=\frac{(1-\varq^{n} )(1-\varq^{n-1} )\cdots(1-\varq^{n-k+1} )}{(1-\varq^{k})(1-\varq^{k-1})\cdots(1-\varq^{1})}.
	\end{equation*}
\end{dfn}

\begin{dfn} \label{dfn:GaussianMultinomial}
	Let $n\in \N_0$ and $J=\{j_i \mid i\in [r]\}\subseteq [n-1]$ with $j_1\leq\dots\leq j_r$. 
	The \emph{Gaussian multinomial coefficient} $\binom{n}{J}_{\varq}$ 
	is the polynomial in $\varq$ defined as 
	\begin{equation*}
		\binom{n}{J}_{\varq}
		:=\binom{n}{j_1}_{\varq} \binom{n-j_{1}}{j_2-j_1}_{\varq}
		\dots \binom {n-j_{r-1}}{j_r-j_{r-1}}_{\varq}.
	\end{equation*}
\end{dfn}

We write $\Perm_n$ for the symmetric group of degree~$n$, a Coxeter
group with Coxeter generators $s_1,\dots,s_{n-1}$. For $\sigma\in
\Perm_n$, we write $\ell(\sigma)$ for the Coxeter length of $\sigma$
and $\Des(\sigma) = \{i \in [n-1] \mid \ell(\sigma s_i) < \ell
(\sigma) \}$ for its (right) descent set. The unique $\ell$-longest
element in $\Perm_n$ is denoted $\sigma_0$, with $\ell(\sigma_0) =
\binom{n}{2}$. The identities
\begin{equation} \label{eq:relationsw0w}
	\ell(\sigma \sigma_0) = \ell(\sigma_0) - \ell(\sigma), \quad \Des(\sigma \sigma_0) 
	= [n-1] \setminus \Des(\sigma),
\end{equation} and 
\begin{equation} \label{eq:GaussianMultinomialDescentSet}
	\binom{n}{J}_{\varq} = \sum_{\sigma \in \Perm_n,\, \Des(\sigma) \subseteq J}
	\varq^{\ell(\sigma)}
\end{equation}
for $J \subseteq [n-1]$ are well-known; {see
  \cite[Proposition~3.2.3]{BjoernerBrenti/05} for \eqref{eq:relationsw0w} and
  \cite[Proposition~1.7.1]{Stanley/12} for
  \eqref{eq:GaussianMultinomialDescentSet}}. We represent permutations
$\sigma\in \Perm_n$ by their \emph{one-line notation}, i.e.\ the word
$\sigma(1) \sigma(2) \dots \sigma(n)$ in the letters $[n]$.

\subsection{Counting submodules of \texorpdfstring{$\mathfrak{o}$}{o}-modules} 
\label{subsec:mod}

We recall some well-known facts about the enumeration of submodules of
finitely generated $\lri$-modules, where $\lri$ is a cDVR. 
We consider \textcolor{black}{finite} modules in
\Cref{subsubsec:SubgroupsFiniteAbelianPGroups} and torsion-free modules in
\Cref{subsubsec:torsionfree.o-mod}. 

\subsubsection{Finite $\lri$-modules} 
\label{subsubsec:SubgroupsFiniteAbelianPGroups} 
Let $\parti_n\subset \N_0^n$ be the set of integer partitions of at most $n$
(non-zero) parts, i.e.\ the set of tuples $\lambda=(\lambda_j)_{j\in [n]}$
with $\lambda_i\in \mathbb{N}_0$ for $i\in [n]$ and
$\lambda_i\geq \lambda_{i+1}$ for $i\in [n-1]$. By convention,
$\lambda_{n+1}=0$ for $\lambda\in \parti_n$.  We call a finite $\lri$-module
\emph{of isomorphism type $\lambda$} if it is isomorphic to the product
$C_{\lri,\lambda}:=\lri/{\mfp^{\lambda_1}}\times \dots \times
\lri/{\mfp^{\lambda_n}}$ of finite cyclic $\lri$-modules. By slight abuse of
notation we call the $\lambda_i$ the \emph{elementary divisors}
of~$C_{\lri,\lambda}$.

Let $\lambda,\mu\in\parti_n$. We write $\mu \leq \lambda$ if
$\mu_i\leq \lambda_i$ for every $i\in [n]$. Let $\alpha(\lambda,\mu;\lri)$ be
the number of submodules of $C_{\lri,\lambda}$ of isomorphism type~$\mu$. The
following formula for $\alpha(\lambda,\mu;\lri)$ was recorded
in~\cite{Butler/87}. \textcolor{black}{(There, it is only stated for the case
  $\lri=\Zp$, i.e.\ for finite abelian $p$-groups. The formula's proof,
  however, works just as well for any cDVR.)} \textcolor{black}{We denote the
  \emph{conjugate partition} of $\lambda$ by $\lambda'$,
  following~\cite[Section~1.8]{Stanley/12}.}

\begin{pro}[\cite{Butler/87}]\label{pro:Burkhoff}
Let $\mu \leq \lambda$ be partitions, with conjugate partitions
$\mu ' \leq \lambda'$. Then
\begin{equation}\label{equ:alpha.1}
	\alpha(\lambda,\mu;\lri)
	=\prod_{k\geq 1} \cardres^{\mu_k'(\lambda_k'-\mu_k')}
    \binom{\lambda_k'-\mu_{k+1}'}{\mu_k'-\mu_{k+1}'}_{\cardres^{-1}}.
\end{equation}
\end{pro}

For later use, we obtain an alternative expression
for~$\alpha(\lambda,\mu;\lri)$ in~\Cref{pro:alter.alpha}.

\begin{dfn} \label{dfn:MjLjMjTj}
Let $\mu \leq \lambda\in\parti_n$. Let $m_{j}\in \N_0$ for $j\in[2n]$
be such that the multisets $\lambda \cup \mu$ and $\{m_j\}_{j\in [2n]}$
are equal and $m_1\geq m_2\geq \dots \geq m_{2n}$. Let $L_0=M_0=0$ and, for $j\in[2n]$,
\begin{align*}
  L_j &= \# \{i\in[n] \mid \lambda_i \geq m_j\},\\
  M_j &= \# \{i\in[n] \mid \mu_i \geq m_j\}.
\end{align*}
\end{dfn}

We note that similar but different integers $L_j$ and $M_j$ are
defined in \cite[(2.13)]{SV1/15}.

\begin{exm} \label{exm:LjMj}
  For $\lambda=(4,2,1)$ and $\mu=(3,2,0)$ in $\parti_3$, we find that
  $(m_j)_{j\in [6]} = (4,3,2,2,1,0)$, $(L_j)_{j\in \{0,\dots,6\}} =
  (0,1,1,2,2,3,3)$, and $(M_j)_{j\in \{0,\dots,6\}} =
  (0,0,1,2,2,2,3)$.  \Cref{fig:IllustrationOfExmLjMj} illustrates this
  example.
\end{exm}

\begin{figure}
	\centering
	\begin{tikzpicture}[scale=0.5, font=\small]
		\def\fillcolorlambda{lime!50!white};
		\def\drawcolorlambda{lime!50!black};
		\def\fillcolormu{magenta!50!white};
		\def\drawcolormu{magenta!50!black};
		\def\locationtext{above};
		
		\begin{scope}[xshift=0.5cm]
		\draw [draw=\drawcolorlambda, fill=\fillcolorlambda] (-5,6) rectangle (-4,10);
		\draw [draw=\drawcolorlambda, fill=\fillcolorlambda] (-4,8) rectangle (-3,10);
		\draw [draw=\drawcolorlambda, fill=\fillcolorlambda] (-3,9) rectangle (-2,10);
		\node[\drawcolorlambda] [\locationtext] at (-4.5,10) {$\lambda_1$};
		\node[\drawcolorlambda] [\locationtext] at (-3.5,10) {$\lambda_2$};
		\node[\drawcolorlambda] [\locationtext] at (-2.5,10) {$\lambda_3$};
		\end{scope}
		
		\begin{scope}[xshift=-0.5cm]
		\draw [draw=\drawcolormu, fill=\fillcolormu] (3,7) rectangle (2,10);
		\draw [draw=\drawcolormu, fill=\fillcolormu] (4,8) rectangle (3,10);
		\draw [draw=\drawcolormu, fill=\fillcolormu] (5,10) rectangle (4,10);
		\node[\drawcolormu] [\locationtext] at (2.5,10) {$\mu_1$};
		\node[\drawcolormu] [\locationtext] at (3.5,10) {$\mu_2$};
		\node[\drawcolormu] [\locationtext] at (4.5,10) {$\mu_3$};
		\end{scope}
		
		\draw[\drawcolorlambda,-stealth]  (-3,7) to[out=-90,in=90] (0,5);
		\draw[\drawcolormu,-stealth] (3,7) to[out=-90, in=90] (0,5);
		
		\begin{scope}[yshift=-1cm]
			\draw [draw=\drawcolorlambda, fill=\fillcolorlambda] (-3,1) rectangle (-2,5);
			\draw [draw=\drawcolorlambda, fill=\fillcolorlambda] (-1,3) rectangle (0,5);
			\draw [draw=\drawcolorlambda, fill=\fillcolorlambda] (1,4) rectangle (2,5);
			\draw [draw=\drawcolormu, fill=\fillcolormu] (-2,2) rectangle (-1,5);
			\draw [draw=\drawcolormu, fill=\fillcolormu] (0,3) rectangle (1,5);
			\draw [draw=\drawcolormu, fill=\fillcolormu] (2,5) rectangle (3,5);
			\node[\drawcolorlambda] [\locationtext] at (-2.5,5) {$m_1$};
			\node[\drawcolorlambda] [\locationtext] at (-0.5,5) {$m_3$};
			\node[\drawcolorlambda] [\locationtext] at (1.5,5) {$m_5$};
			\node[\drawcolormu] [\locationtext] at (-1.5,5) {$m_2$};
			\node[\drawcolormu] [\locationtext] at (0.5,5) {$m_4$};
			\node[\drawcolormu] [\locationtext] at (2.5,5) {$m_6$};
		\end{scope}
	\end{tikzpicture}
	\caption{Illustration of~\Cref{exm:LjMj}.  }
	\label{fig:IllustrationOfExmLjMj}
\end{figure}

The following lemma resembles \cite[Lemmas~2.16 and~2.17]{SV1/15}. 
\begin{lem} \label{pro:alter.alpha}
Let $\lambda,\mu\in\parti_n$ with $\mu\leq \lambda$. Then
	\begin{equation} \label{eq:AlternativeExpressionAlpha}
		\alpha(\lambda,\mu;\lri)
		=\prod_{j=1}^{2n}
        \binom{L_{j}-M_{j-1}}{M_j-M_{j-1}}_{\cardres^{-1}}
        \cardres^{M_j(L_j-M_j)(m_j-m_{j+1})}.
	\end{equation}
\end{lem}
\begin{proof}
The product in \eqref{equ:alpha.1} is indexed by integers~$k$. Suppose
that $j\in [2n]$ and $k\in [\lambda_1]$ are such that $m_{j}\geq k >
m_{j+1}$. Then $\lambda_k'=L_j$ and $\mu_k'=M_j$. Hence
\eqref{equ:alpha.1} reads
\begin{equation*}
	\alpha(\lambda,\mu;\lri)=\prod_{j=1}^{2n} \prod_{k=m_{j+1}+1}^{m_j} \cardres^{M_j(L_j-M_j)} \binom{L_j-\mu_{k+1}'}{M_j-\mu_{k+1}'}_{\cardres^{-1}}.
\end{equation*}
Now $\mu_{k+1}'$ is equal to $M_j$ if $m_{j}> k > m_{j+1}$ and equal
to $M_{i}$ if $k=m_j$ and $i=\max(\{ i\in [2n] \mid M_{i} < M_{j}
\})$.  If $\mu_{k+1}' = M_j$, then
$\binom{L_j-\mu_{k+1}'}{M_j-\mu_{k+1}'}_{\cardres^{-1}}=1$.  Removing
these factors from the product we obtain
\begin{align*}
\alpha(\lambda,\mu;\lri) 
	&= \prod_{j=1}^{2n} \binom{L_j-M_{j-1}}{M_j-M_{j-1}}_{\cardres^{-1}} 
	\prod_{k=m_{j+1}+1}^{m_j} \cardres^{M_j(L_j-M_j)} \\ 
	& = \prod_{j=1}^{2n} 
	\binom{L_{j}-M_{j-1}}{M_j-M_{j-1}}_{\cardres^{-1}}
    \cardres^{M_j(L_j-M_j)(m_j-m_{j+1})}. \qedhere
\end{align*}
\end{proof}

Given $\lambda_1\in \N_0$, we write $\lambda_1^{(n)}$ for $(\lambda_1)_{j\in [n]}\in\parti_n$. 
The following is obvious.
\begin{cor} \label{cor:CorAlpha} 
Let $\lambda\in\parti_n$ and set $I:= \{i\in[n-1] \mid \lambda_{i}>\lambda_{i+1} \}$. Then
	\begin{equation*}
		\alpha\left( \lambda_1^{(n)},\lambda;\lri\right)
		= \binom{n}{I}_{\cardres^{-1}} \prod_{j=1}^{n} \cardres^{j(n-j)(\lambda_j-\lambda_{j+1})}.
	\end{equation*}
\end{cor}

\subsubsection{Free $\lri$-modules}
\label{subsubsec:torsionfree.o-mod}
Let $\pi$ be a uniformiser of $\mathfrak{o}$, i.e.\ a generator of
$\mfp$.

\begin{dfn} \label{dfn:ElementaryDivisorType} 
Let $\Lambda$ be a submodule of $\lri^n$ of finite index. 
Let $\{\pi^{\lambda_j} \}_{j\in [n]}$ with 
$\lambda_1\geq \dots \geq \lambda _n$ be the multiset 
of elementary divisors of $\lri^n/\Lambda$. 
The partition $\lambda=(\lambda_j)_{j\in [n]}\in\parti_n$ 
is the \emph{elementary divisor type} of $\Lambda$, written
$\varepsilon(\Lambda)=\lambda$.
\end{dfn} 

We note that~\Cref{cor:CorAlpha} yields the number of submodules of
$\lri^n$ of elementary divisor type $\lambda$. This proves the
following proposition, which counts submodules of fixed elementary
divisor type.  Given $\lambda=(\lambda_j)_{j\in [n]}\in\parti_n$, we
set $|\lambda|=\sum_{i=1}^n\lambda_i$.

\begin{pro}[{\cite[Section 4.2]{CSV/24}}]
\label{thm:CountingSubmodulesFixedElementaryDivisors} 
Given a partition $\lambda=(\lambda_j)_{j\in [n]}\in\parti_n$,
\begin{equation*}
	\sum_{\substack{\Lambda \leq \lri^n \\ \varepsilon(\Lambda)=\lambda }} 
	\lvert \lri^n: \Lambda \rvert^{-s}
	=\alpha( \lambda_1^{(n)}, \lambda;\lri) \cardres^{-s |\lambda|}.
\end{equation*}
\end{pro}

The following proposition counts submodules containing a given
submodule.
\begin{pro}[{\cite[Section 4.3]{CSV/24}}]
\label{thm:CountingSubmodulesContaining} 
Let $M\leq \lri^n$ be a submodule \textcolor{black}{of finite index} with
elementary divisor type $\varepsilon(M)=\mu$.  Then
\begin{equation*}
	\sum_{\substack{ \Lambda \leq \lri^n \\ \Lambda \geq M}}
    \lvert \lri^n: \Lambda \rvert^{-s}
    =\sum_{\substack{ \nu\in \parti_n \\ \nu \leq \mu }} 
    \alpha( \mu, \nu; \lri ) \; \cardres^{-s |\nu|}.
\end{equation*}
\end{pro}
\begin{proof} 
Observe that $\lri^n/M \cong C_{\lri,\mu}$ and
$|\lri^n:\Lambda|=\cardres^{|\nu|}$ if $\varepsilon(\Lambda)=\nu$.
\end{proof}

\subsection{Convex polyhedral cones in~\texorpdfstring{$\mathbb{R}^m$}{Qm}} \label{subsec:cones}

We recall some general nomenclature and results for convex polyhedral cones.
We largely follow \cite[p. 477]{Stanley/12}.

The \emph{dimension} $\dim A$ of a subset $A\subseteq \R^m$ is the 
dimension of the subspace of $\R^m$ generated by~$A$.
A \emph{cone} in $\R^m$ is a subset $\mcC\subseteq \R^m$ that is
closed under addition and scaling by non-negative real numbers. 
The convex cone generated by $A\subseteq \R^m$ is
\begin{equation*} \label{eq:ConeGeneratedBy}
	\mcC_A:= \{ a_1x_1+\dots +a_t x_t \mid t\geq 0,\,  
	x_1,\dots, x_t\in A,\, a_1,\dots, a_t\in \R_{\geq 0}\}.
\end{equation*}
A \emph{linear half-space} of $\mathbb{R}^m$ is a subset of 
$\mathbb{R}^m$ of the form $\mathcal{H}= \{v\in\mathbb{R}^m \mid 
 w\cdot v\geq 0\}$ for a vector $w\in \R^m\setminus\{0\}$. 
A \emph{convex polyhedral cone} $\mcC$ is the intersection of
finitely many half-spaces. 
It is \emph{pointed} if it does not contain a line.

Let $\mcC$ be a convex polyhedral cone in~$\mathbb{R}^m$.
A \emph{supporting hyperplane} for $\mcC$ is a hyperplane 
$\mathcal{H}$ that divides $\mathbb{R}^m$ into two linear
half-spaces $\mathcal{H}^+$ and $\mathcal{H}^-$ such that $\mcC$
$\subseteq \mathcal{H}^{+}$ or $\mcC\subseteq \mathcal{H}^{-}$.  
A \emph{face} $\mathcal{F}$ of $\mcC$ is either an intersection 
$\mcC\cap \mathcal{H}$ of $\mcC$ with a supporting hyperplane $\mathcal{H}$
or equal to~$\mcC$.
Faces of dimension one are called
\emph{extreme rays} and faces of dimension $m-1$ are called \emph{facets}.
The convex polyhedral cone $\mcC$ is \emph{simplicial} if it has exactly
$\dim \mcC$ extreme rays.

For $x\in \mathbb{R}^m$ and $\varepsilon>0$, let $N_{\varepsilon}(x)$
be the closed ball of radius $\varepsilon$ around~$x$.  Let $A$ be any
subset of $\R^m$ and $\textup{aff}(A)$ the affine hull of~$A$.  The
\emph{relative interior} $\textup{relint}(A)$ of $A$ is the set of
points $a\in A$ such that there is an $\epsilon>0$ with
$N_{\varepsilon}(x)\cap \textup{aff}(A)$ contained in~$A$.  If $A$ is
a convex polyhedral cone, then the relative interior of $A$ is the set
of points in $A$ that are not contained in any face of $A$ of lower
dimension than~$A$.

\begin{dfn}[{\cite[p. 477]{Stanley/12}}]
	\label{def:tria}
	Let $\mcC$ be a convex polyhedral cone in~$\mathbb{R}^m$.
	A \emph{triangulation} of $\mcC$ is a finite family 
	$\Gamma=\{\mcK_{u}\} _{ u\in U}$ of simplicial polyhedral cones 
	$\mcK_{u}$ such that the following hold:
	\begin{itemize}
		\item $\mcC=\bigcup_{u\in U}\mcK_u$,
		\item for each $\mcK_u\in \Gamma$, every face of $\mcK_u$ is
		an element of $\Gamma$, and
		\item for every $\mcK_u,\mcK_v\in \Gamma$, the intersection 
		$\mcK_u\cap \mcK_v$ is a common face of $\mcK_u$ and~$\mcK_v$.
	\end{itemize}
	An element of $\Gamma$ is called a \emph{face} of~$\Gamma$.
\end{dfn}

\begin{rem} \label{rem:RelativeInteriorTriangulation}
	If $\Gamma=\{\mcK_{u}\}_{ u\in U }$ is a triangulation of $\mcC$,
	then $\mcC=\bigcup_{u\in U} \textup{relint}(\mcK_u)$ and this union is disjoint.
\end{rem}

\begin{pro}[{\cite[Lem.~4.5.1]{Stanley/12}}]
	\label{pro:tria}
	Every pointed polyhedral cone $\mcC$ has a triangulation 
	whose one-dimensional faces are the extreme rays of~$\mcC$.  
\end{pro}

Let $P$ be a poset with partial order relation~$\leq_P$.  An \emph{interval}
in~$P$ is a subset of~$P$ of the form
$[x,y]=\{ z\in P \mid x \leq_P z \leq_P y \}$ for some $x\leq_P y\in P$.  An
interval is \emph{non-trivial} if $x<_P y$.  The element $x$ \emph{covers} $y$
in $P$ if $[y,x]=\{x,y\}$ \textcolor{black}{and $x\neq y$}.  The poset $P$ is
\emph{graded} if it is endowed with a rank function $\rk:P\rightarrow \N_0$,
i.e.\ a function satisfying $\rk(x)>\rk(y)$ if $x>_P y$ in~$P$ and
$\rk(x)=\rk(y)+1$ if $x$ covers~$y$.  A graded poset $\mathcal{P}$ is
\emph{Eulerian} if, in any non-trivial interval, the number of elements of
even rank and the number of elements of odd rank coincide.

Let $\mcC$ be a convex polyhedral cone in~$\mathbb{R}^m$.  The
\emph{lattice of faces} $L(\mcC)$ is the poset consisting of the faces
of $\mcC$ ordered by inclusion, together with an additional
element~$\varnothing$, called the \emph{empty face}. We set
$\dim(\varnothing)=-1$.  The following is well known and is a
consequence of {\cite[Cor.~3.5.4]{BeckSanyal/18} and
  \cite[Cor.~3.3.3]{BeckSanyal/18}}.
\begin{pro}\label{prop_poset}
  The lattice of faces $L(\mcC)$ of $\mcC$ is Eulerian:
  \textcolor{black}{the rank of a face $\mathcal{F}$ is the length of
    $[\varnothing,\mathcal{F}]$, viz.\ the length of a saturated chain
    from the empty face to $\mathcal{F}$.}  If $\mcC$ is pointed, then
  the rank of a face $\mathcal{F}$ is $\dim(\mathcal{F})+1$.
\end{pro}

\subsection{Monoids in \texorpdfstring{$\mathbb{N}_0^m$}{N0m}} 
\label{sec:cones.LHDE} 
We discuss monoids in $\mathbb{N}_0^m$, in particular those that are
associated with systems of homogeneous linear equations with integer
coefficients. We largely follow \cite[Sec.~4.5]{Stanley/12}.

A \emph{monoid} $F$ in $\mathbb{N}_0^m$ is a subset of
$\mathbb{N}_0^m$ that contains zero and is closed under addition.  The
\emph{interior} of $F$, denoted by $\overline{F}$, is the set of
points in $F$ that lie in the relative interior of~$\mcC_F$.  The
\emph{completely fundamental} elements $\textup{CF}(F)$ of $F$ are the
elements $\beta\in F$ such that if $n\in \N$ and $\alpha,\alpha'\in F$
are such that $n\beta = \alpha + \alpha'$, then $\alpha=i\beta$ and
$\alpha'=(n-i)\beta$ for some $i\in \N_0$ with $i\leq n$.  A system of
$r$ homogeneous linear equations with integer coefficients in $m$
variables $\alpha_1$, \dots, $\alpha_m$ can always be written as $\Phi
(\alpha_1,\dots,\alpha_m)^{\top}=0$ for an $r\times m$ matrix $\Phi$
over~$\mathbb{Z}$.  The set of solutions in $\mathbb{N}_0^m$ of this
system,
\begin{equation} \label{eq:ConeDiophantineEquation}
	E = E_\Phi :=\{\alpha\in \mathbb{N}_{0}^m \mid 
	\Phi \alpha=0\},
\end{equation}
is a monoid in~$\mathbb{N}_0^m$. 
\begin{rem} \label{rem:associatedCone}
The convex cone $\mcC_E$ generated by $E$ is a pointed convex
polyhedral cone. The completely fundamental elements of $E$ each
generate an extreme ray of $\mcC_E$ and vice versa.
\end{rem}

\begin{rem} \label{rem:AssumptionE}
We may assume that the rank of $\Phi$ is $r$ by deleting dependent
rows of~$\Phi$. If $E\cap \N^m=\emptyset$, then there must be an $i\in
[m]$ such that the $i$-th entry of $\alpha$ is $0$ for every
$\alpha\in\mcC$. In this case, we can just ignore the $i$-th
coordinate. In general, we assume that no coordinates are redundant
and that $E\cap \N^m\neq \emptyset$. Then the interior $\overline{E}$
is the set $E\cap \mathbb{N}^m$ of positive points in~$E$.
\end{rem}

\begin{rem} \label{rem:SlackVariables}
 Through slack variables, \eqref{eq:ConeDiophantineEquation} can be
 used to study monoids defined by linear inequalities as well.
 Concretely, let $\Phi\in\Mat_{k\times m}(\Z)$ and consider the monoid
 $\mcS = \left\{v\in\N_0^m \mid \Phi v\geq 0\right\}$.  The points in
 $\mcS$ are in bijection with the points in
\begin{equation*}
\left\{v\in\N_0^m, \gamma \in\N_0^k \mid \Phi v- \gamma = 0 \right\},
\end{equation*}
	where $\gamma$ is a tuple of $k$ (slack) variables. 
\end{rem}	

A monoid $F$ in $\mathbb{N}_0^m$ is \emph{simplicial} if there exist
$\Q$-linearly independent tuples $\alpha_1$, \dots, $\alpha_t\in F$,
called \emph{quasigenerators} of $F$, such that
\begin{equation*}
	F=\{ \gamma\in \mathbb{N}^m \mid \exists n\in \mathbb{N}, 
	\exists a_1,\dots,a_t\in \mathbb{N}_0: 
	n\gamma=a_1\alpha_1+\dots+a_t\alpha_t \}.
\end{equation*}
A monoid $F$ is simplicial if and only if $\mcC_F$ is a simplicial
polyhedral cone. In this case, the completely fundamental elements
$\textup{CF}(F)$ are quasigenerators of~$F$.  The
interior~$\overline{F}$ of a simplicial monoid $F$ can be
characterised by
\begin{equation*}
	\overline{F}=\{ \gamma\in \mathbb{N}^m \mid \exists n\in \mathbb{N}, 
	\exists a_1,\dots,a_t\in \mathbb{N}: 
	n\gamma=a_1\alpha_1+\dots+a_t\alpha_t \}.
\end{equation*}

The \emph{support} of a tuple
$x=(x_1,\ldots,x_m)\in \color{black}{\mathbb{R}^m}$ is the set
$\supp(x):=\{i\in [m] \mid x_i\neq 0\}$.  The \emph{support} of a set
$V\subseteq \mathbb{R}^m$ is $\supp(V):=\bigcup_{v\in V} \supp(v)$.  Suppose
that $E=E_\Phi$ for some $\Phi\in \Mat_{r\times m}(\mathbb{Z})$.  The
\emph{lattice of supports} $L(E)$ of $E$ is the set
$\{ \supp(\alpha) \mid \alpha\in E \}$ of supports of tuples in $E$, ordered
by inclusion.  The next result identifies the posets $L(\mcC_E)$ and~$L(E)$.

\begin{thm}[{\cite[p. 479]{Stanley/12}}]\label{theorem_poset_bij}
	The map
	\begin{equation*} \label{eq:theorem_poset_bij}
		L(\mcC_E) \rightarrow L(E):\mcF \mapsto
		\supp(\mcF)
	\end{equation*}
	is a poset isomorphism.
\end{thm} 
By~\Cref{rem:AssumptionE} we may assume that $L(E)$ has $[m]$ as
greatest element.

\subsection{Generating functions of subsets of \texorpdfstring{$\mathbb{N}_0^m$}{N0m}}
\label{sec:generating_functions_cones} 
\label{sec:GeneratingSeries} 

We discuss a generating function associated with subsets of
$\mathbb{N}_0^m$, in particular monoids and their interiors. We
largely follow \cite[Sec.~4.5]{Stanley/12}.

For a subset $X\subseteq \mathbb{N}_0^m$, define the generating function
\begin{equation} \label{eq:GeneratingSeriesCone}
	X(\Ztup) :=
	\sum_{\atup\in X}
	\Ztup^{\atup}  
	\in \Q\llbracket \Ztup \rrbracket,
\end{equation}
in the indeterminates $\Ztup=(Z_j)_{j\in [m]}$, where
$\Ztup^{\atup}=\prod_{j\in [m]} Z_j^{\alpha_j}$ for
$\atup=(\alpha_j)_{j\in[m]}\in\N_0^m$.  The sets~$X$ whose generating
functions we consider are often monoids~$F$ or $E=E_\Phi$, or the
\textcolor{black}{interiors}~$\overline{F}$ or~$\overline{E}$ of such monoids.

Consider a monoid $F$ in $\mathbb{N}_0^m$ that is simplicial with
quasigenerators $\alpha_1,\dots,\alpha_t$.  Define the following
finite subsets of $F$ which depend on the choice of quasigenerators:
\begin{align*}
	D_{F}&:=\{x\in F \mid  
	x=a_1\alpha_1+\ldots+a_t\alpha_t,\ 	0\leq a_i<1 \}, \\
	D_{\overline{F}}&:=\{x\in F\mid  
	x=a_1\alpha_1+\ldots+a_t\alpha_t,\  0< a_i\leq 1  \}.\label{def:D.circ} 
\end{align*}

\begin{thm}[{\cite[Cor.~4.5.8]{Stanley/12}}]\label{theorem_generating_function_simplicial_cones}
Let $F \subseteq \mathbb{N}_0^m$ be a simplicial monoid with
quasigenerators $\alpha_1,\dots,\alpha_t$. The generating functions
$F(\Ztup)$ and $\overline{F}(\Ztup)$ are rational and given by:
	\begin{align}
		F(\Ztup)&= \frac{
			\sum_{\beta\in D_{F}} \Ztup^{\beta}
		}{
			\prod_{i=1}^{t}\left(1-\Ztup^{\alpha_i}\right)
		},\nonumber\\ 
		\overline{F}(\Ztup) &= \frac{
			\sum_{\beta\in D_{\overline{F}}} \Ztup^{\beta}
		}{
			\prod_{i=1}^{t}(1-\Ztup^{\alpha_i})
		}.
	\end{align}
\end{thm}

For monoids $E$ of the form \eqref{eq:ConeDiophantineEquation}, we have the following result.
\begin{thm}[{\cite[Theorem 4.5.11]{Stanley/12}}]
	\label{thm:genfun.cones}
	Let $E=E_\Phi$ for some $\Phi\in \Mat_{r\times
          m}(\mathbb{Z})$. The generating functions $E(\Ztup)$ and
        $\overline{E}(\Ztup)$ are rational. When written in lowest terms, both these rational functions have denominator
	\begin{equation*}
		\prod_{\beta\in \textup{CF}(E)}\left(1-\Ztup^{\beta}\right).
	\end{equation*}
\end{thm}

The following two theorems are reciprocity results for $X(\Ztup)$, for
simplicial monoids and monoids of the form
\eqref{eq:ConeDiophantineEquation} respectively.  Let
$\Ztup^{-1}=(Z_j^{-1})_{j\in [m]}$.

\begin{thm}[{\cite[Lemma 4.5.13]{Stanley/12}}]
Let $F\subseteq \mathbb{N}^m$ be a simplicial monoid of
dimension~$n$. Then
	\begin{equation*}
		\overline{F}(\Ztup^{-1}) = (-1)^n F(\Ztup).
	\end{equation*}
\end{thm}

\begin{thm}[{\cite[Theorem 4.5.14]{Stanley/12}}]\label{sta}
	Let $E=E_\Phi$ for some $\Phi\in \Mat_{r\times
          m}(\mathbb{Z})$, \textcolor{black}{ assume that $E\subseteq
          \mathbb{N}^m$} and write $n=\dim(E)$. Then
	\begin{equation*}
		\overline{E}(\Ztup^{-1}) = (-1)^n E(\Ztup).
	\end{equation*}
\end{thm}

\subsection{The submonoids \texorpdfstring{$F_{E,A}$}{FE,A} and subsets \texorpdfstring{$\overline{F}_{E,A}$}{FE,A} of \texorpdfstring{$E=E_\Phi$}{E=EPhi}}
\label{sec:FEAIEAC}
Let $E=E_\Phi$ for some $\Phi\in \Mat_{r\times m}(\mathbb{Z})$.
We define submonoids $F_{E,A}$ and subsets $\overline{F}_{E,A}$ of~$E$ where~$A$ is an element of the lattice of supports $L(E)$ of~$E$.
We formulate some reciprocity results for their generating functions.

For $A\in L(E)$, define
\begin{align*}
	F_{E,A}:=\{ \alpha \in E \mid \supp(\alpha)\subseteq A \}, \label{eq:FEA}\\
	\overline{F}_{E,A}:=\{ \alpha \in E \mid \supp(\alpha)= A \}. 
\end{align*} 
The first set $F_{E,A}$ is a submonoid of~$E$.  The convex cone
$\mcC_{F_{E,A}}$ is the unique (cf.~\Cref{theorem_poset_bij}) face of
$\mcC_E$ whose support is~$A$.  In other words, $\mcC_{F_{E,A}}$ is
the inverse image of $A$ under the map in
\eqref{eq:theorem_poset_bij}.  The second set $\overline{F}_{E,A}$ is
a subset of $E$ and is the interior of~$F_{E,A}$.  Clearly,
$E=F_{E,[m]}$ and
\begin{equation}
	F_{E,A}=\bigcup_{B\in L(E), B\subseteq A} \overline{F}_{E,B},  \label{eq:disjointUnionFEA}
\end{equation}
where the union is disjoint.

\begin{rem} \label{rem:FEAequalsEPhiA}
	If $\alpha\in F_{E,A}$, then the $i$-th coordinate of $\alpha$
        is zero for all $i\in[m] \backslash A$.  \textcolor{black}{In
          what follows we focus on the support of these submonoids.}
        Therefore, the coordinates $[m] \backslash A$ of $F_{E,A}$ may
        be discarded.  \textcolor{black}{We may thus identity}
        $F_{E,A}=E_{\Phi_A}$, where $\Phi_A$ is \textcolor{black}{a}
        matrix found by removing the columns $[m] \backslash A$ from
        $\Phi$ and deleting resulting dependent rows.
\end{rem}

\begin{rem}
	For every $A\in L(E)$, let $\Ztup_A$ be the tuple $(Z_j)_{j\in [m]}$ where $Z_j$ is an indeterminate for $j\in A$ and $Z_j=0$ for $j\in [m]\backslash A$. Then $F_{E,A}(\Ztup)=E(\Ztup_A)$.
\end{rem}

The following is a reciprocity result for the submonoids $F_{E,A}
\subseteq E$ and subsets $\overline{F}_{E,A} \subseteq E$.
\begin{pro} \label{pro:reciprocityFEA}
	Let $E=E_\Phi$ for $\Phi\in \Mat_{r\times m}(\mathbb{Z})$. For all $A\in L(E)$,
\begin{align} 
 \overline{F}_{E,A} (\Ztup^{-1}) &=(-1)^{ \dim F_{E,A} } F_{E,A}
 (\Ztup) = (-1)^{ \dim
   \overline{F}_{E,A} } \sum_{B\in L(E),B\subseteq A}
 \overline{F}_{E,B} (\Ztup). \label{eq:reciprocityGEA}
	\end{align}
\end{pro}
\begin{proof}
The first equality is an application of~\Cref{sta} to $F_{E,A}$, which
is applicable because of~\Cref{rem:FEAequalsEPhiA}.  Recall that
$\overline{F}_{E,A}$ is the interior of $F_{E,A}$. Using
\eqref{eq:disjointUnionFEA}, the second equality follows.
\end{proof}

\Cref{coro_complement} states an alternative reciprocity result for
$\overline{F}_{E,A}$ that is analogous to \cite[Lemma 2.17]{Voll/11}.
To prove it, we need the following lemma.
\begin{lem} \label{prop_eulerian_poset}
	Let $E=E_\Phi$ for $\Phi\in \Mat_{r\times m}(\mathbb{Z})$. For all $A,C\in L(E)$ with $A\subseteq C$,
	\begin{equation} \label{eq:prop_eulerian_poset}
		\sum_{B\in L(E) , A\subseteq B\subseteq C} (-1)^{\dim F_{E,B}}
		=\begin{cases}
			(-1)^{\dim F_{E,A} } & \textup{if } A = C, \\
			0 & \textup{otherwise.}
		\end{cases}
	\end{equation}
\end{lem}
\begin{proof}
	If $A=C$, then the summation in \eqref{eq:prop_eulerian_poset}
        has exactly one summand, namely $(-1)^{\dim F_{E,A} }$.  In
        general, the set $\{ B\in L(E)\mid A\subseteq B\subseteq C \}$
        is an interval in $L(E)$.  If $A\neq C$, then it is a
        non-trivial interval.  Combining~\Cref{prop_poset} with
        ~\Cref{theorem_poset_bij}, we find that $L(E)$ is an Eulerian
        poset where the rank of $B\in L(E)$ is~$\dim F_{E,B}+1$.  Recall
        that Eulerian means that in every non-trivial interval, the
        number of elements of even rank and the number of elements of
        odd rank coincide.  Therefore, the summation
        \eqref{eq:prop_eulerian_poset} completely cancels out if
        $A\neq C$.
\end{proof}

\begin{cor} \label{coro_complement}
 Let $n$ be the dimension of~$E$. Let $A\subseteq \supp(E)$ be such
 that $A\in L(E)$ or $\supp(E)\backslash A\in L(E)$. Then
\begin{equation} \label{eq:coro_completement}
		\sum_{B \in L(E), B\supseteq A}
                \overline{F}_{E,B}(\Ztup^{-1}) =(-1)^{n}\sum_{C\in
                  L(E),C\supseteq \supp(E)\backslash A}
                \overline{F}_{E,C}(\Ztup).
	\end{equation}
\end{cor}
\begin{proof}[Proof (adapted from {\cite[Lemma 2.17]{Voll/11}})]
 Using \eqref{eq:reciprocityGEA}, we find that
	\begin{align}
		\sum_{B\in L(E), B\supseteq A} \overline{F}_{E,B}(\Ztup^{-1})
		&= \sum_{\substack{B\in L(E), B\supseteq A}} (-1)^{ \dim \overline{F}_{E,B} } 
		\sum_{C\in L(E), C\subseteq B} \overline{F}_{E,C} (\Ztup) \nonumber\\
		&= \sum_{C \in L(E)} \left( \sum_{\substack{B\in L(E),B\supseteq C\cup A}} (-1)^{\dim F_{E,B}} \right) \overline{F}_{E,C} (\Ztup), \label{eq:coro_completement2}
	\end{align}
	where we used that $\dim \overline{F}_{E,B}=\dim F_{E,B}$.  If
        $A\in L(E)$, then also $A\cup C\in L(E)$ for $C \in L(E)$.
        Therefore~\Cref{prop_eulerian_poset} implies that the
        expression between brackets in \eqref{eq:coro_completement2}
        is $(-1)^{d}$ when $C\cup A= \supp(E)$ or, equivalently,
        $C\supseteq \supp(E)\backslash A$, and zero otherwise. Thus
        \eqref{eq:coro_completement} holds if $A\in L(E)$.  Notice
        that the roles of $A$ and $\supp(E)\setminus A$ in
        \eqref{eq:coro_completement} are symmetric, so
        \eqref{eq:coro_completement} also holds if $\supp(E)\setminus
        A\in L(E)$.
\end{proof}

\subsection{The subsets \texorpdfstring{$I_{E,A,C}$}{IE,A,C} of \texorpdfstring{$E=E_\Phi$}{E=EPhi}} \label{sec:IEAC}
Let $E=E_\Phi$ for $\Phi\in \Mat_{r\times m}(\mathbb{Z})$.
We define subsets $I_{E,A,C}$ of $E$ for all $A,C\in L(E)$ with $A\subseteq C$. 
We show a reciprocity result for their generating functions and specify a decomposition as a disjoint union of interiors of simplicial monoids.

For all $A,C\subseteq [m]$ with $A\subseteq C$, define
\begin{equation}\label{dfn:C.AB}
	I_{E,A,C}:=
	\{\alpha\in E \mid A \subseteq \supp(\alpha) \subseteq C \}.
\end{equation} 
In other words, the elements in $I_{E,A,C}$ are the elements of $E$
that have positive entries in the coordinates indexed by elements in
$A$, non-negative entries in the coordinates indexed by elements in
$C\backslash A$, and zeros elsewhere.  Obviously,
\begin{equation}
	I_{E,A,C}
	=\bigcup_{B\in L(E), A\subseteq B\subseteq C} \overline{F}_{E,B}, \label{eq:disjointUnionHEAB}
\end{equation}
where the union is disjoint.

We formulate a reciprocity result for $I_{E,A,C}$ when $A,C\in L(E)$ and $C\backslash A\in L(E)$.
\begin{pro}\label{cor:inversion_variable}
	Let $E=E_\Phi$ for $\Phi\in \Mat_{r\times m}(\mathbb{Z})$. Suppose that $A,C\in L(E)$ with $A\subseteq C$ and $C\backslash A\in L(E)$. Then
	\begin{equation} \label{eq:inversion_variable}
		I_{E,A,C}(\Ztup^{-1}) 
		=(-1)^{\dim I_{E,A,C}} I_{E,C\backslash A,C}(\Ztup).
	\end{equation}
\end{pro}
\begin{proof}
	By \eqref{eq:disjointUnionHEAB},
	\begin{equation}
		I_{E,A,C}(\Ztup^{-1}) = \sum_{B\in L(E), A\subseteq B\subseteq C} \overline{F}_{E,B}(\Ztup^{-1}) .
	\end{equation}
	By \eqref{eq:reciprocityGEA}, it follows that
	\begin{align}
		I_{E,A,C}(\Ztup^{-1})
		&=\sum_{B\in L(E), A\subseteq B\subseteq C} (-1)^{ \dim \overline{F}_{E,B} } 
		\sum_{D\in L(E), D\subseteq B} \overline{F}_{E,D} (\Ztup) \\
		&=\sum_{D\in L(E), D \subseteq C} \left( \sum_{B\in L(E), A\cup D \subseteq B \subseteq C} (-1)^{ \dim \overline{F}_{E,B} } \right) \overline{F}_{E,D} (\Ztup).\nonumber
	\end{align}
	Since, $A$ and $D$ are in $L(E)$, it follows that $A\cup D\in
        L(E)$. Therefore by~\Cref{prop_eulerian_poset}, the expression
        between brackets is $(-1)^{\dim \overline{F}_{E,C}}$ when
        $A\cup D=C$ or, equivalently, $C\backslash A\subseteq D$, and
        zero otherwise. Thus we find
	\begin{equation*}
		I_{E,A,C}(\Ztup^{-1})
		=\sum_{D\in L(E), C\backslash A \subseteq D\subseteq C} (-1)^{\dim \overline{F}_{E,C}} \overline{F}_{E,D} (\Ztup).
	\end{equation*}
	Since $(-1)^{\dim \overline{F}_{E,C}}$ does not depend on $D$, it may be pulled out of the summation. Using \eqref{eq:disjointUnionHEAB}, and the fact that $\dim \overline{F}_{E,C}=\dim I_{E,A,C}$, we then find \eqref{eq:inversion_variable}.
\end{proof}

The following proposition gives a decomposition of $I_{E,A,C}$ as a disjoint union of interiors of simplicial monoids.
\begin{pro} \label{pro:TriangulationIEAC}
	Let $E=E_\Phi$ for $\Phi\in \Mat_{r\times m}(\mathbb{Z})$ and $A,C\in L(E)$ with $A\subseteq C$. 
	There is a finite family $\Gamma=\{ K_u \}_{ u\in U }$ of simplicial monoids $K_u\subseteq \N_0^{m}$ such that 
	\begin{itemize}
		\item $I_{E,A,C}=\bigcup_{u\in U} \overline{K}_u$ and this union is disjoint,
		\item $\textup{CF}(K_u) \subseteq \textup{CF}(E)$ for all $u\in U$.
	\end{itemize}
\end{pro}
\begin{proof}
	\eqref{eq:disjointUnionHEAB} already writes $I_{E,A,C}$ as a
        disjoint union of interiors of monoids $F_{E,B}$, but these
        are not simplicial in general.  By~\Cref{rem:FEAequalsEPhiA}
        and~\Cref{rem:associatedCone}, $\mcC_{F_{E,B}}$ is a pointed
        convex polyhedral cone with extreme rays generated by the
        completely fundamental elements of $F_{E,B}$.  Since
        $\mcC_{F_{E,B}}$ is a face of $\mcC_{E}$, the completely
        fundamental elements of $F_{E,B}$ are completely fundamental
        elements of~$E$.  Therefore, by~\Cref{pro:tria},
        $\mcC_{F_{E,B}}$ has a triangulation $\Gamma_{B}=\{ \mcK_u
        \}_{ u\in U_{B} }$, where each $\mcK_u$ for $u\in U_B$ is a
        simplicial polyhedral cone and the one-dimensional faces are
        generated by completely fundamental elements of~$E$.  Let
        $U_{B}^\circ$ be the set of $u\in U_{B}$ such that $\mcK_u$ is
        not contained in any of the facets of $\mcC_{F_{E,B}}$.
        Equivalently, $U_{B}^\circ$ is the set of $u\in U_{B}$ such
        that $\textup{relint}(\mcK_u)$ is contained in
        $\textup{relint}(\mcC_{F_{E,B}})$.  Set $K_u= \mcK_u \cap
        F_{E,B}$ for $u\in U_{B}^\circ$. Then
	\begin{equation*}
		\overline{F}_{E,B}
		=\textup{relint}(\mcC_{F_{E,B}})\cap F_{E,B}
		=\bigcup_{u\in U_{B}^\circ} \textup{relint}(\mcK_u) \cap F_{E,B}
		=\bigcup_{u\in U_{B}^\circ} \overline{K_u}.
	\end{equation*}
	Setting $U=\bigcup_{B\in L(E), A \subseteq B \subseteq C} U_{B}^{\circ}$,
	we find using \eqref{eq:disjointUnionHEAB} that $\bigcup_{u\in U} \overline{K_u}=I_{E,A,C}$ and this union is disjoint because the $\overline{\mcK_u}$ coming from the same triangulation $\Gamma_{B}$ are disjoint, and the union in \eqref{eq:disjointUnionHEAB} is also disjoint.
\end{proof}

\section{The subsets \texorpdfstring{$G_{I,\sigma}$}{GI,sigma} and \texorpdfstring{$H_{I,J}$}{HI,J} of \texorpdfstring{$\mathbb{N}_0^{d+d'}$}{N0d+d'}} 
\label{sec:SubsetsUsedIn}
In this section, we define the specific subsets of $\N_0^m$, for
various~$m$, that are used in the later sections to write down
formulas for the considered subalgebra zeta functions.
In~\Cref{sec:MonoidEsigma,sec:Eno}, we define monoids
$E_{\sigma}\subseteq \N_0^{m_{\sigma}}$ and $E_\no\subseteq
\N_0^{d+d'+1}$.  In~\Cref{subsec:CISigma,sec:altdes}, we discuss
subsets $G_{I,\sigma}$ of $\N_0^{d+d'}$, which are used to express a
formula for $\zeta_{\mff_{2,d}(\lri)}(s)$
in~\Cref{subsec:finitesumformula} and $\zeta^w_{\mff_{2,d}(\lri)}(s)$
in~\Cref{subsec:novlp}.  In~\Cref{subsec:HIJ}, we discuss subsets
$H_{I,J}$ of $\N_0^{d+d'}$, which are used to express a formula for
$\zeta^{\no}_{\mff_{2,d}(\lri)}(s)$ in~\Cref{subsec:novlpalt}.  We
also prove some properties of the associated generating functions that
are used to prove properties of $\zeta_{\mff_{2,d}(\lri)}(s)$,
$\zeta^w_{\mff_{2,d}(\lri)}(s)$, and
$\zeta^{\no}_{\mff_{2,d}(\lri)}(s)$
in~\Cref{sec:ResultsPAdicZetaFunction}.  Dyck words and the relation
between the $G_{I,\sigma}$ and $H_{I,J}$ are discussed
in~\Cref{subsec:dyck}.

\subsection{The monoids \texorpdfstring{$E_{\sigma}\subseteq \N_0^{m_{\sigma}}$}{EsigmaN0msigma}} 
\label{sec:MonoidEsigma}
Set $d':=\binom{d}{2}$.  We define monoids $E_{\sigma}\subseteq
\N_0^{m_{\sigma}}$ for certain permutations $\sigma\in \Perm_{2d'}$.
This is done by defining a matrix $\Phi_{\sigma}$ and setting
$E_{\sigma}=E_{\Phi_{\sigma}}$ as in
\eqref{eq:ConeDiophantineEquation}.

The permutations $\sigma\in \Perm_{2d'}$ for which we define a monoid
$E_\sigma$ are the following:
\begin{dfn} \label{dfn:mcs}
	Let $\Spec_{2d'}$ be the set of permutations $\sigma\in \Perm_{2d'}$ such that 
	\begin{equation} \label{eq:cond1sigma}
		\lvert \{ l\in [i] \mid \sigma(l)\leq d' \} \rvert
		\leq \lvert \{ l\in [i] \mid \sigma(l)> d' \} \rvert
	\end{equation}
	for all $i\in [2d']$ and if $i<j\in [2d']$ are such that $\sigma(i),\sigma(j)\in [d']$ and $\sigma(i)>\sigma(j)$, then
	\begin{equation*}
		\sigma(i)>\sigma(i+1)>\dots>\sigma(j-1)>\sigma(j).
	\end{equation*}
\end{dfn}

\begin{exm}
Let $d=d'=3$. Then $123456 \not\in\Spec_{2d'}$ as
\eqref{eq:cond1sigma} is not satisfied for $i\in [5]$. Also
$653421\not\in \Spec_{2d'}$ because $3<5$ are such that $\sigma(3)=3$,
$\sigma(5)=2\in [3]$, and $\sigma(3)=3>\sigma(5)=2$, yet
$\sigma(3)=3\not>\sigma(4)=4>\sigma(5)=2$. However, $451632\in
\Spec_{2d'}$.
\end{exm}

We fix an indexing of the set $\left\{ (i,j)\in [d]^2 \mid
i<j\right\}\sqcup [d']$ by $[2d']$.

\begin{dfn} \label{dfn:BijectionB} Define the bijection 
	\begin{align*}
		b: \left\{ (i,j)\in [d]^2 \mid i<j\right\}\sqcup [d'] &\to [2d']:\\
		(i,j) & \mapsto d'+j-1+(i-1)(2d-2-i)/2\\ 
		j & \mapsto j.
	\end{align*}
\end{dfn}

\begin{rem}
	The map $b$ respects the lexicographical ordering of the pairs $(i,j)$ with
	$i<j$.
\end{rem}

\begin{exm}
	If $d=4$, then $b$ maps
	\begin{align*}
		& 1\mapsto 1, \quad 4\mapsto 4, \quad (1,2)\mapsto 7, \quad (2,3)\mapsto 10, \\
		& 2\mapsto 2, \quad 5\mapsto 5, \quad (1,3)\mapsto 8, \quad (2,4)\mapsto 11, \\
		& 3\mapsto 3, \quad 6\mapsto 6, \quad (1,4)\mapsto 9, \quad (3,4)\mapsto 12.
	\end{align*}
\end{exm}

Next, we associate a tuple of length $d+d'$ to every element of
$[2d']$.  For $i\in [d+d']$, let $\delta_i\in \N_0^{d+d'}$ be the
$i$th unit basis vector.  Recall that we write $x^{(m)}$ for the tuple
$(x)_{j\in [m]}$.
\begin{dfn}[Corresponding tuple] \label{dfn:CorrespondingVector} 
 Let $i\in [d']$. The \emph{tuple $v_i$ corresponding to} $i$ is
	\begin{equation*}
		v_i:=\sum_{k=d+i}^{d+d'} \delta_k =(0^{(d+i-1)},
                1^{(d'-i+1)}) \in \N_0^{d+d'}.
	\end{equation*}
	Let $i\in d'+[d']$ and $b^{-1}(i)=(j,k)$. The \emph{tuple
        $v_{i}$ corresponding to $i$} is
	\begin{equation*}
		v_{i}:=\sum_{l=j}^{d} \delta_l+\sum_{l=k}^{d} \delta_l
		=(0^{(j-1)}, 1^{(k-j)}, 2^{(d-k+1)},0^{(d')})  \in \N_0^{d+d'}.
	\end{equation*}
	For $i,j\in [d+d']$, let $v_{i,j}$ be the $j$-th component of $v_i$.
\end{dfn}

\begin{exm} \label{exm:CorrespondingVector}
	Let $d=3$ and $i=4$. Then $b^{-1}(i)=(1,2)$ and $v_{4}=(1,2,2,0,0,0)$. 
	Now let $i=5$. Then $b^{-1}(i)=(1,3)$ and $v_{5}=(1,1,2,0,0,0)$.
\end{exm}

Using the integers $v_{i,j}$, we define the matrix $\Phi_{\sigma}$ and
monoid $E_{\sigma}$.

\begin{dfn} \label{dfn:PhiISigma}
	For $\sigma\in \Spec_{2d'}$ set $R_{\sigma}:= \left\{ i \in
        [2d'-1] \mid \sigma(i)> d' \textup{ or } \sigma(i+1)>d'
        \right\}$, $r_{\sigma}:=\lvert R_{\sigma} \rvert$, and
        $m_{\sigma}:=d+d'+r_{\sigma}$.  Let
        $w_{i,j}^{\sigma}:=\allowbreak
        v_{\sigma(i),j}-v_{\sigma(i+1),j}$ for $i\in[2d'-1]$ and $j\in
        [d+d']$.  Then $\Phi_{\sigma}\in\Mat_{r_\sigma\times
          m_\sigma}(\Z)$ is the matrix whose row corresponding to $i
        \in R_{\sigma}$ has
	\begin{center}
		\begin{tabular}{rl}
			$w_{i,j}^{\sigma}$ & in column $j\in [d+d']$,\\ 
			$-1$ & in column $d+d'+i$,\\ 
			$0$ & in the remaining columns.
		\end{tabular}
	\end{center}
	Let $E_{\sigma}$ be the monoid 
	$E_{\Phi_{\sigma}}$ associated with the matrix $\Phi_{\sigma}$
	as in \eqref{eq:ConeDiophantineEquation}.
\end{dfn}

\begin{rem}
	The matrix $\Phi_{\sigma}$ in~\Cref{dfn:PhiISigma} is only
        defined for permutations $\sigma\in \Spec_{2d'}$, rather than
        all~$\sigma\in \Perm_{2d'}$. \textcolor{black}{This ensures that
          $E_{\sigma}\setminus \{(0^{m_{\sigma}})\}$ is non-empty
          (see~\Cref{pro:prIequalsG}); cf.~\Cref{rem:AssumptionE}.}
\end{rem}

\setcounter{MaxMatrixCols}{11}
\begin{exm}
	Let $\sigma = 451632\in \Spec_6$. Then $r_{\sigma}=4$, $m_{\sigma}=10$, and
	\begin{equation*}
		\Phi_{\sigma}=
		\begin{pmatrix}
			0 & 1 & 0 & 0 & 0 & 0 & -1 & 0 & 0 & 0 \\
			1 & 1 & 2 & -1 & -1 & -1 & 0 & -1 & 0 & 0 \\
			0 & -1 & -2 & 1 & 1 & 1 & 0 & 0 & -1 & 0 \\
			0 & 1 & 2 & 0 & 0 & -1 & 0 & 0 & 0 & -1 
		\end{pmatrix}.
	\end{equation*}
\end{exm}

\subsection{The subsets \texorpdfstring{$G_{I,\sigma}\subseteq \mathbb{N}_0^{d+d'}$}{GI,sigmaN0d+d'}} 
\label{subsec:CISigma}
We define subsets $G_{I,\sigma}\subseteq \mathbb{N}_0^{d+d'}$ for
certain pairs $(I,\sigma)$ with $I\subseteq [d-1]$ and $\sigma\in
\Spec_{2d'}$.  They are used in
\Cref{subsec:finitesumformula,subsec:novlp} to write down formulas for
$\zeta_{\mff_{2,d}(\lri)}(s)$ and $\zeta^w_{\mff_{2,d}(\lri)}(s)$.
The relevant pairs $(I,\sigma)$ for which we define a set
$G_{I,\sigma}$ are the following:
\begin{dfn}\label{dfn:mcW}
Let $\mcW_d$ be the set of pairs $(I,\sigma)$ with $I\subseteq [d-1]$
and $\sigma\in \Spec_{2d'}$ such that the following system of
inequalities in the variables $r_1$, \dots, $r_{d}$ has non-zero
solutions in~$\N_0^d$:
	\begin{equation} \label{eq:mcW2}
		\begin{cases}
			r_i>0 & \textup{for }i\in I, \\
			r_i=0 & \textup{for }i\in [d'-1]\backslash I, \\
			\sum_{k=1}^d (v_{i,k}-v_{j,k}) r_k > 0 
			& \textup{for } i,j\in d+[d'] \textup{ with } \sigma^{-1}(i)<\sigma^{-1}(j) \textup{ and } i<j, \\
			\sum_{k=1}^d (v_{i,k}-v_{j,k}) r_k \geq 0 
			& \textup{for } i,j\in d+[d'] \textup{ with } \sigma^{-1}(i)<\sigma^{-1}(j) \textup{ and } i>j.
		\end{cases}
	\end{equation}
\end{dfn}

\begin{exm}
	Let $d=2$, so $d'=1$. The set $\Spec_{2}$ contains only $21$. 
	If $I=\emptyset$, then \eqref{eq:mcW2} reduces to one equation: $r_1=0$,
	and therefore any $r_2\in \N$ together with $r_1=0$ gives a non-zero solution.
	If $I=\{1\}$, then \eqref{eq:mcW2} reduces to one inequality: $r_1>0$,
	and therefore any pair $r_1\in \N$, $r_2\in \N_0$ gives a non-zero solution.
	Thus $\mcW_2=\{ (\emptyset,21), (\{1\},21) \}$.
\end{exm}

For $\sigma\in \Perm_{2d'}$, let
\begin{align*}\Asc(\sigma) &:=\{i\in [2d'-1] \mid \sigma(i)<\sigma(i+1)\},\\
  \Des(\sigma) &:= \{i\in [2d'-1] \mid \sigma(i)>\sigma(i+1)\}
\end{align*}
We call elements of $\Asc(\sigma)$ \emph{ascents} and elements of
$\Des(\sigma)$ \emph{descents} of~$\sigma$.  Let $\rtup$ and $\stup$
be short for $r_1,\dots,r_d$ and $s_1,\dots,s_{d'}$ respectively.

\begin{dfn} \label{dfn:GISigma} 
	Write $\mathbb{N}_0^{d+d'} = \left\{ (\rtup, \stup) \mid
        r_i,s_j\in\mathbb{N}_0\right\}$.  For $(I,\sigma)\in \mcW_d$,
        the set $G_{I,\sigma}$ is the set of tuples $(\rtup,\stup) \in
        \mathbb{N}_0^{d+d'}$ that satisfy the following equations and
        inequalities:
	\begin{numcases}{}
		r_i>0 & \textup{ for } i\in I, \label{eq:GISigma1} \\
		r_i= 0 & \textup{ for } i\in [d-1]\backslash I, \label{eq:GISigma2} \\
		\sum_{j=1}^d w_{i,j}^{\sigma} r_j
		+\sum_{j=1}^{d'} w_{i,d+j}^{\sigma} s_j 
		> 0 & \textup{ for } i\in \Asc(\sigma), \label{eq:GISigma3} \\
		\sum_{j=1}^d w_{i,j}^{\sigma} r_j
		+\sum_{j=1}^{d'} w_{i,d+j}^{\sigma} s_j 
		\geq 0 & \textup{ for } i\in \Des(\sigma). \label{eq:GISigma4}
	\end{numcases}
\end{dfn}
\begin{exm}
	Let $d=2$ and $(I,\sigma)=(\{1\},21)$. Then $G_{I,\sigma}$ is the set of tuples $(r_1, r_2, s_1)\in \N_0^3$ such that $r_1>0$ and $r_1+2r_2-s_1\geq 0$.
\end{exm}
\begin{rem}
 \textcolor{black}{By design of the sets $\mcW_d$ in \Cref{dfn:mcW},
   the sets $G_{I,\sigma}$ in \Cref{dfn:GISigma} are non-empty.}
\end{rem}

\begin{rem}\label{rmk:dim_GISigma}
	The dimension of $G_{I,\sigma}$ is $d+d'-\left\lvert
        [d-1]\backslash I \right\rvert$. \textcolor{black}{Indeed,
          only the equations \eqref{eq:GISigma2} cut down the
          dimension from the ambient space $\N_0^{d+d'}$, namely by
          $|[d-1]\setminus I|$. Hence the maximum
	\begin{equation*}
		\max\left\{ \dim G_{I,\sigma} \mid (I,\sigma)\in \mcW_d \right\} =
		d+d',
	\end{equation*}
	 is attained for the pairs $(I,\sigma)\in \mathcal{W}_d$ with
         $I=[d-1]$.}
\end{rem}

We describe which entries $s_j$ of $(\rtup,\stup)\in G_{I,\sigma}$ are
positive.
\begin{dfn} \label{dfn:Jsigma}
	For $\sigma\in \Spec_{2d'}$, define
	\begin{equation*} \label{eq:Jsigma}
		J_{\sigma}:=\{j\in [d'-1] \mid \sigma^{-1}(j)<\sigma^{-1}(j+1) \}.
	\end{equation*}
\end{dfn}
\begin{exm}
	Let $d=2$ and $(I,\sigma)=(\{1\},21)$. Then $J_\sigma=\emptyset$.
\end{exm}

\begin{pro} \label{pro:Jsigmasj}
	Let $(\rtup,\stup)\in G_{I,\sigma}$. 
	If $j\in J_{\sigma}$, then $s_j>0$.
	If $j\in [d'-1]\setminus J_{\sigma}$, then $s_j=0$.
\end{pro}
\begin{proof}
	Let $j\in J_{\sigma}$,
        i.e.\ $\sigma^{-1}(j)<\sigma^{-1}(j+1)$.  Summing the common
        left-hand side of \eqref{eq:GISigma3} and \eqref{eq:GISigma4}
        over $i\in [\sigma^{-1}(j),\sigma^{-1}(j+1)-1]$ results in
        $s_j$.  There necessarily is an ascent in the interval
        $[\sigma^{-1}(j),\sigma^{-1}(j+1)-1]$.  Therefore, summing
        \eqref{eq:GISigma3} over $i\in
              [\sigma^{-1}(j),\sigma^{-1}(j+1)-1]\cap \Asc(\sigma)$
              and \eqref{eq:GISigma4} over $i\in
              [\sigma^{-1}(j),\sigma^{-1}(j+1)-1]\cap \Des(\sigma)$
              results in $s_j>0$.  Now let $j\in [d'-1]\backslash
              J_{\sigma}$, i.e.\ $\sigma^{-1}(j)>\sigma^{-1}(j+1)$.
              Summing the common left-hand side of \eqref{eq:GISigma3}
              and \eqref{eq:GISigma4} over $i\in
              [\sigma^{-1}(j+1),\sigma^{-1}(j)-1]$ results in $-s_j$.
              Because $\sigma\in \Spec_{2d'}$, there can only be
              descents in the interval
              $[\sigma^{-1}(j+1),\sigma^{-1}(j)-1]$.  Thus summing
              \eqref{eq:GISigma4} over $i\in
                    [\sigma^{-1}(j+1),\sigma^{-1}(j)-1]$ yields
                    $-s_j\geq 0$, whence~$s_j=0$.
\end{proof}

\subsection{Alternative description of \texorpdfstring{$G_{I,\sigma}$}{GI,sigma}} \label{sec:altdes}
In \eqref{dfn:C.AB} we defined subsets $I_{E,A,C}$ of $E$, where $A$
and $C$ encoded which entries were positive and non-negative,
respectively. We now describe $G_{I,\sigma}$ using such a set
$I_{E,A,C}$ where $E=E_{\sigma}$ from~\Cref{sec:MonoidEsigma}.

\begin{dfn} \label{dfn:AISigmaBISigma}
	For $\sigma\in \Spec_{2d'}$, let $\{j_i\mid i\in [r_\sigma]\}
        = R_{\sigma}$ with $j_1\geq \dots \geq j_{r_{\sigma}}$.  For
        every $(I,\sigma)\in \mcW_d$, let $A_{I,\sigma}$ and
        $C_{I,\sigma}$ be the following subsets of $[m]$:
	\begin{align*}
		A_{I,\sigma}&:= I\cup \left(d+J_\sigma\right) \cup
		\left(d+d'+ \left\{ i\in [r_{\sigma}] \mid j_i \in \Asc(\sigma) \right\}  \right), \\
		C_{I,\sigma} &:= I \cup \left(d+J_\sigma\right) \cup \left\{d,d+d'\right\} \cup
		\left(d+d'+[r_\sigma]\right).
	\end{align*}
\end{dfn}
\begin{exm}
	Let $d=2$ and $(I,\sigma)=(\{1\},21)$. Then $A_{\{1\},21}=\{1\}$ and $C_{\{1\},21}=\{1,2,3,4\}$.
\end{exm}

\begin{pro}
	\label{pro:prIequalsG}
	Let $(I,\sigma)\in \mcW_d$.  Let $\textup{pr}: \N^{m_{\sigma}}
        \to \N^{d+d'} $ be the projection map which ignores the last
        $r_{\sigma}$ coordinates.  Restricting this projection map to
        the subset $I_{E_{\sigma},A_{I,\sigma},C_{I,\sigma}} \subseteq
        E_{\sigma} \subseteq \N^{m_{\sigma}}$ results in a bijection
        between $I_{E_{\sigma},A_{I,\sigma},C_{I,\sigma}}$ and
        $G_{I,\sigma}$.
\end{pro}
\begin{proof} Let $\gtup$ be short for $\gamma_1,\dots,\gamma_{r_{\sigma}}$.
	Suppose that $(\rtup,\stup, \gtup) \in
        I_{E_{\sigma},A_{I,\sigma},C_{I,\sigma}}$.  Then
        \eqref{eq:GISigma1} is satisfied because $I\subseteq
        A_{I,\sigma}$ and \eqref{eq:GISigma2} is satisfied because
        $([d-1]\backslash I)\cap C_{I,\sigma}=\emptyset$.  Also
        \eqref{eq:GISigma4} is satisfied for all $i\in R_{\sigma}$
        because of the definition of $\Phi_\sigma$ and
        $d+d'+[r_{\sigma}]\subseteq C_{I,\sigma}$.  If $i$ is,
        moreover, an ascent, then \eqref{eq:GISigma3} holds because
        $d+d'+ \left\{ i\in [r_{\sigma}] \mid j_i \in \Asc(\sigma)
        \right\}\in A_{I,\sigma}$.  If $i\in \Des(\sigma)\backslash
        R_{\sigma}$, i.e.\ $\sigma(i+1)<\sigma(i) \in [d']$, then
        $\sigma(i+1)+1=\sigma(i)$ (because $\sigma\in \Spec_{2d'}$)
        and therefore \eqref{eq:GISigma4} simplifies to
        $-s_{\sigma(i+1)}\geq 0$.  As $\sigma(i+1)$ is in $ [d'-1]
        \backslash J_{\sigma}$ and therefore $d+\sigma(i+1)$ is not in
        $C_{I,\sigma}$, it follows that $s_{\sigma(i+1)}=0$ and
        therefore \eqref{eq:GISigma4} holds.  If $i\in
        \Asc(\sigma)\backslash R_{\sigma}$,
        i.e.\ $\sigma(i)<\sigma(i+1)\leq d'$, then \eqref{eq:GISigma3}
        simplifies to $\sum_{j=\sigma(i)+1}^{\sigma(i+1)} s_j > 0$.
        If $\sigma^{-1}(\sigma(i)+1)<i$, then $\sigma(i)+1\in
        J_{\sigma}$ and therefore $d+\sigma(i)+1\in d+J_{\sigma}
        \subseteq A_{I,\sigma}$ and $s_{\sigma(i)+1}>0$.  If
        $\sigma^{-1}(\sigma(i)+1)>i$, then $\sigma(i)\in J_{\sigma}$
        and therefore $d+\sigma(i)\in d+J_{\sigma} \subseteq
        A_{I,\sigma}$ and $s_{\sigma(i)}>0$.  In any case
        \eqref{eq:GISigma3} holds.  Thus we have that $(\rtup,
        \stup)\in G_{I,\sigma}$.
	
	The restricted projection map
        $\textup{pr}|_{I_{E_{\sigma},A_{I,\sigma},C_{I,\sigma}}}$ is
        injective because $\gtup$ are slack variables, therefore are
        uniquely determined by $(\rtup, \stup)$.  To prove that it is
        surjective, let $(\rtup, \stup)\in \nobreak G_{I,\sigma}$.
        Again because $\gtup$ are slack variables, we can find
        $(\gamma_j)_{j\in [r_\sigma]}$ such that $(\rtup,\stup, \gtup)
        \in E_{\sigma}$.  Because of \eqref{eq:GISigma1} and
        \eqref{eq:GISigma2}, we know that for $i\in[d-1]$, $r_i> 0$ if
        and only if $i\in A_{I,\sigma}$ and otherwise $i\not\in
        C_{I,\sigma}$.  Using~\Cref{pro:Jsigmasj}, we find that for
        $j\in[d'-1]$, $s_j> 0$ if and only if $d+j\in (d+J_{\sigma})
        \subseteq A_{I,\sigma}$ and otherwise $d+i\not\in
        C_{I,\sigma}$.  We conclude that $(\rtup,\stup, \gtup) \in
        I_{E_{\sigma},A_{I,\sigma},C_{I,\sigma}}$.
\end{proof}

Let $\Xtup=(X_i)_{i\in [d]}$, $\Ytup=(Y_j)_{j\in [d']}$ and
$\Ztup=(Z_k)_{k\in [r_{\sigma}]}$ be tuples of indeterminates.  The
generating series enumerating the elements of $E_{\sigma}$,
$G_{I,\sigma}$, and $I_{E_{\sigma},A_{I,\sigma},C_{I,\sigma}}$ as in
\Cref{sec:generating_functions_cones} are denoted by
$E_{\sigma}(\bfX,\Ytup,\mathbf{Z})$, $G_{I,\sigma}(\bfX,\Ytup)$, and
$I_{E_{\sigma},A_{I,\sigma},C_{I,\sigma}}(\bfX,\Ytup,\mathbf{Z})$
respectively.  By~\Cref{pro:prIequalsG}, it follows that
$I_{E_{\sigma},A_{I,\sigma},C_{I,\sigma}}(\Xtup,\Ytup,\mathbf{1})=G_{I,\sigma}(\Xtup,\Ytup)$,
where $\mathbf{1}$ is the all-one tuple of length $r_{\sigma}$.

Often, we will use the following subdivision of
$I_{E_{\sigma},A_{I,\sigma},C_{I,\sigma}}$ into simplicial monoids,
which exists by~\Cref{pro:TriangulationIEAC}.
\begin{dfn} \label{dfn:GammaISigma}
	For every $(I,\sigma)\in \mathcal{W}_d$, let $\Gamma_{I,\sigma}=( K_u )_{u\in U_{I,\sigma} }$ be a family of simplicial monoids $K_u\subseteq \N_0^{m_{\sigma}}$ such that 
	\begin{itemize}
		\item $I_{E_{\sigma},A_{I,\sigma},C_{I,\sigma}}=\bigcup_{u\in U_{I,\sigma}} \overline{K}_u$ and this union is disjoint,
		\item $\textup{CF}(K_u) \subseteq \textup{CF}(E_{\sigma})$ for all $u\in U$.
	\end{itemize}
\end{dfn}

\begin{exm} \label{exm:GammaISigma}
	Let $(I,\sigma)=(\emptyset,21)$. Then $m_\sigma=4$ and
        $E_\sigma$ contains all tuples $(r_1, r_2, s_1,
        \gamma_1)\allowbreak\in \N_0^4$ such that
        $r_1+2r_2-s_1-\gamma_1=0$. Moreover, $A_{I,\sigma}=\emptyset$
        and $C_{I,\sigma}=\{2,3,4\}$.  It follows that
        $I_{E_\sigma,A_{I,\sigma},C_{I,\sigma}}$ contains all tuples
        $(r_1, r_2, s_1, \gamma_1)\in \N_0^4$ such that
        $2r_2-s_1-\gamma_1=0$ and $r_1=0$.  One possible
        $\Gamma_{I,\sigma}=\{ K_u \mid u\in U_{I,\sigma} \}$ is the
        following family $( K_0, K_1, K_2, K_3)$:
	\begin{align*}
		K_0&=\{ (0, 0, 0, 0) \}, \\
		K_1&=\{ (0, r_2, s_1, 0)\in \N_0^4 \mid 2r_2-s_1=0 \}, \\
		K_2&=\{ (0, r_2, 0, \gamma_1)\in \N_0^4 \mid 2r_2-\gamma_1=0 \}, \\
		K_3&=\{ (0, r_2, s_1, \gamma_1)\in \N_0^4 \mid 2r_2-s_1-\gamma_1=0 \},
	\end{align*}
	where $K_3$ is simplicial because it has quasigenerators $(0,1,2,0)$ and $(0,1,0,2)$.
\end{exm}

\subsection{The monoid \texorpdfstring{$E_\no\subseteq \N_0^{d+d'+1}$}{En.o.N0d+d'+1}} \label{sec:Eno}
We define a monoid $E_\no\subseteq \N_0^{d+d'+1}$, again via a matrix
$\Phi_\no$.  We list its completely fundamental elements, define a
specific subset~$E_0$, and describe a specific triangulation of the
cone $\mcC_{E_\no}$ generated by $E_\no$.

\begin{dfn}
	Let $\Phi_{\no}$ be the $1\times (d+d'+1)$ matrix
	\begin{equation} \label{eq:Phino}
		\Phi_{\no}:=[{0^{(d-2)}},1,2,{(-1)^{(d'+1)}}].
	\end{equation}
	Let $E_{\no}\subseteq \N_0^{d+d'+1}$ be the monoid $E_{\Phi_{\no}}$ as in \eqref{eq:ConeDiophantineEquation}.
\end{dfn}

Recall from~\Cref{rem:associatedCone} that the convex cone
$\mcC_{E_\no}$ generated by $E_\no$ is a pointed convex polyhedral
cone.  The completely fundamental elements of $E_\no$ each lie on an
extreme ray of $\mcC_{E_\no}$. Therefore by~\Cref{theorem_poset_bij},
the completely fundamental elements of $E_{\no}$ correspond to the
minimal (non-empty) supports in $L(E_{\no})$.  By~\eqref{eq:Phino}, we
see that the $2d'+d$ minimal (non-empty) supports in $L(E_{\no})$ are
\begin{numcases}{}
	\{i\} & \textup{for }i\in [d-2]; \label{eq:minimalsupport1} \\
	\{d-1,i\} & \textup{for }i\in d+[d'+1]; \label{eq:minimalsupport2} \\
	\{d,i\} & \textup{for }i\in d+[d'+1]. \label{eq:minimalsupport3}
\end{numcases}
For $i\in [d+d'+1]$, let $\delta_i\in \N_0^{d+d'+1}$ be the $i$th unit
basis vector.  The $2d'+d$ completely fundamental elements of $E_\no$
are the following:
\begin{numcases}{}
	\delta_i & \textup{for }i\in [d-2]; \\
	\delta_{d-1}+\delta_i & \textup{for }i\in d+[d'+1]; \\
	\delta_d + 2\delta_i & \textup{for }i\in d+[d'+1].
\end{numcases}
The completely fundamental element $\delta_d+2\delta_{d+d'+1}$ will receive special attention. 

By~\Cref{theorem_poset_bij}, the $2$-faces of $\mcC_{E_\no}$ can be
found by looking at the elements of $L(E_{\no})$ that
\textcolor{black}{properly} contain at least one of the sets in
\eqref{eq:minimalsupport1}-\eqref{eq:minimalsupport3}, yet do not
strictly contain any non-empty elements of $L(E_{\no})$ that are not
listed in \eqref{eq:minimalsupport1}-\eqref{eq:minimalsupport3}.  We
are especially interested in the $2$-faces of $\mcC_{E_\no}$ that
contain $\delta_d+2\delta_{d+d'+1}$.  The elements of $L(E_{\no})$
that contain $\{d,d+d'+1\}$ and do not strictly contain any non-empty
elements of $L(E_{\no})$ not listed in
\eqref{eq:minimalsupport1}-\eqref{eq:minimalsupport3} are the
following:
\begin{numcases}{}
	\{i,d,d+d'+1\} & \textup{for }i\in [d-2];\\
	\{d-1,d,d+d'+1\}; & \\
	\{ d,i,d+d'+1\} & \textup{for }i\in d+[d'].
\end{numcases}
Note that the sets $\{ d-1,d,i,d+d'+1\}$ for $i\in d+[d']$ are
elements of $L(E_{\no})$, but they strictly contain the set $\{
d-1,d,d+d'+1\}$, which is not listed in
\eqref{eq:minimalsupport1}-\eqref{eq:minimalsupport3}, and therefore
they do not correspond to a $2$-face.  The $2$-faces of $\mcC_{E_\no}$
that contain $\delta_d+2\delta_{d+d'+1}$ are thus generated by
\begin{numcases}{}
	\{\delta_d+2\delta_{d+d'+1},\delta_i\} & \textup{for }i\in [d-2]; \\
	\{\delta_d+2\delta_{d+d'+1},\delta_{d-1}+\delta_{d+d'+1}\}; & \\
	\{ \delta_d+2\delta_{d+d'+1},\delta_d + 2\delta_i\} & \textup{for }i\in d+[d'].
\end{numcases}

\begin{dfn} \label{dfn:mcC0}
	Let $\mcC_0$ be the subcone (not a face) of $\mcC_{E_\no}$ generated by the set
	\begin{equation*} \label{eq:generatorsmcC0}
		\{ \delta_i \mid i\in [d-2]\} \cup 
		\{ \delta_{d-1}+\delta_{d+d'+1} \} \cup
		\{ \delta_d + 2\delta_{i} \mid i\in d+[d'+1]\},
	\end{equation*}
	and let $E_0=\mcC_0\cap E_\no$.
\end{dfn}

\begin{rem} \label{rem:mcC0simplicial}
	Being generated by $d+d'$ linearly independent elements of $E_\no$,
	$\mathcal{C}_0$ and $E_0$ are simplicial and have dimension $d+d'$.
\end{rem}

\begin{pro}[{adaptation of \cite[Lem.~4.5.1]{Stanley/12}}]
	\label{pro:special_tria}
	The pointed polyhedral cone $\mcC_{E_\no}$ has a triangulation 
	$\Gamma=( \mcK_u )_{ u\in U }$ whose one-dimensional faces are 
	the extreme rays of $\mcC_{E_\no}$ and 
	there is a $u\in U$ with $\mcK_u=\mcC_0$.  
\end{pro}
\begin{proof}
	Use the algorithm in the proof of \cite[Lem.~4.5.1]{Stanley/12}, 
	while ordering the extreme rays generated by each of the elements of
	\eqref{eq:generatorsmcC0} that are not $\delta_d+2\delta_{d+d'+1}$ first.
\end{proof}
\begin{rem} \label{rem:tria}
	By definition, $\mcC_0$ contains all $2$-faces of 
	$\mcC_{E_\no}$ that contain $\delta_d+2\delta_{d+d'+1}$.
	Therefore in the triangulation $\Gamma$ from~\Cref{pro:special_tria},
	the faces of $\mcC_0$ are the only elements 
	that contain $\delta_d+2\delta_{d+d'+1}$.
\end{rem}

\subsection{The subsets \texorpdfstring{$H_{I,J}\subseteq \N_0^{d+d'}$}{HI,JN0d+d'}}
\label{subsec:HIJ}
We define subsets $H_{I,J}\subseteq \N_0^{d+d'}$ for every $I\subseteq
[d-1]$ and $J\subseteq [d'-1]$.  These subsets are used
in~\Cref{subsec:novlpalt} to give a formula for
$\zeta^{\no}_{\mff_{2,d}}(\varq,t)$.

\begin{dfn} \label{dfn:HIJ}
	For every $I\subseteq [d-1]$ and $J\subseteq [d'-1]$, let $H_{I,J}$ be the set of tuples
	$(\rtup, \stup) \in \mathbb{N}_0^{d+d'}$
	that satisfy the following equations and inequalities:
	\begin{numcases}{}
		r_i>0 & \textup{ for } i\in I, \label{eq:HIJ1} \\
		r_i= 0 & \textup{ for } i\in [d-1]\backslash I, \label{eq:HIJ2} \\
		s_j>0 & \textup{ for } j\in J, \label{eq:HIJ3} \\
		s_j= 0 & \textup{ for } j\in [d'-1]\backslash J, \label{eq:HIJ4} \\
		r_{d-1}+2r_{d} - \sum_{j=1}^{d'} s_j 
		\geq 0. \label{eq:HIJ5} 
	\end{numcases}
\end{dfn}
\begin{exm}
	Let $d=2$, $I=\{1\}$, and $J=\emptyset$. Then $H_{I,J}$ is the set of tuples $(r_1,r_2,s_1)\in \N_0^3$ such that $r_1>0$ and $r_1+2r_2-s_1\geq0$.
\end{exm}

Just as~\Cref{pro:prIequalsG} described $G_{I,\sigma}$ as a set of the form $I_{E_{\sigma},A,C}$, 
we now describe the sets $H_{I,J}$ as sets of the form $I_{E_{\no},A,C}$.
\begin{dfn} \label{dfn:AIJCIJ}
	For every $I\subseteq [d-1]$ and $J\subseteq [d'-1]$, let
        $A_{I,J}$ and $C_{I,J}$ be the following subsets of
        $[d+d'+1]$:
	\begin{align*}
		A_{I,J}&:=I\cup (d+J),\\
		C_{I,J}&:=I\cup (d+J)\cup \{d,d+d',d+d'+1\}.
	\end{align*}
\end{dfn}
\begin{exm}
	Let $d=2$, $I=\{1\}$, and $J=\emptyset$. Then $A_{\{1\},\emptyset}=\{1\}\cup \emptyset$ and $C_{\{1\},\emptyset}=\{1\} \cup \emptyset \cup \{2,3,4\}$.
\end{exm}

\begin{pro} \label{pro:HIJasprI}
Let $\textup{pr}: \N^{d+d'+1} \to \N^{d+d'} $ be the projection map
which ignores the last coordinate.  For every $I\subseteq [d-1]$ and
$J\subseteq [d'-1]$, restricting this projection map to the subset
$I_{E_{\no}, A_{I,J},C_{I,J}} \subseteq E_{\no}\subseteq \N^{d+d'+1}$
results in a bijection between $I_{E_{\no}, A_{I,J},C_{I,J}}$
and~$H_{I,J}$.
\end{pro}
\begin{proof}
	Suppose that $(\rtup,\stup,\gamma_1) \in I_{E_{\no}, A_{I,J},C_{I,J}}$.
	Then \eqref{eq:HIJ1} is satisfied because $I\subseteq A_{I,J}$ and \eqref{eq:HIJ2} is satisfied because $([d-1]\backslash I)\cap C_{I,J}=\emptyset$. 
	Also \eqref{eq:HIJ3} is satisfied because $(d+J)\subseteq A_{I,J}$ and \eqref{eq:HIJ4} is satisfied because $(d+([d'-1]\backslash J))\cap C_{I,J}=\emptyset$. 
	Lastly \eqref{eq:HIJ5} follows from the definition of $\Phi_{\no}$. 
	Thus $(\rtup,\stup) \in H_{I,J}$.
	
	Suppose that $(\rtup,\stup) \in H_{I,J}$. 
	Let $\gamma_1=r_{d-1}+2r_{d} - \sum_{j=1}^{d'} s_j$. 
	Then $(\rtup,\stup,\gamma_1) \in I_{E_{\no}, A_{I,J},C_{I,J}}$. Moreover, this is the unique element $\gamma_1\in \N_0$ such that $(\rtup,\stup,\gamma_1) \in I_{E_{\no}, A_{I,J},C_{I,J}}\subseteq E_{\no}$, 
	because if $(\rtup,\stup,\gamma_1)\in E_{\no}$, then $r_{d-1}+2r_{d} - \sum_{j=1}^{d'} s_j-\gamma_1=\nobreak0$.
\end{proof}

\begin{rem} \label{rem:decomposition_mcK}
	Note that
	\begin{equation*}\label{eq:decomposition_mcK}
		E_\no=\bigcup_{I\subseteq [d-1],J\subseteq [d'-1]} 
		I_{E_{\no}, A_{I,J},C_{I,J}},
	\end{equation*}  
	and this union is disjoint. 
\end{rem}

We record the following reciprocity result for the generating
functions $H_{I,J}(\Xtup,\Ytup)$.
\begin{pro} \label{pro:reciprocityHIJ}
	Let $K\subseteq [d-1]$ and $L\subseteq [d'-1]$, 	
	then
	\begin{equation*}
		\sum_{\substack{ I\subseteq [d-1], I\supseteq K, \\
				J\subseteq [d'-1], J\supseteq L}} H_{I,J}(\Xtup^{-1},\Ytup^{-1})
		= (-1)^{d+d'} 
		\sum_{\substack{ I\subseteq [d-1], I\supseteq [d-1]\backslash K, \\
				J\subseteq [d'-1], J\supseteq [d'-1]\backslash L} }
		X_d Y_{d'} H_{I,J}(\Xtup,\Ytup).
	\end{equation*}
\end{pro}
\begin{proof}
	By~\Cref{pro:HIJasprI} and \eqref{eq:disjointUnionHEAB},
	\begin{equation*}
		H_{I,J}(\Xtup^{-1},\Ytup^{-1})
		=I_{E_{\no}, A_{I,J},C_{I,J}} (\Xtup^{-1},\Ytup^{-1},1) =\sum_{ \substack{ B\in L(E_{\no}), \\ A_{I,J}\subseteq B \subseteq C_{I,J} }} \overline{F}_{E_{\no},B} (\Xtup^{-1},\Ytup^{-1},1).
	\end{equation*}
	Therefore summing over $I\supseteq K$ and $J\supseteq L$ results in
	\begin{equation*}
		\sum_{I\supseteq K, J\supseteq L} H_{I,J}(\Xtup^{-1},\Ytup^{-1})
		= \sum_{B\in L(E_{\no}), A_{K,L} \subseteq B} 
		\overline{F}_{E_{\no},B} (\Xtup^{-1},\Ytup^{-1},1).
	\end{equation*}
	Note that any subset of $[d+d'+1]$ that contains $d$ and $d+d'$ 
	is an element of the lattice of supports~$L(E_{\no})$.
	 In particular, $[d+d'+1] \backslash A_{K,L}$ is an element of $L(E_{\no})$.	
	Therefore, we may use~\Cref{coro_complement} to obtain
	\begin{equation*}
		\sum_{I\supseteq K, J\supseteq L} H_{I,J}(\Xtup^{-1},\Ytup^{-1})
		= (-1)^{\dim(E_{\no})} 
		\sum_{ \substack{ D\in L(E_{\no}), \\ D\supseteq [d+d'+1] \backslash A_{K,L} } }
		\overline{F}_{E_{\no},D} (\Xtup,\Ytup,1).
	\end{equation*}
	The dimension of $E_{\no}$ is $d+d'$ 
	because it is a subset of $\N^{d+d'+1}$ 
	subjected to one linear equation. 
	If $D\in L(E_{\no})$ with $D\supseteq [d+d'+1] \backslash A_{K,L}$, 
	then there are unique subsets $I\subseteq [d-1]$ and $J\subseteq [d-1]$ containing $[d-1]\backslash K$ and $[d'-1] \backslash L$ respectively such that $D=C_{I,J}$ and vice versa.
	Thus
	\begin{equation*}
		\sum_{I\supseteq K, J\supseteq L} H_{I,J}(\Xtup^{-1},\Ytup^{-1})
		= (-1)^{d+d'} 
		\sum_{ \substack{ I\supseteq [d-1]\backslash K, \\ J\supseteq [d'-1]\backslash L} }
		\overline{F}_{E_{\no},C_{I,J}} (\Xtup,\Ytup,1).
	\end{equation*}
	It is easy to verify that for $I\subseteq [d-1]$ and $J\subseteq [d'-1]$, the map
	\begin{equation*}
		\overline{F}_{E_{\no},C_{I,J}} \longrightarrow H_{I,J}  : 
		(a_1,\ldots,a_d,b_1,\ldots,b_{d'},z) \longmapsto  (a_1,\ldots,a_d-1,b_1,\ldots,b_{d'}-1)
	\end{equation*}
	is a (well-defined) bijection. Thus
	\begin{equation*}
		\overline{F}_{E_{\no},C_{I,J}} (\Xtup,\Ytup,1)
		=X_dY_{d'} H_{I,J}(\Xtup,\Ytup). \qedhere
	\end{equation*}
\end{proof}

\subsection{Dyck words and the relation between \texorpdfstring{$H_{I,J}$}{HI,J} and \texorpdfstring{$G_{I,\sigma}$}{GIsigma}}
\label{subsec:dyck}
We associate a Dyck word $w_{\sigma}$ to permutations $\sigma\in
\Spec_{2d'}$.  Then we describe how the sets $G_{I,\sigma}$ and
$H_{I,J}$ are related.

A \emph{Dyck word of length $2d'$} is a word $w$ in the letters $\bfz$
and $\bfo$ such that $\bfz$ and $\bfo$ each occur $d'$ times in $w$
and no initial segment of $w$ contains more ones than zeroes. For
example, $\bfz\bfz\bfo\bfz\bfo\bfo$ is a Dyck word of length $6$,
whereas $\bfz\bfo\bfo\bfz\bfz\bfo$ is not as the initial segment
$\bfz\bfo\bfo$ contains more ones than zeroes.  We write $\mcD_{2d'}$
for the set of Dyck words of length~$2d'$.  The Dyck word
$\bfz^{d'}\bfo^{d'}\in\mcD_{2d'}$ is called the \emph{trivial Dyck
word of length~$2d'$}.

\begin{dfn} \label{dfn:DyckWordAssociated}
	Let $\sigma\in \Spec_{2d'}$.
	The \emph{Dyck word $w_\sigma$ associated with $\sigma$} 
	is the Dyck word of length $2d'$ where for each $i\in [2d']$,
	the $i$-th letter of $w_\sigma$ is $\bfz$ if 
	$\sigma(i)>d'$ and $\bfo$ if $\sigma(i)\leq d'$.
\end{dfn}

\begin{exm}
	Let $d=3$ and $\sigma=451623\in \Spec_6$. Then
	$w_\sigma=\bfz\bfz\bfo\bfz\bfo\bfo=\bfz^2\bfo\bfz\bfo^2$.
\end{exm}

\begin{rem}
	Note that $w_\sigma$ is indeed a Dyck word because $\sigma\in \Spec_{2d'}$. 
	It would not necessarily be a Dyck word if $\sigma$ was a general permutation in $\Perm_{2d'}$.
\end{rem}

The following result links the sets $G_{I,\sigma}$ and $H_{I,J}$.
Recall \Cref{dfn:Jsigma} of $J_\sigma$.
\begin{pro}\label{pro:decomp}
	Let $I\subseteq [d-1]$, $J\subseteq [d'-1]$, and
	\begin{equation*}
		\Spec_{I,J}:= \{ \sigma\in \Spec_{2d'} \mid 
		(I,\sigma)\in\mcW_d, J_\sigma = J, w_\sigma= \bfz^{d'}\bfo^{d'}  \}.
	\end{equation*}
	Then $H_{I,J}=\bigcup_{\sigma\in \Spec_{I,J}} G_{I,\sigma}$
        and this union is disjoint.
\end{pro}
\begin{proof}
	Suppose that $\sigma\in \Spec_{I,J}$ and $(\rtup,\stup)\in G_{I,\sigma}$.
	Then \eqref{eq:HIJ1} and \eqref{eq:HIJ2} hold because of \eqref{eq:GISigma1} and \eqref{eq:GISigma2}. 
	Using~\Cref{pro:Jsigmasj}, we find that \eqref{eq:HIJ3} and \eqref{eq:HIJ4} hold.
	Lastly, \eqref{eq:HIJ5} holds because 
	\eqref{eq:HIJ5} is what \eqref{eq:GISigma4} reduces to 
	when $w_\sigma= \bfz^{d'}\bfo^{d'}$ and $i=d'$.
	Thus $(\rtup,\stup)\in H_{I,J}$.
	
	Conversely, suppose that $(\rtup,\stup)\in H_{I,J}$.
	Let $\sigma\in \Perm_{2d'}$ be the unique permutation such that
	\begin{equation*} \label{eq:sigmafromtuple}
		\sum_{j=1}^d v_{\sigma(i),j} r_j 
		+ \sum_{j=1}^{d'} v_{\sigma(i),d+j} s_j
		\geq \sum_{j=1}^d v_{\sigma(i+1),j} r_j 
		+ \sum_{j=1}^{d'} v_{\sigma(i+1),d+j} s_j
	\end{equation*}
	for every $i\in [2d'-1]$ and $\sigma(i)>\sigma(i+1)$ if equality holds. 
	The inequality \eqref{eq:HIJ5} 
	implies that 
	$\{ \sigma(i) \mid i\in [d'] \} = d'+[d']$
	and $\{ \sigma(i) \mid i\in d'+[d'] \} = [d']$,
	therefore $w_{\sigma}=\bfz^{d'}\bfo^{d'}$.
	It also follows by this fact and the construction of $\sigma$ 
	that $\sigma\in \Spec_{2d'}$.
	Also $(I,\sigma)\in \mcW_{2d'}$ because
	if $\rtup$ is non-zero, then it is a non-zero solution to \eqref{eq:mcW2},
	and if $\rtup$ is zero, then $I=\emptyset$, $\sigma(i)=2d'+1-i$ for $i\in [d']$, 
	and $\delta_{d'}$ is a non-zero solution to \eqref{eq:mcW2}.
	If $j\in J_{\sigma}$, then summing the right-hand side minus the left-hand side of 
	\eqref{eq:sigmafromtuple} over all $i\in [ \sigma^{-1}(j), \sigma^{-1}(j+1)-1]$ results in
	$s_j> 0$. Therefore $J_{\sigma} \subseteq J$. 
	Similarly if $j\in [d'-1]\setminus J_{\sigma}$, then summing
	the right-hand side minus the left-hand side of 
	\eqref{eq:sigmafromtuple} over all $i\in [ \sigma^{-1}(j+1), \sigma^{-1}(j)-1]$ results in
	$-s_j \geq 0$. Therefore $[d'-1]\backslash J_{\sigma} \subseteq [d'-1]\backslash J$
	and we conclude that $J_{\sigma}=J$.
	Obviously \eqref{eq:GISigma1} and \eqref{eq:GISigma2} 
	hold because of \eqref{eq:HIJ1} and \eqref{eq:HIJ2}.
	Moreover, \eqref{eq:GISigma3} and \eqref{eq:GISigma4}
	hold by construction of $\sigma$.
	Thus $(\rtup,\stup)\in G_{I,\sigma}$.
	
	The disjointness follows from the definition of $G_{I,\sigma}$.
\end{proof}

\section{The subalgebra zeta function of \texorpdfstring{$\mff_{2,d}$}{f2,d}} 
\label{sec:DerivationFormula} 
The main result in this section is~\Cref{thm:main}, which gives an
explicit formula for~$\zeta_{\mff_{2,d}(\lri)}(s)$.  The remainder of the
section works towards proving this formula.

\subsection{Overview of the proof of \Cref{thm:main}}\label{subsec:informal.main}
\textcolor{black}{Our starting point for the finite-sum formula for
  $\zeta_{\mff_{2,d}(\lri)}(s)$ we give in \Cref{thm:main} is a
  formula, provided in \Cref{prop:GSS}, for this zeta function as an
  \emph{infinite} sum, indexed by a set $\mcA_d$ of pairs
  $(\lambda,\nu)$ of partitions of at most $d$ resp.\ $d'$ parts;
  see~\Cref{def:Ad}. The summands of this infinite sum feature
  polynomials of the form \eqref{equ:alpha.1}, enumerating submodules
  of certain finite $\lri$-modules. Recall that these factorize as
  products of $q_\lri^{-1}$-multinomials coefficients and powers
  of~$q_{\lri}$. The membership condition for $\mcA_d$ requires $\nu$
  be dominated by the partition obtained from the sums of pairs of
  elements of $\lambda$.}

\textcolor{black}{In \Cref{subsec:part.Ad} we partition the infinite
  set $\mcA_d$ into the fibres of a surjective map $\omega:\mcA_d \to
  \mcW_d$ onto the \emph{finite} set $\mcW_d$ introduced in
  \Cref{dfn:mcW}; see~\Cref{dfn:omega}. Two facts are crucial. First,
  the $q_{\lri}^{-1}$-multinomial coefficients of the polynomials
  featuring in \Cref{prop:GSS} are constant on the fibres of $\omega$.
  We collect them into products of Gaussian multinomial coefficients
  called $\GMC_{I,\sigma}$, for $(I,\sigma)\in\mcW_d$;
  see~\Cref{dfn:GMCISigma}. Second, the fibres of $\omega$ are, by
  design, in bijection with the sets $G_{I,\sigma}$ introduced in
  \Cref{dfn:GISigma}. Taken together, these two facts allow us to
  rewrite the infinite sum in \eqref{eq:GSS.formula} by the finite
  expression \eqref{equ:sum}, namely as a sum over $\mcW_d$, whose
  summands are products of $\GMC_{I,\sigma}$s and generating series of
  the cones $G_{I,\sigma}$, upon a substitution of variables specified
  in the numerical data maps $\chi_\sigma$; see \Cref{dfn:num.data}.}

\subsection{An infinite series expression for 
	\texorpdfstring{$\zeta_{\mff_{2,d}(\lri)}(s)$}{zetaf2,d(o)(s)}} \label{sec:ExplicitFormula}

In~\Cref{prop:GSS} we write~$\zeta_{\mff_{2,d}(\lri)}(s)$ as an
infinite sum over pairs of partitions. Recall that $\parti_n$ is the
set of integer partitions of at most $n$ (non-zero) parts, i.e. the
set of tuples $\lambda=(\lambda_i)_{i\in[n]}\in \N_0^n$ with
$\lambda_i\geq \lambda_{i+1}$ for $i\in [n-1]$.  We start by
associating a partition $\mu_{\lambda}$ of at most $d'$ parts with
every partition $\lambda$ of at most~$d$ parts.
\begin{dfn}\label{def:mu_i}
	Given $\lambda\in \parti_d$, let $\mu_{\lambda}\in \parti_{d'}$ be the 
	integer partition $(\mu_j)_{j\in [d']}$ such that the multisets
	$\{ \mu_j \mid j\in [d']\}$ and
	$\left\{ \lambda_{i}+\lambda_{i'} \mid i< i'\in [d] \right\}$ coincide.
\end{dfn}
\noindent
Informally speaking, the integers $\mu_j$ are the integers
$\lambda_i+\lambda_{i'}$ brought in order.
\begin{exm}
	For $\lambda=(3,2,2,0)\in\parti_4$ we have $\mu_{\lambda}=(5,5,4,3,2,2)\in\parti_{\binom{4}{2}}$.
\end{exm}

\begin{dfn}\label{def:Ad}
  Let $$\mcA_d := \left\{ (\lambda,\nu)\in \parti_d\times \parti_{d'}
  \mid \nu \leq \mu_{\lambda}\right\}.$$
\end{dfn}

\textcolor{black}{The following result draws heavily on
  \cite[Thm.~2]{GSS/88}.} 
Recall that $\lambda_1^{(n)}:=(\lambda_1)_{i\in[n]}\in \parti_n$ and $\lvert (\lambda_i)_{i\in[n]}\rvert:=\sum_{i=1}^n \lambda_i$.

\begin{pro} \label{prop:GSS}
	For all cDVR $\mathfrak{o}$,
	\begin{equation}\label{eq:GSS.formula}
		\zeta_{\mff_{2,d}(\lri)}(s)
                =\sum_{(\lambda,\nu)\in\mcA_d} \alpha\left(
                \lambda_1^{(d)}, \lambda; \lri \right) \alpha(
                \mu_\lambda, \nu; \lri ) \cardres^{-s \lvert \lambda
                  \rvert } \cardres^{(d-s) \lvert \nu \rvert},
	\end{equation}
	where $\alpha( \lambda, \mu; \lri )$ is discussed
        in~\Cref{subsubsec:SubgroupsFiniteAbelianPGroups}.
\end{pro}
\begin{proof}
	Recall that, as an $\lri$-module, $\mff_{2,d}(\lri)$ is
        generated by $\{ x_i \mid i\in[d] \} \cup \{ [x_i,x_j] \mid
        i,j\in [d]\}$. Let $\mcL_1$ be the rank-$d$ submodule
        generated by $\{ x_i \mid i\in[d] \}$ and $\mcL_2$ be the
        rank-$d'$ submodule generated by $\{ [x_i,x_j] \mid i,j\in
        [d]\}$. Hence $\mff_{2,d} = \mcL_1\oplus\mcL_2$
        \textcolor{black}{as $\lri$-modules}.
	
	Given a submodule $\Lambda\leq \mff_{2,d}(\lri)$, we associate
        two submodules $\Lambda_1$ and $\Lambda_2$ of $\mcL_1$ and
        $\mcL_2$ respectively. The second submodule, $\Lambda_2$, is
        $\Lambda\cap \mcL_2$, while the first submodule $\Lambda_1$ is
        the unique submodule $\Lambda_1\leq \mcL_1$ such that
        \textcolor{black}{$(\Lambda_1\oplus
          \mcL_2)/\mcL_2=(\Lambda+\mcL_2)/\mcL_2$}. This way
        $\Lambda$ is not necessarily equal to $\Lambda_1\oplus
        \Lambda_2$, but $\lvert \mff_{2,d}(\lri):\Lambda \rvert =
        \lvert \mcL_1:\Lambda_1 \rvert \cdot \lvert \mcL_2:\Lambda_2
        \rvert$ always holds.  The condition that $\Lambda$ is a
        subalgebra of $\mff_{2,d}(\lri)$ is equivalent to
        $[\Lambda_1,\Lambda_1]$ being a submodule of $\Lambda_2$.  In
        complete analogy with \cite[Lemma 6.1]{GSS/88} one shows that
	\begin{equation*}
		\zeta_{\mff_{2,d}(\lri)}(s)=	\sum_{\substack{\Lambda_1 \leq \mcL_1 }} \sum_{\substack{ \Lambda_2 \leq \mcL_2 \\ [\Lambda_1,\Lambda_1]\leq \Lambda_2 } } \lvert \mcL_1: \Lambda_1 \rvert^{-s} \lvert \mcL_2: \Lambda_2 \rvert^{d-s}.
	\end{equation*}
	Recall~\Cref{dfn:ElementaryDivisorType} of the elementary divisor type
	$\varepsilon(\Lambda)$ of a submodule $\Lambda$. 
	We write the zeta function as a sum over the elementary divisor type of $\Lambda_1$:
	\begin{equation} \label{eq:proofexplicitformula2}
		\zeta_{\mff_{2,d}(\lri)}(s)=\sum_{\lambda\in \parti_d} \sum_{\substack{\Lambda_1 \leq \mcL_1 \\ \varepsilon(\Lambda_1)=\lambda }} \lvert \mcL_1: \Lambda_1 \rvert^{-s} \sum_{\substack{ \Lambda_2 \leq \mcL_2 \\ [\Lambda_1,\Lambda_1]\leq \Lambda_2 } }  \lvert \mcL_2: \Lambda_2 \rvert^{d-s}.
	\end{equation}
	Let $\lambda$ be the elementary divisor type
        $\varepsilon(\Lambda_1)$ of $\Lambda_1\leq \mcL_1$.  Then
        $\left\{ \pi^{\lambda_{i}+\lambda_j} \mid i< j\in [d]
        \right\}$ yields the elementary divisor type
        of~$[\Lambda_1,\Lambda_1]\leq \mcL_2$. \textcolor{black}{More
          precisely, by~\Cref{def:mu_i},
          $\mu_{\varepsilon(\Lambda_1)}$ is the elementary divisor
          type $\varepsilon([\Lambda_1,\Lambda_1])$ of
          $[\Lambda_1,\Lambda_1]\leq \mcL_2$.}  When counting
          submodules $\Lambda_2\leq \mcL_2$ that contain a given
          submodule $[\Lambda_1,\Lambda_1]$, the only thing that is
          important is the elementary divisor type of the given
          submodule $[\Lambda_1,\Lambda_1]$, see
          ~\Cref{thm:CountingSubmodulesContaining}. Since the
          elementary divisor type of $[\Lambda_1,\Lambda_1]$ is
          completely determined by $\varepsilon(\Lambda_1)$, we can
          conclude that counting submodules $\Lambda_2$ that contain a
          given submodule $[\Lambda_1,\Lambda_1]$ is also completely
          determined by~$\varepsilon(\Lambda_1)=\lambda$.  Therefore,
          the last two summations in \eqref{eq:proofexplicitformula2}
          are independent counting problems, connected by the
          elementary divisor type of $\Lambda_1$:
	\begin{equation}
	 \zeta_{\mff_{2,d}(\lri)}(s)=\sum_{\lambda\in \parti_d} \left(
	 \sum_{\substack{\Lambda_1 \leq \mcL_1 \\ \varepsilon(\Lambda_1)=\lambda }}
	 \lvert \mcL_1: \Lambda_1 \rvert^{-s} \right) \left(\sum_{\substack{
	     \Lambda_2 \leq \mcL_2 \\ M_\lambda \leq \Lambda_2 } } \lvert
	 \mcL_2: \Lambda_2 \rvert^{d-s}\right),
	\end{equation}
	where $M_\lambda$ is any submodule of $\mcL_2$ that has elementary
	divisor type $\mu_\lambda$.
	
	The first counting problem (in the first pair of brackets), is
        to count submodules $\Lambda_1\leq \mcL_1$ with fixed
        elementary divisor type. The solution to this first counting
        problem is discussed in
        \Cref{thm:CountingSubmodulesFixedElementaryDivisors}. The
        result in this case is
	\begin{equation*}
		\sum_{\substack{\Lambda_1 \leq \mcL_1 \\ \varepsilon(\Lambda_1)=\lambda }} \lvert \mcL_1: \Lambda_1 \rvert^{-s}
		=\alpha( \lambda_1^{(d)}, \lambda; \lri  ) \cardres^{-s \lvert \lambda \rvert}.
	\end{equation*}
     The second counting problem (in the second pair of brackets) is
     to count submodules $\Lambda_2\leq \mcL_2$ that contain a given
     submodule $M_\lambda$ of which we know the elementary divisor
     type, namely~$\mu_{\lambda}$. The solution to this second
     counting problem is discussed in
     ~\Cref{thm:CountingSubmodulesContaining} and the result in this
     case is
	\begin{equation*}
	 \sum_{\substack{ \Lambda_2 \leq \mcL_2 \\ M_\lambda \leq \Lambda_2 }
	 } \lvert \mcL_2: \Lambda_2 \rvert^{d-s}
	 =\sum_{\substack{\nu\in \parti_{d'} \\ \nu\leq \mu_{\lambda}}} \alpha(
	 \mu_\lambda, \nu; \lri ) \cardres^{(d-s) \lvert \nu \rvert }.  \qedhere
	\end{equation*}
\end{proof}

The formula for $\zeta_{\mff_{2,d}(\lri)}(s)$ in \Cref{prop:GSS}
is explicit, \textcolor{black}{contains however} infinite sums.  In the
following sections, the formula will be written as a finite sum, where
each summand will be a product of Gaussian multinomial coefficients
and a substitution of a series of the form
\eqref{eq:GeneratingSeriesCone}. This will make the formula amenable
to computer algebra systems capable of enumerating integral points in
polyhedra.

\subsection{The factor \texorpdfstring{$\alpha( \mu_\lambda, \nu; \lri )$}{alpha(mulambda,nu;o)}
	in terms of
        \texorpdfstring{$\sigma_{\lambda,\nu}$}{sigmalambda,nu}} We
associate a permutation $\sigma_{\lambda,\nu}\in \Perm_{2d'}$ with
each pair $(\lambda,\nu)\in \parti_d\times \parti_{d'}$.  Then we
rewrite the factor $\alpha( \mu_\lambda, \nu; \lri )$ that appeared in
\eqref{eq:GSS.formula} as a product of a power of $\cardres$ and a
product of Gaussian binomial coefficients that depends only on
$\sigma_{\lambda,\nu}$.  Recall \Cref{dfn:BijectionB} of the map~$b$.
\begin{dfn}\label{dfn:sigmalambdanu}
	Let $\lambda \in \parti_d$ and $\nu\in \parti_{d'}$.  Then
        $\sigma_{\lambda,\nu}$ is the permutation $\sigma\in
        \Perm_{2d'}$ defined inductively as follows.  Consider the
        multiset $\Sigma = \{\lambda_{i}+\lambda_j\}_{i<j\in [d]}\cup
        \{ \nu_i \}_{i\in [d']}$.  Let $\sigma(1)$ be the maximal
        $b$-value among the indices of the elements of $\Sigma$ which
        are maximal.  Now assume that $i>1$. To find $\sigma(i)$,
        consider the subset of $\Sigma$ comprising the elements whose
        indices have an image under $b$ that is not in $\{ \sigma(j)
        \mid j<i\}$.  Let $\sigma(i)$ be the maximal $b$-value among
        the indices of the elements of this subset which are maximal.
\end{dfn}
\begin{exm} \label{exm:sigmalambdanu}
	Let $d=3$, $\lambda=(5,4,1)$, and $\nu=(6,2,2)$. Then
	$\lambda_1+\lambda_2=9$, $\lambda_1+\lambda_3=6$, and
	$\lambda_2+\lambda_3=5$, thus $\Sigma = \{9,6,5,6,2,2\}$. Clearly,
	$\lambda_1+\lambda_2=9$ is the (unique) maximal element of $\Sigma$. 
	Therefore $\sigma(1)=b( (1,2) )=4$. Among the
	remaining elements, both $\lambda_1+\lambda_3=6$ and $\nu_1=6$ are
	maximal. Of the two, the index of $\lambda_1+\lambda_3$ has a greater
	image under~$b$, whence $\sigma(2)=b( (1,3) )=5$. Among the
	remaining elements, $\nu_1=6$ is maximal, whence
	$\sigma(3)=b(1)=1$. Continuing in this way, we find that
	$\sigma_{\lambda,\nu} = 451632$.
\end{exm}
\noindent
Recall \Cref{dfn:mcs} of the set $\Spec_{2d'}$.
\begin{rem} \label{rem:equivalenceS2d'}
	Let $\lambda \in \parti_d$ and $\nu\in \parti_{d'}$.  Then
        $\nu\leq \mu_\lambda$ if and only if $\sigma_{\lambda,\nu}\in
        \Spec_{2d'}$.
\end{rem}

Next, we define integers $L_j(\sigma)$ and $M_j(\sigma)$ for all
$\sigma\in \Spec_{2d'}$ and $j\in \{0\}\cup [2d']$.  These are related
to the integers $M_j$ and $L_j$ from~\Cref{dfn:MjLjMjTj} as discussed
in the proof of~\Cref{lem:BinomialsSigma}.
\begin{dfn}\label{dfn:LjMj}
	For $\sigma\in \Spec_{2d'}$ and $j\in \{0\}\cup [2d']$, define
	$L_0(\sigma):=0$, $M_0(\sigma):=0$,
    \begin{align*}
      L_j(\sigma) &:= \#\{ \sigma(i) \mid i\in [j] \} \cap
      (d' + [d']) &&\text{ for all }j\in [2d'],\\ 
      M_j(\sigma) &:= \#\{ \sigma(i) \mid i\in [j] \} \cap 
      [d'] &&\text{ for all }j\in [2d'].
    \end{align*}         
\end{dfn}
\begin{exm} \label{exm:LjMjSigma}
  Let $d=3$, $\sigma = 451623 \in \Spec_6$, and $j=3$. 
  Then $L_3(\sigma)=\lvert \{ 4,5,1 \} \cap \{ 4,5,6\} \rvert=2$ and 
  $M_3(\sigma)=\lvert \{ 4,5,1\} \cap \{ 1,2,3\} \rvert=1$.
\end{exm}

Recall that $\Asc(\sigma):=\{i\in [2d'-1] \mid
\sigma(i)<\sigma(i+1)\}$.  The following lemma writes the product of
Gaussian binomial coefficients in
\eqref{eq:AlternativeExpressionAlpha} in a way that depends only
on~$\sigma_{\lambda,\nu}$, using the integers $L_j(\sigma)$ and
$M_j(\sigma)$ from~\Cref{dfn:LjMj}.

\begin{lem} \label{lem:BinomialsSigma}
	Let $(\lambda,\nu)\in\mcA_d$.  Let
        $\sigma:=\sigma_{\lambda,\nu}\in \Spec_{2d'}$ and $M_j$ and
        $L_j$ be as in~\Cref{dfn:MjLjMjTj} with $\mu=\mu_{\lambda}$.
        Let $r:=\lvert \Asc(\sigma) \rvert+1$ and $\{ j_i \mid i \in
        [r-1] \}:=\Asc(\sigma)$ with $j_i<j_{i+1}$ for
        $i\in[r-2]$. Moreover, set $j_0:=0$ and $j_r:=2d'$.  Then
	\begin{equation}
	   \prod_{j\in[2d']}
          \binom{L_{j}-M_{j-1}}{M_j-M_{j-1}}_{\cardres^{-1}} 
          =
          \prod_{i\in [r]}
          \binom{L_{j_i}(\sigma)-M_{j_{i-1}}(\sigma)}{M_{j_i}(\sigma)-M_{j_{i-1}}(\sigma)}_{\cardres^{-1}}
          .
          \label{eq:PGMC}
	\end{equation}
\end{lem}
\begin{proof}
	For $i\in [r]$, let $k_i$ be the smallest element of 
	$(j_{i-1},j_i]$ such that $m_{k_i}=m_{k_i+1}=\dots=m_{j_i}$. 
	It follows directly that $L_{k_i}=L_{k_i+1}=\dots=L_{j_i}$ 
	and $M_{k_i}=M_{k_i+1}=\dots=M_{j_i}$.
	It also follows that $m_{k_i-1}>m_{k_i}$.
	We claim that $M_{j_{i-1}}=M_{j_{i-1}+1}=\dots=M_{k_i-1}$ as well. 
	Indeed, suppose that $M_{j}\neq M_{j+1}$ for some $j\in [j_{i-1},k_i-1)$. 
	Then $m_{j}> m_{j+1}$ and there is a $\nu_l$ with $\nu_l=m_{j+1}$,
	that is, there is a $j'\in [j+1,k_i)$ with $\sigma(j')=l\in [d']$.
	By construction, $m_{j+1}\geq m_{k_i-1}>m_{k_i}=m_{j_i}$, 
	thus we find that $m_j>m_{j+1}=\nu_l>m_{j_i}$. 
	Both if $\sigma(j_i)\in [d']$ or $\sigma(j_i)\in d'+[d']$, 
	this implies that there is an ascent of $\sigma$ 
	in the interval $[j+1,j_i)$, which is a contradiction. 
	Thus, we showed that the factor corresponding to $j\in [2d']$
	in \eqref{eq:PGMC} is one for all $j\in (k_i,j_i]$ and 
	$j\in (j_{i-1},k_i)$ with $i\in [r]$.
	The remaining factors are
	\begin{equation}
		\binom{L_{k_i}-M_{k_i-1}}{M_{k_i}-M_{k_i-1}}_{\cardres^{-1}} =
		\binom{L_{j_i}-M_{j_{i-1}}}{M_{j_i}-M_{j_{i-1}}}_{\cardres^{-1}}
	\end{equation}
	for all $i\in [r]$.  If $j_i\in \Asc(\sigma)$, then
        $m_{j_i}>m_{j_i +1}$ and therefore $L_{j_i}=L_{j_i}(\sigma)$
        and $M_{j_i}=M_{j_i}(\sigma)$.  The same conclusion holds for
        $j_0=0$ and $j_{r}=2d'$.  Thus \eqref{eq:PGMC} holds.
\end{proof}

The following proposition writes the factor $\alpha( \mu_\lambda, \nu;
\lri )$ in \eqref{eq:GSS.formula} as a product of a power of
$\cardres$ and a product of Gaussian binomial coefficients that
depends only on~$\sigma_{\lambda,\nu}$.

\begin{pro}\label{pro:D.GMC}
Let $(\lambda,\nu)\in\mcA_d$ Let $\sigma=\sigma_{\lambda,\nu}$ and
write $\{m_i \mid i\in [2d']\} := \mu_{\lambda} \cup \nu$ with
$m_i\geq m_{i+1}$ for all $i\in [2d']$ as in~\Cref{dfn:MjLjMjTj}.
Then
	\begin{equation} \label{eq:alphaagain}
		\alpha( \mu_\lambda, \nu; \lri )
		= \left( \prod_{i\in [r]}
			\binom{L_{j_i}(\sigma)-M_{j_{i-1}}(\sigma)}{M_{j_i}(\sigma)-M_{j_{i-1}}(\sigma)}_{\cardres^{-1}}
		\right) 
		\cardres^{ \sum_{j\in [2d']}
	        M_{j}(\sigma)(L_{j}(\sigma)-M_{j}(\sigma)) (m_{j}-m_{j+1})
		}.
              \end{equation}
\end{pro}

\begin{proof}
Recall the formula for $\alpha( \mu_\lambda,\nu; \lri)$ in
\eqref{eq:AlternativeExpressionAlpha}. ~\Cref{lem:BinomialsSigma}
rewrites the product of Gaussian binomial coefficients in
\eqref{eq:AlternativeExpressionAlpha} as the product of Gaussian
binomial coefficients in \eqref{eq:alphaagain}.  That the power of
$\cardres$ in \eqref{eq:AlternativeExpressionAlpha} equals the power
of $\cardres$ in \eqref{eq:alphaagain} follows from the fact that
$m_j=m_{j+1}$ if $L_j\neq L_j(\sigma)$ or $M_j\neq M_j(\sigma)$.
\end{proof}

\subsection{A finite partitioning of $\mcA_d$}\label{subsec:part.Ad}
The formula for $\zeta_{\mff_{2,d}(\lri)}(s)$ in~\Cref{prop:GSS} is an
infinite sum over~$\mcA_d$.  We partition this infinite series into a
finite number of summations indexed by the elements of the set
$\mcW_d$ from~\Cref{dfn:mcW}.  More precisely, the set $\mcA_d$ is
partitioned by the finitely many fibres of the following map~$\omega$.

\begin{dfn}\label{dfn:omega}
	Define the map
	\begin{equation*}
		\omega: \mcA_d \to \mcW_d : \quad (\lambda,\nu) \mapsto
                \omega(\lambda,\nu)=(I,\sigma),
	\end{equation*}
	where $I = \{i\in [d-1] \mid \lambda_i > \lambda_{i+1}\}$ and $\sigma=\sigma_{\lambda,\nu}$ as in~\Cref{dfn:sigmalambdanu}.
\end{dfn}
\begin{exm}
	Let $d=3$, $\lambda=(5,4,1)$, and $\nu=(6,2,2)$. The first component of $\omega(\lambda,\nu)$ is $I=\{1,2\}$, as $\lambda_1=5>\lambda_2=4$ and $\lambda_2=4>\lambda_3=1$. By~\Cref{exm:sigmalambdanu}, $\sigma_{\lambda,\nu}=451632$.
	Thus $\omega(\lambda,\nu)=(\{1,2\},451632)$
\end{exm}

\begin{rem}
The map $\omega$ is surjective by design of $\mcW_d$.
\end{rem}
	
\begin{rem}
One motivation for~\Cref{dfn:omega} of the map $\omega$ is that the
Gaussian multinomial coefficients in \eqref{eq:SumDISigma} only depend
on the image $\omega(\lambda,\nu)=(I,\sigma)$.  This allows for the
Gaussian multinomial coefficients to be pulled out of the summation
over $\lambda$ and $\nu$.
\end{rem}

\label{subsec:fibres}
Next, we show that the elements of 
the fibre $\omega^{-1}(I,\sigma)$ of $\omega$
are in bijection with the elements of the set
$G_{I,\sigma}$ from~\Cref{subsec:CISigma}.

\begin{dfn} \label{dfn:coor}
	For integer partitions  $\lambda \in \parti_d$ and $\nu\in \parti_{d'}$, 
	define
	\begin{align*}
		r_i &=
		\lambda_{i}-\lambda_{i+1} \textup{ for all } i\in [d-1], & r_d&=\lambda_d,\\
		s_j &= \nu_j-\nu_{j+1} \textup{ for all }j\in [d'-1], & s_{d'}&=\nu_{d'}.
	\end{align*}
\end{dfn}

\begin{exm}
	Let $d=3$, $\lambda=(5,4,1)$, and $\nu=(6,2,2)$. Then $r_1=5-4=1$, $r_2=4-1=3$, $r_3=1$, $s_1=6-2=4$, $s_2=2-2=0$, and $s_3=2$.
\end{exm}

Recall the definition of the corresponding tuples $v_{i}$ in~\Cref{dfn:CorrespondingVector}.
\begin{lem} \label{lem:mirjsj}
	Let $\lambda \in \parti_d$ and $\nu\in \parti_{d'}$ be integer partitions 
	with $\nu\leq \mu_\lambda$. Let $\sigma=\sigma_{\lambda,\nu}$ and
	write $\{m_i \mid i \in [2d'] \} =
	\mu_{\lambda} \cup \nu$ with $m_i\geq m_{i+1}$ for all $i\in[2d'-1]$ as in~\Cref{dfn:MjLjMjTj}.
	Then
	\begin{equation*}
		m_{i}=\sum_{j=1}^d v_{\sigma(i),j} r_j
		+\sum_{j=1}^{d'} v_{\sigma(i),d+j} s_j.
	\end{equation*}
\end{lem}
\begin{proof}
	Suppose that $i\in [2d']$ is such that $\sigma(i)\in [d']$. Then $v_{\sigma(i),j}=0$ for $j\in [d]$ and
	\begin{equation*}
		m_i=\nu_{\sigma(i)}=\sum_{j=\sigma(i)}^{d'} s_j
		=\sum_{j=1}^{d'} v_{\sigma(i),d+j} s_j.
	\end{equation*}
	Suppose that $i\in [2d']$ is such that $\sigma(i)\in d'+[d']$ and let $b^{-1}(\sigma(i))=(l,m)$ where $b$ is the map from~\Cref{dfn:BijectionB}. Then $v_{\sigma(i),d+j}=0$ for $j\in [d']$ and
	\begin{equation*}
		m_i=\lambda_{l}+\lambda_{m}
		=\sum_{j=l}^{d} r_j + \sum_{j=m}^{d} r_j
		=\sum_{j=1}^d v_{\sigma(i),j} r_j. \qedhere
	\end{equation*}
\end{proof}

\begin{pro}\label{pro:cone_equations}
	Let $(I,\sigma)\in \mcW_d$. The map
	\begin{equation} \label{eq:cone_equations}
		\omega^{-1}(I,\sigma) \to G_{I,\sigma}:
		(\lambda,\nu) \mapsto (\rtup, \stup),
              \end{equation}
	where $\rtup=(r_i)_{i\in [d']}$ and $\stup=(s_i)_{i\in [d']}$ are as in~\Cref{dfn:coor},
	is a bijection.
\end{pro}

\begin{proof}
	We first show that the map is well-defined.
	Let $\lambda \in \parti_d$ and $\nu\in \parti_{d'}$ be integer partitions 
	with $\nu\leq \mu_\lambda$ and $\omega(\lambda,\nu)=(I,\sigma)$.
	Since $\lambda$ and $\nu$ are integer partitions, we have that
	$r_i\geq 0$ and $s_j\geq 0$ for $i\in [d]$ and $j\in [d']$.
	The (in)equalities \eqref{eq:GISigma1} and \eqref{eq:GISigma2} 
	follow from the definition of $I$ in~\Cref{dfn:omega} as 
	$\{i\in [d-1] \mid \lambda_i > \lambda_{i+1}\}$.
	Let $\{m_i \mid i \in [2d']\} =
	\mu_{\lambda} \cup \nu$ with $m_i\geq m_{i+1}$ for all $i\in [2d'-1]$ as in~\Cref{dfn:MjLjMjTj}.
	Then \eqref{eq:GISigma4} holds by~\Cref{lem:mirjsj}. 
	Moreover, from the definition of $\sigma$ 
	in~\Cref{dfn:omega}, it follows that $m_i>m_{i+1}$ when 
	$i\in \Asc(\sigma)$, and therefore using~\Cref{lem:mirjsj}, 
	\eqref{eq:GISigma3} also holds. Thus 
	$(\rtup,\stup)$ is indeed an element of 
	$G_{I,\sigma}$ and \eqref{eq:cone_equations} is well defined.
	
	Suppose that $(\rtup,\stup)\in G_{I,\sigma}$.
	Let $\lambda_i=\sum_{j=i}^d r_j$ and $\nu_i=\sum_{j=i}^{d'} s_j$.
	Then $\lambda_i \geq \lambda_{i+1}$ and 
	$\nu_j\geq \nu_{j+1}$ because $r_i,s_j\geq 0$ 
	for all $i\in [d-1]$ and $j\in [d'-1]$.
	Let $\{m_i \mid i \in [2d']\} =
	\mu_{\lambda} \cup \nu$ with $m_i\geq m_{i+1}$ for all $i\in [2d'-1]$ as in~\Cref{dfn:MjLjMjTj}.
	Combining \eqref{eq:GISigma3} and \eqref{eq:GISigma4} with 
	\Cref{lem:mirjsj}, we find that $\sigma_{\lambda,\nu}=\sigma$ 
	and therefore $\nu\leq \mu_{\lambda}$ by~\Cref{rem:equivalenceS2d'}.
	Lastly, \eqref{eq:GISigma1} and \eqref{eq:GISigma2}
	imply that $\lambda_i> \lambda_{i+1}$ for all $i\in I$ and
	$\lambda_i=\lambda_{i+1}$ for all $i\in [d-1]\backslash I$.
	Thus $(\lambda,\nu)$ is an element of $\omega^{-1}(I,\sigma)$
	that gets mapped to $(\rtup,\stup)$, 
	proving that \eqref{eq:cone_equations} is surjective.
	The injectivity is trivial.
\end{proof}

\subsection{An explicit finite sum formula for \texorpdfstring{$\zeta_{\mff_{2,d}(\lri)}(s)$}{zetaf2d(o)(s)}} 
\label{subsec:finitesumformula}
\Cref{thm:main} is the main result of~\Cref{sec:DerivationFormula}.
It writes the subalgebra zeta function $\zeta_{\mff_{2,d}(\lri)}(s)$
as a finite sum indexed by the elements of $\mcW_d$, whose summands
are a product of Gaussian multinomial coefficients and a substitution
of a generating series of a set $G_{I,\sigma}$
from~\Cref{subsec:CISigma}.

\begin{dfn} \label{dfn:GMCISigma}
	Let $(I,\sigma)\in \mcW_d$ and $\{j_i\mid i\in \{0\} \cup [r]\}:=
	\Asc(\sigma)\cup \{0,2d'\}$ with $j_i<j_{i+1}$ for all $i\in\{0\} \cup [r]$. 
	The \emph{product of Gaussian multinomial coefficients 
		$\GMC_{I,\sigma}$ associated with $(I,\sigma)$} 
	is 
	\begin{equation*}
		\GMC_{I,\sigma}=\binom{d}{I}_{\varq^{-1}} \left(
		\prod_{i\in [r]}
		\binom{L_{j_i}(\sigma)-M_{j_{i-1}}(\sigma)}{M_{j_i}(\sigma)-M_{j_{i-1}}(\sigma)}_{\varq^{-1}}
		\right) \in \Z[\varq^{-1}].
	\end{equation*}
\end{dfn}
\begin{exm}
	Let $d=3$ and $(I,\sigma) =(\{1,2\},451623)$. Then
	\begin{align*}
          \GMC_{I,\sigma} &= \binom{3}{\{1,2\}}_{\varq^{-1}} \binom{1-0}{0-0}_{\varq^{-1}} \binom{2-0}{1-0}_{\varq^{-1}} \binom{3-1}{2-1}_{\varq^{-1}} \binom{3-2}{3-2}_{\varq^{-1}}\\
          &= q^{-5} + 4q^{-4} + 7q^{-3} + 7q^{-2} + 4q^{-1} + 1.
	\end{align*}
\end{exm}

Let $\varq$ and $t$ be indeterminates and $\Xtup=(X_i)_{i\in[d]}$ and
$\Ytup=(Y_j)_{j\in [d']}$ be tuples of
indeterminates. \textcolor{black}{We write $\mathcal{A}$ for the
  subalgebra of $\Q(\bfX,\bfY)$ generated by $\bfX^{\pm1},\bfY^{\pm1}$
  and $(1-\bfX^{\bfa}\bfY^{\bfb})^{-1}$ with
  $(0,0)\neq(\bfa,\bfb)\in\N_0^{d+d'}$.}

  \textcolor{black}{In fact, $\mathcal{A}$ comprises all rational
    functions in $\bfX$ and $\bfY$ we will encounter;
    cf.~\Cref{thm:genfun.cones}. We work with $\mathcal{A}$ to make
    substitutions like the following well-defined.}
\begin{dfn}\label{dfn:num.data} \label{dfn:chi_sigma}
  Let $\sigma\in \Spec_{2d'}$. Define $x_i(\sigma),y_j(\sigma)\in
  \Q(\varq,t)$ to be
	\begin{alignat}{2}
		x_i(\sigma)&:= 
		\varq^{\sum_{k\in[2d']} M_{k}(\sigma)(L_{k}(\sigma)-M_{k}(\sigma))  
			(v_{\sigma(k),i}-v_{\sigma(k+1),i} )  }
		\cdot  \varq^{i(d-i)} \cdot t^{i} 
		\quad && \text{for } i\in[d],\nonumber\\
		y_j(\sigma)&:=
		\varq^{\sum_{k\in[2d']} M_{k}(\sigma)(L_{k}(\sigma)-M_{k}(\sigma))  
			(v_{\sigma(k),d+j}-v_{\sigma(k+1),d+j} ) } 
		\cdot \varq^{jd}t^{j}
		\quad && \text{for }j \in [d'],\label{eq:num.data}
	\end{alignat}
	where $v_{\sigma(2d'+1),i}=0$ when $i\in [d+d']$.
	The \emph{numerical data map} $\chi_\sigma$ is
	\begin{equation*}
		\chi_\sigma: \textcolor{black}{\mathcal{A}} \rightarrow
                \Q(\varq,t): X_i \mapsto x_i(\sigma), Y_j \mapsto
                y_j(\sigma).
	\end{equation*}
\end{dfn}

\begin{exm}
	Let $\sigma=451623\in \Spec_6$ and $i=1$. 
	Then $v_{\sigma(1),1}=v_{\sigma(2),1}=1$ and $v_{\sigma(k),1}=0$ for all $k\in \{3,4,5,6\}$. 
	Thus                      
	\begin{equation*}
	 	x_1(\sigma)=\varq^{ 0(1-0)(1-1) + 0(2-0)(1-0)}\cdot \varq^{1(3-1)}\cdot t^{1} = \varq^{2}t.
	\end{equation*}                                                               
	Similarly, let $j=1$, then $v_{\sigma(3),4}=1$ and $v_{\sigma(k),4}=0$ for all $k\in \{1,2,4,5,6\}$. 
	Thus
	\begin{equation*}
		y_1(\sigma)=\varq^{ 0(2-0)(0-1)+1(2-1)(1-0)}\cdot \varq^{1\cdot 3} t^{1} = \varq^{4}t.
	\end{equation*} 
\end{exm}

\begin{rem}\label{rem:num.data}
Notice that, for each $i\in[d]$, the first exponent appearing in
\eqref{eq:num.data} can be rewritten as
	\begin{equation} \label{eq:num.data.exp}
		\sum_{k\in[2d']} \left( M_{k}(\sigma)(L_{k}(\sigma)-M_{k}(\sigma)) -  M_{k-1}(\sigma)(L_{k-1}(\sigma)-M_{k-1}
		(\sigma)) \right) v_{\sigma(k),i}.
	\end{equation}
	Given $i\in [d]$, $v_{\sigma(k),i}$ can only be non-zero if
        $\sigma(k)>d'$.  In this case, $L_{k-1}(\sigma)<L_{k}(\sigma)$ and
        $M_{k-1}(\sigma)=M_{k}(\sigma)$ and therefore
        $M_{k}(\sigma)(L_{k}(\sigma)-M_{k}(\sigma)) -
        M_{k-1}(\sigma)(L_{k-1}(\sigma)-M_{k-1} (\sigma))$ is non-negative. It
        follows that \eqref{eq:num.data.exp} is non-negative as well.
\end{rem}

Recall that $G_{I,\sigma}(\Xtup,\Ytup)$ is the series enumerating the
elements of $G_{I,\sigma}$.  Let $\zeta_{\mff_{2,d}}(\varq,t)$ be the
bivariate rational function in $\Q(\varq,t)$ such that
$\zeta_{\mff_{2,d}}(\cardres,\cardres^{-s})=\zeta_{\mff_{2,d}(\lri)}(s)$
for all cDVR~$\lri$.

\begin{thm} \label{thm:main}
	We have
	\begin{equation} \label{equ:sum} 
		\zeta_{\mff_{2,d}}(\varq,t) =
		\sum_{(I,\sigma)\in \mathcal{W}_d} \GMC_{I,\sigma}
		\chi_\sigma(G_{I,\sigma}(\Xtup,\Ytup)).
	\end{equation}	
\end{thm}
\begin{proof}	
	By~\Cref{prop:GSS} we may write
	\begin{equation} \label{eq:Partitioned}
		\zeta_{\mff_{2,d}(\lri)}(s)
		= \sum_{(I,\sigma)\in \mathcal{W}_d} 
		\sum_{(\lambda,\nu)\in \omega^{-1}(I,\sigma)}
		\alpha\left( \lambda_1^{(d)}, \lambda; \lri \right)
		\alpha( \mu_\lambda, \nu; \lri ) 
		\cardres^{-s \lvert \lambda \rvert } \cardres^{(d-s) \lvert \nu \rvert}.
	\end{equation}
	\Cref{cor:CorAlpha} with $n=d$ tells us that
	\begin{equation} \label{eq:alpha1}
		\alpha\left( \lambda_1^{(d)}, \lambda; \lri \right) =
		\binom{d}{I}_{\cardres^{-1}} \prod_{i=1}^{d}
		\cardres^{i(d-i)(\lambda_i-\lambda_{i+1})}.
	\end{equation}
	For $(\lambda,\nu)\in\mcA_d$, let
	\begin{equation}
		D_{\lambda,\nu} =  
		q_{\lri}^{ \sum_{j\in [2d']} M_{j}(\sigma) 
			( L_{j}(\sigma)-M_{j}(\sigma) ) 
			( m_{j}-m_{j+1} ) }  
		q_{\lri}^{\sum_{i=1}^d i(d-i)(\lambda_i-\lambda_{i+1})}
		(q_{\lri}^{-s})^{\lvert \lambda \rvert} 
		(q_{\lri}^{d-s})^{\lvert \nu \rvert},
		\label{eq:DISigmaAgain}
	\end{equation}
	where $\sigma=\sigma_{\lambda,\nu}$ and
	$\{m_i \mid i \in [2d']\} =
	\mu_{\lambda} \cup \nu$ with $m_i\geq m_{i+1}$ for all $i\in [2d'-1]$ as in~\Cref{dfn:MjLjMjTj}.
	Using \eqref{eq:alpha1} and~\Cref{pro:D.GMC} in 
	\eqref{eq:Partitioned} results in
	\begin{equation} \label{eq:SumDISigma}
		\zeta_{\mff_{2,d}}(q,t)
		= \sum_{(I,\sigma)\in \mathcal{W}_d} \GMC_{I,\sigma}
		\sum_{(\lambda,\nu)\in\omega^{-1}(I,\sigma)} 
		D_{\lambda,\nu}.
	\end{equation}
	It now suffices to show that for each $(I,\sigma)\in \mathcal{W}_d$,
	\begin{equation} \label{eq:mt1}
		\sum_{(\lambda,\nu)\in\omega^{-1}(I,\sigma)} 
		D_{\lambda,\nu}= \chi_\sigma(G_{I,\sigma}(\Xtup,\Ytup)).
	\end{equation}
	By~\Cref{pro:cone_equations}, the summands on the left-hand 
	side of \eqref{eq:mt1} are in bijection with the elements of 
	$G_{I,\sigma}$ and it suffices to show that
	\begin{equation}\label{form}
		D_{\lambda,\nu}=\left(\prod_{i=1}^d x_i(\sigma)^{r_i}\right)\left(\prod_{j=1}^{d'}
		y_j(\sigma)^{s_j}\right),
	\end{equation}
	where the $r_i$ and $s_j$ are as in~\Cref{dfn:coor}. 
	The $D_{\lambda,\nu}$ in \eqref{eq:DISigmaAgain} are written 
	as a product of four powers of $q_{\lri}$, which we analyze in turn.
	\begin{enumerate}
		\item The first power of $q_{\lri}$ has exponent $\sum_{j\in [2d']} M_{j}(\sigma) 
		( L_{j}(\sigma)-M_{j}(\sigma) ) 
		( m_{j}-m_{j+1} )$.
		It follows from~\Cref{lem:mirjsj} that
		\begin{equation}
			m_{k}-m_{k+1}
			=\sum_{i=1}^d (v_{\sigma(k),i}-v_{\sigma(k+1),i}) r_i +
			\sum_{j=1}^{d'}(v_{\sigma(k),d+j}-v_{\sigma(k+1),d+j}) s_{j},
		\end{equation}
		for all $k\in [2d']$, where $v_{\sigma(2d'+1),i}=0$ for each $i\in [d+d']$.
		We may thus rewrite this first power as
		\begin{multline}
			\left(\prod_{i=1}^d \left(q_{\lri}^{ \sum_{k\in [2d']}
				M_{k}(\sigma)(L_{k}(\sigma)-M_{k}(\sigma)) (v_{\sigma(k),i} - v_{\sigma(k+1),i})
			}\right)^{r_i}\right) \cdot \\
			\left(\prod_{j=1}^{d'} \left(q_{\lri}^{ \sum_{k\in [2d']} 
				M_{k}(\sigma)(L_{k}(\sigma)-M_{k}(\sigma))
				(v_{\sigma(k),d+j} - v_{\sigma(k+1),d+j}) }\right)^{s_j}\right).\label{num.data.1}
		\end{multline}
		\item The second power $q_{\lri}$ is easily rewritten as follows,
		\begin{equation}
			q_{\lri}^{\sum_{i=1}^d i(d-i)(\lambda_i-\lambda_{i+1})}=\left(\prod_{i=1}^d
			\left(q_{\lri}^{i(d-i)}\right)^{r_i}\right)\left(\prod_{j=1}^{d'}
			\left(1\right)^{s_j}\right).\label{num.data.2}
		\end{equation}
		\item The third and fourth powers of $q_{\lri}$ can be written as
		\begin{equation}
			q_{\lri}^{-s \lvert \lambda \rvert}q_{\lri}^{(d-s) \lvert \nu \rvert}=\left(\prod_{i=1}^d \left(t^{i}\right)^{r_i}\right)\left(\prod_{j=1}^{d'} \left(q_{\lri}^{jd}t^{j}\right)^{s_j}\right).\label{num.data.3}
		\end{equation}
	\end{enumerate}
	The product of \eqref{num.data.1} with the right-hand sides of
	\eqref{num.data.2} and \eqref{num.data.3} indeed results in 
	\eqref{form}.
\end{proof}

\section{Overlap type zeta functions, functional equation, and pole at zero} 
\label{sec:ResultsPAdicZetaFunction} In this section, we
present results on the $\mfp$-adic zeta
function~$\zeta_{\mff_{2,d}(\lri)}(s)$.
\textcolor{black}{In~\Cref{subsec:novlp} we decompose it as a sum of
  \emph{overlap type zeta functions}
  $\zeta^{w}_{\mff_{2,d}(\lri)}(s)$, indexed by the Dyck words $w\in
  \mcD_{2d'}$.}  Special attention goes to one overlap type zeta
function called the \emph{no-overlap zeta function}.  Informally and
purely heuristically speaking, it enumerates ``most'' of the
subalgebras of~$\mff_{2,d}(\lri)$; see
\Cref{rem:pirita}.~\Cref{thm:funeq.novlp} establishes a functional
equation for the no-overlap zeta function,
while~\Cref{thm:simple.pole.novlp} proves that it has a simple pole at
zero.  In~\Cref{subsec:si.po} we prove that the $\mfp$-adic zeta
function $\zeta_{\mff_{2,d}(\lri)}(s)$ has a simple pole at zero as
well, confirming a conjecture of Rossmann in the relevant cases;
see~\Cref{thm:simple.pole}.

\subsection{Overlap types and overlap zeta functions}\label{subsec:novlp}
We define the overlap type $w(\mfh)$ of a subalgebra
$\mfh\leq\mff_{2,d}(\mathfrak{o})$ of finite index and define an
overlap type zeta function $\zeta^{w}_{\mff_{2,d}(\lri)}(s)$ for each
overlap type~$w$, which enumerates the subalgebras of
$\mff_{2,d}(\mathfrak{o})$ with that overlap type.  Then we
slightly adapt~\Cref{thm:main} to a formula for each overlap
type zeta function $\zeta^{w}_{\mff_{2,d}(\lri)}(s)$.

Recall from~\Cref{subsec:dyck} that $\mcD_{2d'}$ denotes the set of Dyck
words of length $2d'$ and $w_{\sigma}$ is the Dyck word associated
with $\sigma$.  Recall, moreover, the permutation
$\sigma_{\lambda,\nu}$ associated with $(\lambda,\nu)$ from
\Cref{dfn:sigmalambdanu}.
\begin{dfn} \label{dfn:overlaptype} Let $\mathfrak{h}$ be a subalgebra of
  $\mff_{2,d}(\mathfrak{o})$ of finite index.  Let $\lambda$ be the elementary
  divisor type of
  $(\mathfrak{h}+[\mff_{2,d}(\mathfrak{o}),\mff_{2,d}(\mathfrak{o})])/[\mff_{2,d}(\mathfrak{o}),\mff_{2,d}(\mathfrak{o})]$
  in
  $\mff_{2,d}(\mathfrak{o})/[\mff_{2,d}(\mathfrak{o}),\mff_{2,d}(\mathfrak{o})]$
  as in~\Cref{dfn:ElementaryDivisorType}.  Let $\nu$ be the elementary divisor
  type of
  $\mathfrak{h}\cap [\mff_{2,d}(\mathfrak{o}),\mff_{2,d}(\mathfrak{o})]$ in
  $[\mff_{2,d}(\mathfrak{o}),\mff_{2,d}(\mathfrak{o})]$.  The \emph{overlap
    type} $w(\mfh)$ of $\mfh$ is the Dyck word
  $w_{\sigma_{\lambda,\nu}}\in \mcD_{2d'}$.
 
  We say that $\mfh$ has \emph{no overlap} if
  $w_{\sigma_{\lambda,\nu}}$ is the trivial Dyck
  word~$\bfz^{d'}\bfo^{d'}$.  Equivalently, $\mathfrak{h}$ has no
  overlap if and only if
  \begin{equation}\label{equ:nolap}
    \mu_1\geq \dots\geq \mu_{d'} \geq \nu_1\geq\dots\geq
    \nu_{d'}\end{equation} where $(\mu_j)_{j\in [d']}:=\mu_{\lambda}$
  as in~\Cref{def:mu_i}. \textcolor{black}{As the multiset of the
    numbers $\mu_j$ is the multiset of the
    numbers~$\lambda_i+\lambda_{i'}$, this is in turn} equivalent to
  the elementary divisors of
  $[\mff_{2,d}(\mathfrak{o}),\mff_{2,d}(\mathfrak{o})]/(\mathfrak{h}\cap
  [\mff_{2,d}(\mathfrak{o}),\mff_{2,d}(\mathfrak{o})])$ all being less
  than or equal to all the elementary divisors of
  $[\mff_{2,d}(\mathfrak{o}),\mff_{2,d}(\mathfrak{o})]/[\mathfrak{h},\mathfrak{h}]$.
  Otherwise, we say that $\mfh$ has \emph{overlap}.
\end{dfn}

\begin{dfn}\label{dfn:nooverlap}
	Let $w \in \mcD_{2d'}$. The \emph{overlap type $w$ zeta
        function} is defined as
	\begin{equation*}
		\zeta^{w}_{\mff_{2,d}(\lri)}(s) 
		:=\sum_{\mathfrak{h}\leq \mff_{2,d}(\mathfrak{o}), w(\mfh)=w} 
		\lvert \mff_{2,d}(\mathfrak{o}) : \mathfrak{h} \rvert^{-s}.
	\end{equation*}
	In particular, the \emph{no-overlap zeta function} is defined
        as
	\begin{equation*}
          \zeta^{\no}_{\mff_{2,d}(\lri)}(s)
          :=\zeta^{\bfz^{d'}\bfo^{d'}}_{\mff_{2,d}(\lri)}(s)
          =\sum_{\mathfrak{h}\leq \mff_{2,d}(\mathfrak{o}),
            w(\mfh)=\bfz^{d'}\bfo^{d'} } \lvert
          \mff_{2,d}(\mathfrak{o}) : \mathfrak{h} \rvert^{-s}.
	\end{equation*}
	Obviously $\zeta_{\mff_{2,d}(\lri)}(s) =
        \sum_{w\in\mcD_{2d'}}\zeta^w_{\mff_{2,d}(\lri)}(s)$. \end{dfn}

\begin{rem}\label{rem:pirita}
  \textcolor{black}{In~\cite{Paajanen/07}, Paajanen used an analysis of
    analogues of $\zeta^{\no}_{\mff_{2,d}}(q,t)$ for local zeta
    functions enumerating normal subgroups of finitely generated,
    torsion-free class-$2$-nilpotent groups to obtain a lower bound
    for the abscissa of convergence of these groups' normal zeta
    functions; see \Cref{subsec:rel.work} and
    \cite[Thm.~1.3]{Paajanen/07}. That this lower bound is in fact the
    exact value was later proven by the third author in
    \cite[Thm.~3]{Voll/05a}.}

  \textcolor{black}{This fact may support our heuristic that ``most''
    subalgebras of $\mff_{2,d}(\lri)$ should not have overlap: that
    the numbers $\mu_i$ and $\nu_j$ in \Cref{equ:nolap} intermingle
    should result in containment relations in subquotients of lattices
    which do not hold generically.}

  \textcolor{black}{It would, in particular, be of considerable
    interest to establish whether the abscissa of convergence of the
    Euler product over the local no-overlap zeta functions
    $\zeta^{\no}_{\mff_{2,d}}(p,p^{-s})$ coincides with the abscissa
    of convergence of the global zeta function
    $\zeta_{F_{2,d}}(s)$. The former is obviously a lower bound for
    the latter.}
\end{rem}

Let $\zeta_{\mff_{2,d}}^{w}(\varq,t)$ be the bivariate rational
function in $\Q(\varq,t)$ such that
$\zeta_{\mff_{2,d}}^{w}(\cardres,\cardres^{-s})$ equals
$\zeta^{w}_{\mff_{2,d}(\lri)}(s)$ for all cDVR
$\lri$. \textcolor{black}{(Its existence follows from the same
  reasoning as the one for $\zeta_{\mff_{2,d}}(\varq,t)$, discussed in
  \Cref{subsec:main.res}.)}  \Cref{thm:main} is
straightforwardly adapted to obtain a formula for the overlap type
zeta functions $\zeta^{w}_{\mff_{2,d}}(q,t)$ as follows.
\begin{thm} \label{thm:overlap}
	For all $w \in \mcD_{2d'}$
	\begin{equation} \label{eq:overlap}
		\zeta^w_{\mff_{2,d}}(\varq,t) = \sum_{\stackrel{(I,\sigma)\in
				\mcW_d}{w_{\sigma}= w }} \GMC_{I,\sigma} \chi_\sigma(G_{I,\sigma}(\Xtup,\Ytup)).
	\end{equation}
\end{thm}
\begin{proof}
	In the statement and proof of~\Cref{prop:GSS}, the summation
        can be restricted to the $(\lambda,\nu)\in\mcA_d$ with
        $w_{\sigma_{\lambda,\nu}}=w$:
	\begin{align*}\label{eq:GSS.formulaoverlap}
		\zeta^{w}_{\mff_{2,d}}(q,t)
                =&\sum_{(\lambda,\nu)\in\mcA_d,\,
                  w_{\sigma_{\lambda,\nu}=w}} \alpha\left(
                \lambda_1^{(d)}, \lambda; \lri \right) \alpha(
                \mu_\lambda, \nu; \lri ) \cardres^{-s \lvert \lambda
                  \rvert } \cardres^{(d-s) \lvert \nu \rvert}.
	\end{align*}
	Similarly, in the statement and proof of~\Cref{thm:main},
	the summation can be restricted to pairs $(I,\sigma)\in \mcW_d$
	with $w_{\sigma}=w$, resulting in \eqref{eq:overlap}.
\end{proof}

\begin{rem}
  \textcolor{black}{The formula~\eqref{eq:overlap} for the functions
    $\zeta^{w}_{\mff_{2,d}}(q,t)$ may be compared with the summands
  $D_{w,\rho}(q,t)$ in \cite[Def.~4.18]{CSV/24}. The latter are
  defined in terms of the generalized Igusa functions introduced in
  \cite[Def.~3.5]{CSV/24}. These in turn are, by defininition, finite
  sums of finite products of geometric progression, weighted by
  functions akin to the~$\GMC_{I,\sigma}$. To write down such an
  Igusa-type formula for the $\zeta^{w}_{\mff_{2,d}}(q,t)$ would seem
  to require a decomposition of the cones $G_{\sigma,I}$ into
  \emph{simple} cones.}
\end{rem}

\subsection{An alternative formula for the no-overlap zeta function} \label{subsec:novlpalt}
We simplify the formula for $\zeta^w_{\mff_{2,d}}(q,t)$
in~\Cref{thm:overlap} in the case when $w=\bfz^{d'}\bfo^{d'}$,
i.e.\ for the no-overlap zeta function
$\zeta^{\no}_{\mff_{2,d}(\lri)}(s)$.  We start by establishing how the
products of Gaussian binomial coefficients $\GMC_{I,\sigma}$ simplify
when $w_\sigma=\bfz^{d'}\bfo^{d'}$.
\begin{lem}\label{lem:GMCnooverlap}
	Suppose that $(I,\sigma)\in \mcW_d$ with
        $w_\sigma=\bfz^{d'}\bfo^{d'}$.  With
	\begin{equation*}
		J_{\sigma} :=\{j\in [d'-1] \mid \sigma^{-1}(j) < \sigma^{-1}(j+1)\},
	\end{equation*}
we have
	\begin{equation}\label{GMC.trivial}
		\GMC_{I,\sigma}=\binom{d}{I}_{\varq^{-1}}\binom{d'}{J_{\sigma}}_{\varq^{-1}}.
	\end{equation}
\end{lem}
\begin{proof}
	If $w_\sigma=\bfz^{d'}\bfo^{d'}$, then $M_0(\sigma)=\dots=M_{d'}(\sigma)=0$
	and therefore 
	\begin{equation*}
		\binom{L_{j_i}(\sigma)-M_{j_{i-1}}(\sigma)}{M_{j_i}(\sigma)-M_{j_{i-1}}(\sigma)}_{\varq^{-1}}
		=1
	\end{equation*}
	for all $j_i\in\Asc(\sigma)\cap [d']$. 
	Moreover, $L_{d'+j}(\sigma)=d'$ 
	and $M_{d'+j}(\sigma)=j$ for all $j\in[d']$.
	Let $j_1>\dots>j_r$ be such that $\Asc(\sigma)\cap (d'+[d'-1])=d'+\{j_i \mid i\in [r]\}$.
	Because $\sigma\in \Spec_{2d'}$, it follows that 
	$\sigma(d'+j_{i-1}+k)= j_i+1-k$ for all $i\in [r]$
	and $k\in [j_i-j_{i-1}-1]$ where $j_0=0$.
	Therefore $\Asc(\sigma)\cap (d'+[d'-1])=d'+J_{\sigma}$.
	Thus 
	\begin{equation*}
		\prod_{j_i\in \Asc(\sigma)\cap (d'+[d'])}
                \binom{L_{j_i}(\sigma)-M_{j_{i-1}}(\sigma)}{M_{j_i}(\sigma)-M_{j_{i-1}}(\sigma)}_{\varq^{-1}}
                = \prod_{j_i\in J_{\sigma}}
                \binom{d'-j_{i-1}}{j_i-j_{i-1}}_{\varq^{-1}}
                =\binom{d'}{J_{\sigma}}_{\varq^{-1}}.\hspace{-0.4cm}\qedhere
	\end{equation*}
\end{proof}

Next, we establish what the numerical data map $\chi_{\sigma}$
simplifies to when $w_\sigma=\bfz^{d'}\bfo^{d'}$.
\begin{dfn} \label{dfn:chino}
	The \emph{no-overlap numerical data map} $\chi_{\no}$ is
	\begin{equation*}
		\chi_{\no}: \textcolor{black}{\mathcal{A}} \rightarrow
                \Q(\varq,t): X_i \mapsto \varq^{i(d-i)}t^{i}, Y_j
                \mapsto \varq^{dj+j(d'-j)}t^{j}.
	\end{equation*}
\end{dfn}
\begin{exm}
 If $d=3$, then $\chi_{\no}(X_1)=\varq^{1(3-1)}t^1=\varq^{2}t$ and
 $\chi_{\no}(Y_1)=\varq^{3\cdot1+1(3-1)}t^1=\varq^{5}t$.
\end{exm}

\begin{rem} \label{rem:chi}
	Let $\sigma\in \Spec_{2d'}$ be such that $w_{\sigma}=\bfz^{d'}\bfo^{d'}$.
	Then the numerical data map $\chi_\sigma$ from~\Cref{dfn:num.data} simplifies to $\chi_{\no}$.
\end{rem}

The following theorem provides an alternative formula for the
no-overlap zeta function $\zeta^{\no}_{\mff_{2,d}(\lri)}(s)$.  It is a
lot less complicated than the general formula for the overlap type $w$
zeta function $\zeta^w_{\mff_{2,d}(\lri)}(s)$ in~\Cref{thm:overlap}
because it has fewer summands and the summands are simpler.  Recall
the sets $H_{I,J}$ from~\Cref{subsec:HIJ}.
\begin{thm} \label{thm:novlp_zeta}
We have	\begin{equation} 
		\zeta^{\no}_{\mff_{2,d}}(\varq,t)
		=\sum_{I\subseteq [d-1],J\subseteq [d'-1]} 
		\binom{d}{I}_{\varq^{-1}} \binom{d'}{J}_{\varq^{-1}} 
		\chi_{\no}({H_{I,J}(\Xtup,\Ytup)}).\label{eq:novlp_zeta}
	\end{equation}
\end{thm}
\begin{proof}
	Using \eqref{GMC.trivial} and~\Cref{rem:chi} in \eqref{eq:overlap}, results in
	\begin{align*}
		\zeta^{\no}_{\mff_{2,d}}(\varq,t)
		&=\sum_{\substack{(I,\sigma)\in\mcW_d,\\ w_\sigma= \bfz^{d'}\bfo^{d'}}}
		\binom{d}{I}_{\varq^{-1}} \binom{d'}{J_\sigma}_{\varq^{-1}}\chi_{\no}(G_{I,\sigma}(\Xtup,\Ytup)).
	\end{align*}
	Now using~\Cref{pro:decomp}, we find \eqref{eq:novlp_zeta}.
\end{proof}

\subsection{A functional equation for the no-overlap zeta function}\label{subsec:funeq.nol} 
\cite[Thm~A]{Voll/10} implies that $\zeta_{\mff_{2,d}}(\varq,t)$
satisfies the functional equation
\begin{equation*} \label{eq:funeq}
	\zeta_{\mff_{2,d}}(\varq^{-1},t^{-1})
	=(-1)^{D} \varq^{\binom{D}{2}}t^{D}\zeta_{\mff_{2,d}}(\varq,t),
\end{equation*}
where $D := d+d' = \binom{d+1}{2}$, the $\Z$-rank of $\mff_{2,d}$.  By
the following result, the no-overlap zeta function
$\zeta^{\no}_{\mff_{2,d}}(\varq,t)$ satisfies the same functional
equation.

\begin{thm}\label{thm:funeq.novlp}
  The no-overlap zeta function $\zeta^{\no}_{\mff_{2,d}}(q,t)$
  satisfies the functional equation
	\begin{equation*} \label{eq:funeq.novlp}
		\zeta^{\no}_{\mff_{2,d}}(\varq^{-1},t^{-1})
		=(-1)^{D} \varq^{\binom{D}{2}}t^{D}\zeta^{\no}_{\mff_{2,d}}(\varq,t).
              \end{equation*}
\end{thm}
\begin{proof}
	This proof follows the proof of \cite[Thm.~2.15]{Voll/11} which refers
	to \cite[Sec.~2 and~3]{Voll/10}. We start from the formula for 
	$\zeta^{\no}_{\mff_{2,d}}(\varq,t)$ stated in~\Cref{thm:novlp_zeta}.
	Using the identity \eqref{eq:GaussianMultinomialDescentSet} 
	for the Gaussian multinomial coefficients, we find
	\begin{equation*}
		\zeta^{\no}_{\mff_{2,d}}(\varq,t)
		=\sum_{\substack{ I\subseteq [d-1] \\ J\subseteq[d'-1]}}
		\lowBracketLeft \sum_{\substack{w\in \Perm_d,\\ \Des(w)\subseteq I}}
		\varq^{-\ell(w)}\lowBracketRight
		\lowBracketLeft\sum_{\substack{v\in \Perm_{d'},\\ \Des(v)\subseteq J}}
		\varq^{-\ell(v)}\lowBracketRight
		\chi_{\no}(H_{I,J}(\Xtup,\Ytup)).
	\end{equation*}
	Reordering the summations, this becomes
	\begin{equation*}
		\zeta^{\no}_{\mff_{2,d}}(\varq,t)
		=\sum_{w\in \Perm_d}
		\varq^{-\ell(w)}
		\sum_{v\in \Perm_{d'}}
		\varq^{-\ell(v)}
		\sum_{ \substack{ \Des(w)\subseteq I \subseteq [d-1]\\
				\Des(v)\supseteq J \subseteq [d'-1]}}\chi_{\no}(H_{I,J}(\Xtup,\Ytup)).
	\end{equation*}
	Inverting $\varq$ and $t$ on both sides and
        using~\Cref{pro:reciprocityHIJ}, we find that
        $\zeta^{\no}_{\mff_{2,d}}(\varq^{-1},t^{-1})$ equals
	\begin{equation*}
		(-1)^{d+d'} 
		\sum_{w\in \Perm_d}
		\varq^{\ell(w)}
		\sum_{v\in \Perm_{d'}}
		\varq^{\ell(v)} 
		\sum_{ \substack{[d-1]\backslash \Des(w) \subseteq I \subseteq [d-1]
				\\ [d'-1]\backslash \Des(v) \subseteq J \subseteq [d'-1] } }
		\chi_{\no}(X_d Y_{d'} H_{I,J}(\Xtup,\Ytup)).
	\end{equation*}
	Using the two equations in \eqref{eq:relationsw0w}, $\zeta^{\no}_{\mff_{2,d}}(\varq^{-1},t^{-1})$ are
	\begin{equation*}
		(-1)^{d+d'} 
		\sum_{w\in \Perm_d}
		\varq^{d'-\ell(ww_0)}
		\sum_{v\in \Perm_{d'}}
		\varq^{\binom{d'}{2}-\ell(vv_0)} 
		\sum_{ \substack{\Des(ww_0) \subseteq I \subseteq [d-1]
				\\ \Des(vv_0) \subseteq J\subseteq [d'-1] } }
		\chi_{\no}(X_d Y_{d'} H_{I,J}(\Xtup,\Ytup)).
	\end{equation*}
	Changing the order of summation again results in
	\begin{equation*}
		\begin{split}
			\zeta^{\no}_{\mff_{2,d}}(\varq^{-1},t^{-1})
			= &(-1)^{d+d'}\varq^{d'+\binom{d'}{2}}
			\sum_{\substack{ I\subseteq [d-1] \\ J\subseteq[d'-1]}}
			\lowBracketLeft\sum_{\substack{ w\in \Perm_d,\\\Des(ww_0)\subseteq I}}
			\varq^{-\ell(ww_0)}\lowBracketRight \\
			&\lowBracketLeft \sum_{ \substack{ v\in \Perm_{d'},\\ \Des(vv_0)\subseteq J}}
			\varq^{-\ell(vv_0)}\lowBracketRight
			\chi_{\no}(X_d Y_{d'} H_{I,J}(\Xtup,\Ytup)).
		\end{split}
	\end{equation*}
	Using \eqref{eq:GaussianMultinomialDescentSet} and $\chi_{\no}(X_d Y_{d'})=\varq^{dd'}t^{d+d'}$ yields
	\begin{equation*}
		\zeta^{\no}_{\mff_{2,d}}(\varq^{-1},t^{-1})
		=(-1)^{d+d'} \varq^{d'+\binom{d'}{2}+dd'} 
		\sum_{\substack{ I\subseteq [d-1] \\ J\subseteq[d'-1]}}
		\binom{d}{I}_{\varq^{-1}}\binom{d'}{J}_{\varq^{-1}} 
		\chi_{\no}(H_{I,J}(\Xtup,\Ytup)).\label{eq:functional_equation}
	\end{equation*} 
	Lastly, using $d'+\binom{d'}{2}+dd'=D$ and~\Cref{thm:novlp_zeta} yields \eqref{eq:funeq.novlp}.
\end{proof}

In light of results such as \cite[Prop.~4.19]{CSV/24}, one might
expect that the functional equation established
in~\Cref{thm:funeq.novlp} might hold for all
$\zeta^{w}_{\mff_{2,d}}(\varq,t)$ where $w\in\mcD_{2d'}$, not just for
$w=\bfz^{d'}\bfo^{d'}$.  Our explicit calculations
(see~\Cref{sec:expl-form}) show that this indeed holds for $d\leq 4$.
\begin{con}
  For all $d\geq 2$ and all $w\in\mcD_{2d'}$, the overlap type zeta
  function $\zeta^{w}_{\mff_{2,d}(\lri)}(s)$ satisfies the functional
  equation
	\begin{equation*} \label{eq:funeq.ovlp}
		\zeta^{w}_{\mff_{2,d}}(\varq^{-1},t^{-1})
		=(-1)^{D} \varq^{\binom{D}{2}}t^{D}\zeta^{w}_{\mff_{2,d}}(\varq,t).
	\end{equation*}
\end{con}

\subsection{The simple pole at zero of the no-overlap zeta function} \label{sec:simplepole}
Next, we study the behaviour of $\zeta^{\no}_{\mff_{2,d}(\lri)}(s)$ at $s=0$.
\begin{thm} \label{thm:simple.pole.novlp}
	The no-overlap zeta function $\zeta^{\no}_{\mff_{2,d}(\lri)}(s)$ has a simple pole
	at $s=0$ for all but finitely many $\cardres$.
\end{thm}
\begin{proof}
	We start from the formula for 
	$\zeta^{\no}_{\mff_{2,d}}(q,t)$ stated in~\Cref{thm:novlp_zeta}.
	Let $I\subseteq [d-1]$ and $J\subseteq [d'-1]$. 
	\Cref{pro:HIJasprI} implies that $H_{I,J}(\Xtup,\Ytup)
	=I_{E_{\no},A_{I,J},C_{I,J}}(\Xtup,\Ytup,1)$.	
	Let $\Gamma_{I,J}=\{ K_u \mid u\in U_{I,J} \}$ be a family of simplicial monoids, satisfying the conditions in~\Cref{pro:TriangulationIEAC} with $E=E_{\no}$, $A=A_{I,J}$, and $C=C_{I,J}$.
	Using that $H_{I,J}(\Xtup,\Ytup)
	=I_{E_{\no},A_{I,J},C_{I,J}}(\Xtup,\Ytup,1)$ and $\bigcup_{u\in U_{I,J}} \overline{K_u}=I_{E_{\no},A_{I,J},C_{I,J}}$ in~\Cref{thm:novlp_zeta} results in
	\begin{equation} \label{eq:NoOverlapSimplicial}
		\zeta^{\no}_{\mff_{2,d}}(\varq,t)
		=\sum_{I\subseteq [d-1],J\subseteq [d'-1]} 
		\binom{d}{I}_{\varq^{-1}} \binom{d'}{J}_{\varq^{-1}} 
		\sum_{u\in U_{I,J}}
		\chi_{\no}(\overline{K_u}(\Xtup,\Ytup,1)).
	\end{equation}	
	By~\Cref{theorem_generating_function_simplicial_cones}, 
	\begin{equation} \label{eq:closedKu}
		\overline{K_u}(\Ztup) 
		= \frac{
			\sum_{\beta\in D_{\overline{K_u}}} \Ztup^{\beta}
		}{
			\prod_{i=1}^{r}(1-\Ztup^{\alpha_i})
		},
	\end{equation}
	where $\alpha_1$, \dots, $\alpha_r$ are quasigenerators of
        $K_u$, $D_{\overline{K_u}}$ is defined in \eqref{def:D.circ},
        and $\Ztup=(\Xtup,\Ytup,1)$.  Therefore
        $\chi_{\no}(\overline{K_u}(\Ztup))|_{\varq\to \cardres, t\to
          \cardres^{-s}}$ has a pole at $s=0$ if and only if there is
        a quasigenerator $\gamma$ of $K_u$ such that $\chi_{\no}(
        \Ztup^{\gamma})=\chi_{\no}( X_1^{\gamma_1} \dots
        Y_{d'}^{\gamma_{d+d'}})$ is a power of $t$. In this case there
        is no cancelation between these factors because of the
        positivity of the numerator.  Looking at~\Cref{dfn:chino},
        this means that the monomial $\Ztup^{\gamma}$ has to have
        degree zero in the variables $X_1$,\dots,$X_{d-1}$,
        $Y_1$,\dots, $Y_{d'}$.  Thus the support of $\gamma$ has to be
        contained in $\{ d,d+d'+1\}$.  We show that there is a $K_u$
        that has a quasigenerator whose support is contained in $\{
        d,d+d'+1\}$.
	
	For each $i\in [d+d'+1]$, let $\delta_i\in \N_0^{d+d'+1}$ be the
        $i$th unit basis vector.  As discussed in~\Cref{sec:Eno},
        $\delta_d+2\delta_{d+d'+1}$ is a completely fundamental
        element of $E_{\no}$ and therefore
        by~\Cref{pro:TriangulationIEAC}, it is a quasigenerator for
        some of the $K_u$.  There cannot be more than one
        quasigenerator of $\mcC_{E_{\no}}$ with the same support and
        thus the multiplicity of the pole at $s=0$ of
        $\chi_{\no}(\overline{K_u}(\Ztup))|_{\varq\to \cardres, t\to
          \cardres^{-s}}$ is at most one.  The order of a pole of a
        sum is at most the maximal order of the poles of the summands.
        Therefore by \eqref{eq:NoOverlapSimplicial},
        $\zeta^{\no}_{\mff_{2,d}(\lri)}(s)$ can have at most a simple
        pole at $s=0$ (the Gaussian binomial coefficients do not
        depend on $s$ and therefore they have no poles or zeros).
	
	It remains to show that the residues of the summands in
        \eqref{eq:NoOverlapSimplicial} do not cancel each other out
        except for possibly finitely many~$\cardres$.  To that end, it
        suffices to show that
	\begin{equation} \label{eq:contributionsum}
			\sum_{I\subseteq [d-1],J\subseteq [d'-1]}  \binom{d}{I}_{\varq^{-1}} \binom{d'}{J}_{\varq^{-1}} 
			\lim_{s\to 0} s \sum_{u\in U_{I,J}}
			\left.\chi_{\no}(\overline{K_u}(\Ztup))\right|_{t\to \varq^{-s}}
		\end{equation}
	is a non-zero rational function in $\varq$
        \textcolor{black}{divided} by $\log \varq$.  The Gaussian
        multinomial co\-ef\-fi\-cients are polynomials in $\varq^{-1}$
        with non-negative coefficients.  By the reasoning earlier in
        this proof,
        $\left.\chi_{\no}(\overline{K_u}(\Ztup))\right|_{\varq\to
          \cardres, t\to \cardres^{-s}}$ can have at most a simple
        pole at $s=\nobreak0$. If it has no pole, then $\lim_{s\to 0}
        s \sum_{u\in U_{I,J}}
        \left.\chi_{\no}(\overline{K_u}(\Ztup))\right|_{t\to
          \varq^{-s}}$ is zero.  Otherwise, \eqref{eq:closedKu} shows
        that the numerator of $\chi_{\no}(\overline{K_u}(\Ztup))$ is a
        polynomial in $\varq$ and $t$ with non-negative coefficients.
        It also shows that the denominator of
        $\chi_{\no}(\overline{K_u}(\Ztup))$ is a product of $\dim K_u$
        polynomials of the form $(1-\varq^{a}t^{b})$ with $a,b\in
        \color{black}{\N}$.  Because we assume that
        $\left.\chi_{\no}(\overline{K_u}(\Ztup))\right|_{\varq\to
          \cardres, t\to \cardres^{-s}}$ has a simple pole at $s=0$,
        exactly one of these factors has $a=\nobreak0$.  Therefore
        multiplying the denominator by $s^{-1}$ and taking the limit
        $s\to 0$ results in a product of polynomials of the form
        $(1-\varq^{a})$ \textcolor{black}{divided by $\log \varq$}.
        Thus in this case, $\lim_{s\to 0} s \sum_{u\in U_{I,J}}
        \left.\chi_{\no}(\overline{K_u}(\Ztup))\right|_{t\to
          \varq^{-s}}$ is a non-zero rational function in~$\varq$
        multiplied by $\log \varq$.  Thus we find that
        \eqref{eq:contributionsum} is indeed a rational function in
        $\varq$ \textcolor{black}{divided} by $\log \varq$.  To show
        that it is non-zero, notice that when evaluated at
        $q=\frac{1}{2}$, both the Gaussian multinomial coefficients,
        the numerators and the factors $(1-\varq^{a})$ in de
        denominators of the $\lim_{s\to 0} s \sum_{u\in U_{I,J}}
        \left.\chi_{\no}(\overline{K_u}(\Ztup))\right|_{t\to
          \varq^{-s}}$ are non-negative numbers, while $\log 1/2$ is a
        negative number.  Therefore the summands in
        \eqref{eq:contributionsum} all have the same sign when
        evaluated in $q=1/2$ and cannot completely cancel out, and the
        claim follows since there is a non-zero summand.
\end{proof}

\subsection{The simple pole at zero of the overlap and subalgebra zeta functions}\label{subsec:si.po}

We study the behaviour of $\zeta_{\mff_{2,d}(\lri)}(s)$ at $s=0$ 
by first looking at the behaviour of 
$\zeta_{\mff_{2,d}(\lri)}(s)-\zeta^{\no}_{\mff_{2,d}(\lri)}(s)$
at $s=0$.

\begin{thm}\label{thm:simple.pole.ovlp}
	The function
        $\zeta_{\mff_{2,d}(\lri)}(s)-\zeta^{\no}_{\mff_{2,d}(\lri)}(s)$
        does not have a pole at $s=0$.
\end{thm}

\begin{proof}
As
$\zeta_{\mff_{2,d}(\lri)}(s)-\zeta^{\no}_{\mff_{2,d}(\lri)}(s)=\sum_{w\in\mcD_{2d'},w\neq
  \bfz^{d'}\bfo^{d'}}\zeta^w_{\mff_{2,d}(\lri)}(s)$,
\Cref{thm:overlap} implies that
	\begin{equation*}\label{eq:zeta-zetan.o.}
		\zeta_{\mff_{2,d}}(\varq,t)-\zeta^{\no}_{\mff_{2,d}}(\varq,t)=\sum_{\substack{(I,\sigma)\in\mcW_d, w_{\sigma}\neq \bfz^{d'}\bfo^{d'}}} \GMC_{I,\sigma} \chi_{\sigma}(G_{I,\sigma}(\Xtup,\Ytup)).
	\end{equation*}
Therefore it suffices to show that for each pair $(I,\sigma)\in
\mcW_{d}$ with $w_{\sigma}\neq \bfz^{d'} \bfo^{d'}$, the rational
function $\chi_{\sigma}(G_{I,\sigma}(\Xtup,\Ytup))|_{\varq\to
  \cardres, t\to \cardres^{-s}}$ has no pole at $s=0$.  Recall
from~\Cref{pro:prIequalsG} that $G_{I,\sigma}$ can be seen as the
projection of $I_{E_{\sigma},A_{I,\sigma},C_{I,\sigma}}$ on the first
$d+d'$ coordinates, thus $G_{I,\sigma}(\Xtup,\Ytup)
=I_{E_{\sigma},A_{I,\sigma},C_{I,\sigma}}(\Xtup,\Ytup,\mathbf{1})$.
Let $\Gamma_{I,\sigma}=\{ K_u \mid u\in U_{I,\sigma} \}$ be as
in~\Cref{dfn:GammaISigma}.  Using that $\bigcup_{u\in U_{I,\sigma}}
\overline{K_u}=I_{E_{\sigma},A_{I,\sigma},C_{I,\sigma}}$, we find
	\begin{equation*} \label{eq:GISigmaSimplicial}
		\chi_{\sigma}(G_{I,\sigma}(\Xtup,\Ytup))
		= \sum_{u\in U_{I,\sigma}}
		\chi_{\sigma}(\overline{K_u}(\Xtup,\Ytup,\mathbf{1})).
	\end{equation*}	
	By~\Cref{theorem_generating_function_simplicial_cones},
        \textcolor{black}{$\chi_{\sigma}(\overline{K_u}(\Xtup,\Ytup,\mathbf{1}))$
          has a denominator of the form
          $\prod_{i=1}^{r}(1-\Ztup^{\alpha_i})$,} where $\alpha_1$,
        \dots, $\alpha_r$ are quasigenerators of $K_u$ and
        $\Ztup=(\Xtup,\Ytup,\mathbf{1})$.  Thus
        $\chi_{\sigma}(\overline{K_u}(\Ztup))|_{\varq\to \cardres,
          t\to \cardres^{-s}}$ has a pole at $s=0$ if and only if
        there is a quasigenerator~$\beta$ of $K_u$ such that
        $\chi_{\sigma}( X_1^{\beta_1} \dots Y_{d'}^{\beta_{d+d'}})$ is
        a power of~$t$.  By the definition of $\chi_\sigma$
        in~\Cref{dfn:num.data}, this happens only when the support of
        $\beta$ is contained in $\{d\}\cup\{d+d'+i \mid i\in
        [r_\sigma]\}$.
	
	Recall that $E_{\sigma}$ is the monoid associated with the matrix 
	$\Phi_{\sigma}\in\Mat_{r_{\sigma}\times m_{\sigma}}(\Z)$ from~\Cref{dfn:PhiISigma}.
	As $w_{\sigma}\neq \bfz^{d'} \bfo^{d'}$, there is an $h\in[2d'-1]$ 
	such that the $h$-th letter of $w_{\sigma}$ is $\bfo$ and the 
	$(h+1)$-th letter of $w$ is $\bfz$. 
	Let $i,j,k$ be such that $\sigma(h)=k\in [d']$ and 
	$b^{-1}( \sigma(h+1) )=(i,j)\in [d]^2$.
	Then $\Phi_{\sigma}$ has a row 
	\begin{equation}\label{matrix_row_1}
		[0^{(i-1)},(-1)^{(j-i)},{(-2)^{(d-j+1)}},0^{(k-1)},1^{(d'-k+1)},0,\ldots,0,-1,0,\ldots,0].
	\end{equation}
	Note that \eqref{matrix_row_1} has the entry $-1$ in column $d$,
	another $-1$ in one other column contained in $\{d+d'+i \mid i\in [r_\sigma]\}$,
	and zero in the other columns contained in $\{d+d'+i \mid i\in [r_\sigma]\}$.
	Therefore multiplying the row \eqref{matrix_row_1} 
	with $\beta$ could not result in zero
	when the support of $\beta$ is contained in 
	$\{d\}\cup\{d+d'+i \mid i\in [r_\sigma]\}$. 
	Consequently, such a tuple~$\beta$ could not satisfy $\Phi_{\sigma} \beta = 0$,
	and therefore not be an element of $K_u\subseteq E_{\sigma}$.
	Thus we find that $\chi_\sigma(\overline{K_u}(\Xtup,\Ytup,\mathbf{1}))|_{\varq\to \cardres, t\to \cardres^{-s}}$
	and by \eqref{eq:GISigmaSimplicial} also $\chi_\sigma(G_{I,\sigma}(\Xtup,\Ytup))|_{\varq\to \cardres, t\to \cardres^{-s}}$ cannot have a pole at $s=0$.
\end{proof} 

\begin{thm}\label{thm:simple.pole}
The function $\zeta_{\mff_{2,d}(\lri)}(s)$ has a simple pole at $s=0$
for all but finitely many~$\cardres$.
\end{thm}
\begin{proof}
	Follows directly from~\Cref{thm:simple.pole.ovlp} and~\Cref{thm:simple.pole.novlp}.
\end{proof}
This result confirms the first part of \cite[Conjecture~IV
($\mathfrak{P}$-adic form)]{Rossmann/15} for the relevant zeta
functions (for all but possibly a finite number of $\cardres$). (The second part is known to hold for $d\in\{2,3,4\}$, by
inspection of the explicit formulas; see
\Cref{cor:pad.zero}.)

\section{Reduced and topological zeta functions}\label{sec:gen.res}
We discuss the reduced and topological zeta functions
$\zeta^{\red}_{\mff_{2,d}}(t)$ and $\zeta^{\topo}_{\mff_{2,d}}(s)$.
Theorems~\ref{thm:Reduced} and \ref{pro:top_zeta} provide formulas for
$\zeta^{\red}_{\mff_{2,d}}(t)$ and $\zeta^{\topo}_{\mff_{2,d}}(s)$
respectively. ~\Cref{thm:pole_reduced} establishes that the reduced
zeta function $\zeta_{\mff_{2,d}}^{\red}(t)$ has a simple pole of
order $\binom{d+1}{2}$ at $t=1$. Theorems~\ref{thm:deg_topo}
and~\ref{thm:top_pole.simple} confirm parts of con\-jec\-tures from
\cite{Rossmann/15} pertaining to the degree and pole at $s=0$ of
topological subalgebra zeta functions, in the relevant special cases.

\subsection{Preliminary definitions}
In preparation, we make a few preliminary definitions. First, we define the counterpart of $\GMC_{I,\sigma}$ that we will use to formulate formulas for $\zeta^{\red}_{\mff_{2,d}}(t)$ and $\zeta^{\topo}_{\mff_{2,d}}(s)$.
\begin{dfn} \label{dfn:MC}
	For $(I,\sigma)\in \mcW_d$, the \emph{product of multinomial coefficients} associated with $(I,\sigma)$ is $\MC_{I,\sigma}:=\GMC_{I,\sigma}|_{\varq\to 1}$.
\end{dfn}

Second, we define integers $a_{\sigma}(\alpha)$ and $b_{\sigma}(\alpha)$ for all $\sigma\in \Spec_{2d'}$ and $\alpha\in \N_0^{m_{\sigma}}$ that are closely related to the numerical data map~$\chi_{\sigma}$.

\begin{dfn} \label{dfn:absigmaalpha}
	Let $\sigma\in \Spec_{2d'}$ and $\alpha\in \N_0^{m_{\sigma}}$. 
	Define $a_{\sigma}(\alpha)$ and $b_{\sigma}(\alpha)$ to be the respectively non-negative and positive integers such that
	$\chi_{\sigma}( (\Xtup,\Ytup,\mathbf{1})^{\alpha})=(1-\varq^{a_{\sigma}(\alpha)}t^{b_{\sigma}(\alpha)})$.	
\end{dfn}
\begin{exm}
	Let $\sigma=21$, $\alpha_1:=(0,1,2,0)$, and $\alpha_2:=(0,1,0,2)$. Then
	\begin{align*}
		\chi_{\sigma}( (\Xtup,\Ytup,\mathbf{1})^{\alpha_{1}})
		&=\chi_{\sigma}( X_1^{0} X_2^{1} Y_1^{2} 1^{0})
		= (\varq^{1}t)^{0} (t^{2})^{1} (\varq^{2}t)^{2} 1^{0}
		= q^{4}t^4, \\
		\chi_{\sigma}( (\Xtup,\Ytup,\mathbf{1})^{\alpha_{2}})
		&=\chi_{\sigma}( X_1^{0} X_2^{1} Y_1^{0} 1^{2})
		= (\varq^{1}t)^{0} (t^{2})^{1} (\varq^{2}t)^{0} 1^{2}
		=t^2.
	\end{align*}
	Thus $a_{\sigma}(\alpha_1)=4$, $b_{\sigma}(\alpha_1)=4$, $a_{\sigma}(\alpha_2)=0$, and $b_{\sigma}(\alpha_2)=2$.
\end{exm}

Recall~\Cref{dfn:GammaISigma} of $\Gamma_{I,\sigma}=\{ K_u \mid u\in
U_{I,\sigma}\}$.  Third, we define a subset $U_{I,\sigma,\max}$ of
$U_{I,\sigma}$ and a positive rational number $c_d$ for each $d\in
\N_{\geq 2}$.

\begin{dfn} \label{dfn:UISigmaMax_and_cd}
	For each $(I,\sigma)\in \mcW_d$, let $U_{I,\sigma,\max}$ be the set of $u\in U_{I,\sigma}$ such that $\dim K_u = d+d'=\binom{d+1}{2}$. 
	Let $c_d$ be the positive rational number
	\begin{equation}\label{eq:cd}
		c_d:=\sum_{(I,\sigma)\in\mcW_d}
		\MC_{I,\sigma} \sum_{u\in U_{I,\sigma,\max}}
		\frac{|D_{\overline{K_u}}|}{\prod_{\alpha\in \textup{CFE}(K_u)}b_{\sigma}(\alpha)},
	\end{equation}
	where $D_{\overline{K_u}}$ is defined in \eqref{def:D.circ}.
\end{dfn}

\begin{rem}
	Although it is not a priori clear, $c_d$ does not depend on the family $\Gamma_{I,\sigma}=\{ K_u \mid u\in U_{I,\sigma}\}$ of simplicial monoids, but only on $d$. 
	This is a consequence of~\Cref{thm:pole_reduced}. 
\end{rem}

We lack a conceptual interpretation of $c_d$. In~\Cref{tab:magic} we
list $c_d$ for $d\leq 6$.

\newcolumntype{C}{>{$}c<{$}}
\begin{table}[ht]
	\begin{center}
		\begin{tabular}{C|CCCCC}
			d&2&3&4&5&6\\  \hline 
			\raisebox{0pt}[12pt][0pt]{$c_d$}& \frac{3}{4} & 
			\frac{25}{54} & \frac{569}{2304} & 
			\frac{3800243}{32400000} & \frac{8743819}{172800000}
		\end{tabular}
	\end{center}
	\caption{Values of $c_d$ for $d\in\{2,3,4,5,6\}$.}
	\label{tab:magic}
\end{table}
 
\subsection{Reduced zeta functions} \label{sec:red} The reduced zeta function
$\zeta_{\mff_{2,d}}^{\red}(t)$ can be obtained by substituting $\varq\to1$ in
$\zeta_{\mff_{2,d}}(\varq,t)$. \textcolor{black}{(For a general definition of
  reduced subalgebra zeta functions in terms of Euler characteristics and
  motivic zeta functions see \cite[Sec.~3]{Evseev/09}.)} We straightforwardly
adapt~\Cref{thm:main} to a formula for $\zeta_{\mff_{2,d}}^{\red}(t)$
and then use it to determine the order and residue of the pole at $t=1$ of
$\zeta_{\mff_{2,d}}^{\red}(t)$. We first define a reduced counterpart of the
numerical data map.
\begin{dfn} \label{dfn:chi_red}
	The \emph{reduced numerical data map} $\chi_{\red}$ is
	\begin{equation*}
		\chi_{\red}: \textcolor{black}{\mathcal{A}}\to \Q(t): \quad X_i \mapsto t^i, Y_j\mapsto t^j.
	\end{equation*}
\end{dfn}
The following theorem is a straightforward adaption
of~\Cref{thm:main}.
\begin{thm} \label{thm:Reduced}
	For all $d\in \N_{\geq2}$,
	\begin{equation*} \label{eq:Reduced} 
		\zeta_{\mff_{2,d}}^{\red}(t) =
		\sum_{(I,\sigma)\in \mathcal{W}_d} \MC_{I,\sigma}
		\chi_{\red} (G_{I,\sigma}(\Xtup,\Ytup)).
	\end{equation*}	
\end{thm}
\begin{proof}
	Follows from~\Cref{thm:main} after substituting $\varq\to1$ on both sides of \eqref{equ:sum}.
\end{proof}

Next, we use this formula to deduce the order and residue of the pole
at $t=1$ of $\zeta_{\mff_{2,d}}^{\red}(t)$.  Recall that $D = d+d' =
\binom{d+1}{2}$ is the $\Z$-rank of $\mff_{2,d}$.

\begin{thm}\label{thm:pole_reduced}
The reduced zeta function $\zeta_{\mff_{2,d}}^{\red}(t)$ has a pole at
$t=1$ of order $D$ with residue
\begin{equation*}
 \lim_{t\to 1}\left((t-1)^{D}\zeta^{\red}_{\mff_{2,d}}(t)\right) =
 (-1)^{D} c_d,
\end{equation*}
where $c_d$ is defined in~\Cref{dfn:UISigmaMax_and_cd}.
\end{thm}

\begin{proof}
 Recall from~\Cref{pro:prIequalsG} that $G_{I,\sigma}$ can be seen as
 the projection of the subset
 $I_{E_{\sigma},A_{I,\sigma},C_{I,\sigma}}\in \N_0^{m_{\sigma}}$ on
 the first $d+d'$ coordinates, or equivalently
 $G_{I,\sigma}(\Xtup,\Ytup)
 =I_{E_{\sigma},A_{I,\sigma},C_{I,\sigma}}(\Xtup,\Ytup,\mathbf{1})$.
 Let $\Gamma_{I,\sigma}=\{ K_u \mid u\in U_{I,\sigma}\}$ be as
 in~\Cref{dfn:GammaISigma}.  Using that $\bigcup_{u\in U_{I,\sigma}}
 \overline{K_u}=I_{E_{\sigma},A_{I,\sigma},C_{I,\sigma}}$
 in~\Cref{thm:Reduced} results in
	\begin{equation} \label{eq:ReducedSimplicial}
		\zeta_{\mff_{2,d}}^{\red}(t) =
		\sum_{(I,\sigma)\in \mathcal{W}_d} \MC_{I,\sigma}
		\sum_{u\in U_{I,\sigma}}
		\chi_{\red}(\overline{K_u}(\Xtup,\Ytup,\mathbf{1})).
	\end{equation}	

	By~\Cref{theorem_generating_function_simplicial_cones}, 
	\begin{equation*} \label{eq:Iglo}
		\overline{K_u}(\Ztup) 
		= \frac{
			\sum_{\beta\in D_{\overline{K_u}}} \Ztup^{\beta}
		}{
			\prod_{\alpha \in \textup{CFE}(K_u)}(1-\Ztup^{\alpha})
		},
	\end{equation*}
	where $D_{\overline{K_u}}$ is defined in \eqref{def:D.circ}
        and $\Ztup=(\Xtup,\Ytup,\mathbf{1})$.  Applying $\chi_{\red}$
        on both sides yields	\begin{equation*}
		\chi_{\red} \left( \overline{K_u}(\Ztup) \right)
		= \frac{
			\sum_{\beta\in D_{\overline{K_u}}} \chi_{\red} (\Ztup^{\beta})
		}{
			\prod_{\alpha \in \textup{CFE}(K_u)}(1-t^{b_{\sigma}(\alpha)})
		},
	\end{equation*}
	where $b_{\sigma}(\alpha)$ was defined
        in~\Cref{dfn:absigmaalpha}.  It follows that
        $\chi_{\red}(\overline{K_u}(\Xtup,\Ytup,\mathbf{1}))$ has a
        pole at $t=1$ of order $\dim K_u$ and the residue of this pole
        is
	\begin{equation} \label{eq:residueK_u}
		\lim_{t\to 1}\left( (t-1)^{\dim K_u} \overline{K_u}(\Ztup) \right)
		= (-1)^{\dim K_u}
		\frac{
			\lvert D_{\overline{K_u}} \rvert
		}{
			\prod_{\alpha \in \textup{CFE}(K_u)} b_{\sigma}(\alpha)
		}.
	\end{equation}
Recall from~\Cref{rmk:dim_GISigma} that the dimension of
$G_{I,\sigma}$, and therefore also of $I_{E_\sigma,A_{I,\sigma},
  C_{I,\sigma}}$ and $K_u$, is at most $D$.  Thus the summand indexed
by $u\in U_{I,\sigma}$ in \eqref{eq:ReducedSimplicial} has a pole at
$t=1$ of order~$D$ if $\dim \mcK_u=D$ and otherwise the pole has a
lower order.  Therefore $\zeta_{\mff_{2,d}}^{\red}(t)$ has a pole at
$t=1$ of order at most $D$.  To prove that the order is exactly $D$,
it suffices to show that the sum of the residues of the summands in
\eqref{eq:ReducedSimplicial} with maximal pole order does not vanish.
In other words, we need to show that summing \eqref{eq:residueK_u}
over all $u\in U_{I,\sigma}$ with $\dim K_u = D$ does not cancel
out. This is of course trivial, because \eqref{eq:residueK_u} has the
same sign for all these $u$. To find the residue of
$\zeta_{\mff_{2,d}}^{\red}(t)$, we may just sum \eqref{eq:residueK_u}
over all these $u$, resulting in $(-1)^D c_d$, where $c_d$ was defined
in~\Cref{dfn:UISigmaMax_and_cd}.
\end{proof}

\subsection{Topological zeta functions}\label{subsec:topo_zeta_funct}
Next, we study the topological zeta
function~$\zeta_{\mff_{2,d}}^{\topo}(s)$.  We give a formula
in~\Cref{pro:top_zeta} and determine its degree
in~\Cref{thm:deg_topo}.  In~\Cref{thm:topo_infi}, we link its
behaviour at infinity to the behaviour at $t=1$ of
$\zeta_{\mff_{2,d}}^{\red}(t)$.
In~\Cref{sec:BehaviourZeroTopological} we determine the behaviour of
$\zeta_{\mff_{2,d}}^{\topo}(s)$ at~$s=0$.

In~\Cref{subsec:main.res}, we informally introduced the topological
zeta function $\zeta^{\topo}_{\mff_{2,d}}(s)$ as the rational function
in $s$ obtained as the first non-zero coefficient of the $\mfp$-adic
zeta function $\zeta_{\mff_{2,d}}(q_{\lri},q_{\lri}^{-s})$, expanded in
$q_{\lri}-1$. \textcolor{black}{(For a general definition of topological
  subalgebra zeta functions in terms of systems of local zeta
  functions see \cite[Def.~5.13]{Rossmann/15}.)} More precisely,
\begin{equation*}
	\zeta^{\topo}_{\mff_{2,d}}(s):= \lim\limits_{\varq\to 1} (\varq-1)^{\textup{rank } \mff_{2,d}} \zeta_{\mff_{2,d}}(\varq,\varq^{-s}).
\end{equation*}
For example, for $a\in \N_0$ and $b\in \N$,
\begin{equation} \label{eq:example_top}
  \lim\limits_{\varq\to 1} (\varq-1) \frac{1}{1-\varq^{a-bs}}  = \frac{1}{bs-a}.
\end{equation}

The following theorem is an adaption of~\Cref{thm:main} 
to a formula for $\zeta^{\topo}_{\mff_{2,d}}(s)$.
\begin{thm} \label{pro:top_zeta}
	\begin{equation}\label{eq:top_zeta}
		\zeta^{\topo}_{\mff_{2,d}}(s)
		= \sum_{(I,\sigma)\in \mathcal{W}_d} \MC_{I,\sigma}
		\sum_{u\in U_{I,\sigma,\max}}
		\frac{|D_{\overline{K_u}}|}{\prod_{\alpha\in \textup{CFE}(K_u)}(b_{\sigma}(\alpha)s-a_{\sigma}(\alpha))}.
	\end{equation}
\end{thm}
\begin{proof}
\textcolor{black}{As in the proof of \Cref{thm:pole_reduced}, we
  leverage \Cref{pro:prIequalsG}. Applying it to \Cref{thm:main}
  yields} 
	\begin{equation*} \label{eq:PadicSimplicial}
		\zeta_{\mff_{2,d}}(q,t) =
		\sum_{(I,\sigma)\in \mathcal{W}_d} \GMC_{I,\sigma}
		\sum_{u\in U_{I,\sigma}}
		\chi_{\sigma}(\overline{K_u}(\Xtup,\Ytup,\mathbf{1})).
	\end{equation*}	
	By~\Cref{theorem_generating_function_simplicial_cones}, 
	\begin{equation} \label{eq:closedKutop}
		\chi_\sigma(\overline{K_u}(\Ztup) )
		= \frac{
			\sum_{\beta\in D_{\overline{K_u}}} \chi_\sigma(\Ztup^{\beta})
		}{
			\prod_{\alpha\in \textup{CFE}(K_u)}(1-\chi_\sigma(\Ztup^{\alpha}))
		},
	\end{equation}
	where $\Ztup=(\Xtup,\Ytup,\mathbf{1})$.  Note that
        $\lim\limits_{\varq\to 1} (\varq-1)^{d+d'} \chi_\sigma(
        \overline{K_u} (\Ztup) )$ vanishes if $\dim K_u < d+d'$, that
        is, $u\in U_{I,\sigma}\backslash U_{I,\sigma,\max}$.  The
        numerator of \eqref{eq:closedKutop} is
        $|D_{\overline{K_u}}|$ after substituting $\varq\to 1$.  For
        each $u\in U_{I,\sigma,\max}$, the denominator of
        \eqref{eq:closedKutop} is $\prod_{\alpha\in
          \textup{CFE}(K_u)}(b_{\sigma}(\alpha)s-a_{\sigma}(\alpha))$
        after multiplication with $(\varq-1)^{-d-d'}$ and taking the
        limit $\varq\to 1$, cf.\ \eqref{eq:example_top}
        and~\Cref{dfn:absigmaalpha}.
\end{proof}

The degree of a rational function is the degree of the numerator minus the
degree of the denominator.  The following confirms \cite[Conj.~I]{Rossmann/15}
for the considered algebras.
\begin{thm}\label{thm:deg_topo}
 The topological zeta function $\zeta_{\mff_{2,d}}^{\topo}(s)$ has
 degree $-D$ in $s$:
\begin{equation*}
 \deg_s\left(\zeta_{\mff_{2,d}}^{\topo}(s)\right) = -D.
\end{equation*}
\end{thm}

\begin{proof}
All summands in \eqref{eq:top_zeta} have degree $-D$ in $s$.  For each
$u\in U_{I,\sigma,\max}$, the numerator of the summand corresponding
to $u$ is $\MC_{I,\sigma}|D_{\mcK_u}|$, which is always positive.
Similarly, the highest degree coefficient of the denominator of the
summand corresponding to $u$ is $\prod_{\alpha\in
  \textup{CFE}(K_u)}b_{\sigma}(\alpha)$, which is also always
positive.  Therefore cancellation of the highest degree terms in
\eqref{eq:top_zeta} is not possible.
\end{proof}

Next, we study the behaviour of the topological zeta function at infinity.
\begin{thm}
	\label{thm:topo_infi} 
	The topological zeta function $\zeta_{\mff_{2,d}}^{\topo}(s)$
        satisfies
	\begin{equation*}          
		\lim\limits_{s\to 0}
		s^{-D}\zeta_{\mff_{2,d}}^{\topo}(s^{-1})=c_d,\label{equ:topo.infi}
	\end{equation*}
	where $c_d$ is defined in~\Cref{dfn:UISigmaMax_and_cd}.
\end{thm}

\begin{proof} 
 Substituting $s^{-1}$ for $s$ in the summand corresponding to $u\in
 U_{I,\sigma,\max}$ in \eqref{eq:top_zeta}, multiplying by $s^{-D}$,
 and taking the limit $s\to 0$ results in
	\begin{align*}
		\lim\limits_{s\to 0} s^{-D}  \frac{|D_{\overline{K_u}}|}{\prod_{\alpha\in \textup{CFE}(K_u)}(b_{\sigma}(\alpha)s^{-1}-a_{\sigma}(\alpha))}
		&=\lim\limits_{s\to 0} \frac{|D_{\overline{K_u}}|}{\prod_{\alpha\in \textup{CFE}(K_u)}(b_{\sigma}(\alpha)-a_{\sigma}(\alpha)s)} \nonumber \\
		&=\frac{|D_{\overline{K_u}}|}{\prod_{\alpha\in \textup{CFE}(K_u)}b_{\sigma}(\alpha)}.
	\end{align*}
	Therefore, using \eqref{eq:top_zeta}, we find
	\begin{equation*}
		\lim\limits_{s\to 0}
		s^{-D}\zeta^{\topo}_{\mff_{2,d}}(s^{-1}) 
		=
		\sum_{\substack{(I,\sigma)\in\mcW_d}}
		\MC_{I,\sigma}
		\sum_{u\in U_{I,\sigma,\max}} 
		\frac{|D_{\overline{K_u}}|}{\prod_{\alpha\in \textup{CFE}(K_u)}b_{\sigma}(\alpha)},
	\end{equation*}
	which is  the definition of $c_d$ in \eqref{eq:cd}.
\end{proof}
The following corollary shows that the behaviour at infinity of $\zeta_{\mff_{2,d}}^{\topo}(s)$ 
is closely related to the behaviour at $t=1$ of $\zeta_{\mff_{2,d}}^{\red}(t)$.
\begin{cor} \label{cor:topo_infi_red_one}
	\begin{equation*}
		\lim\limits_{s\to 0}
		s^{-D}\zeta_{\mff_{2,d}}^{\topo}(s^{-1})=\lim\limits_{t\to1}
		(1-t)^{D}\zeta_{\mff_{2,d}}^{\red}(t).
	\end{equation*}
\end{cor}
\begin{proof}
	Combine~\Cref{thm:topo_infi} with~\Cref{thm:pole_reduced}.
\end{proof}
\Cref{cor:topo_infi_red_one} may be compared with
\cite[Conj.~6.7]{LeeVoll/18}, which describes an analogous phenomenon
for topological and reduced zeta functions associated with graded
ideal zeta functions of free nilpotent Lie rings of arbitrary rank and
nilpotency class.

\subsection{Behaviour at zero of the topological zeta function} \label{sec:BehaviourZeroTopological}
In~\Cref{thm:top_pole.simple}, we show that the topological zeta
function $\zeta^{\topo}_{\mff_{2,d}}(s)$ has a simple pole at $s=0$,
just as the $\mfp$-adic zeta function $\zeta_{\mff_{2,d}(\lri)}(s)$
for all but possibly a finite number of $\cardres$
(see~\Cref{thm:simple.pole}), and determine the residue at this pole.
This theorem confirms \cite[Conj.~IV (Topological form)]{Rossmann/15}
for the considered subalgebra zeta functions.

\begin{thm}\label{thm:top_pole.simple}
	The topological zeta function $\zeta^{\topo}_{\mff_{2,d}}(s)$ has a simple
	pole at $s=0$ with residue 
	\begin{equation*}
		\lim\limits_{s\to 0} s\zeta^{\topo}_{\mff_{2,d}}(s) = \frac{(-1)^{D-1}}{\left(D-1\right)!}.
	\end{equation*}
\end{thm}

\begin{proof}
	We start from the formula for $\zeta^{\topo}_{\mff_{2,d}}(s)$
        in~\Cref{pro:top_zeta}.  Let $(I,\sigma)\in \mathcal{W}_d$ and
        $u\in U_{I,\sigma,\max}$.  The summand corresponding to $u$ in
        \eqref{eq:top_zeta} has numerator $\MC_{I,\sigma} \lvert
        D_{\overline{K_u}} \rvert$ and denominator $\prod_{\alpha\in
          \textup{CFE}(K_u)}(b_{\sigma}(\alpha)s-a_{\sigma}(\alpha))$.
        This numerator is a positive rational number and the
        denominator is zero at $s=0$ if and only if there is a
        $\alpha\in \textup{CFE}(K_u)$ such that $a_{\sigma}(\alpha)$
        is zero.  Recall from~\Cref{dfn:absigmaalpha} that
        $a_{\sigma}(\alpha)$ is the non-negative integer such that
        $\chi_{\sigma}(
        (\Xtup,\Ytup,\mathbf{1})^{\alpha})=(1-\varq^{a_{\sigma}(\alpha)}t^{b_{\sigma}(\alpha)})$.
        Looking at~\Cref{dfn:chi_sigma}, $a_{\sigma}(\alpha)$ is zero
        if and only if the support of $\alpha$ is contained in $\{d\}
        \cup \{ d+d'+i \mid i\in [r_\sigma]\}$.  Thus the summand
        corresponding to $u$ in \eqref{eq:top_zeta} has a pole at
        $s=0$ if and only if there is a completely fundamental element
        of $K_u$ with such a support.  From the proof of
        ~\Cref{thm:simple.pole.ovlp} we know that this cannot happen
        if $w_\sigma$ is not the trivial Dyck word. From the proof of
        ~\Cref{thm:simple.pole.novlp} we know that if $w_{\sigma}$ is
        the trivial Dyck word, then there are $K_u$ that have such a
        completely fundamental element, and, moreover, no $K_u$ can
        have more than one such completely fundamental element.  Thus
        the summand corresponding to $u$ in \eqref{eq:top_zeta} has at
        most a simple pole at $s=0$, and therefore
        $\zeta^{\topo}_{\mff_{2,d}}(s)$ has at most a simple pole
        there.  Let
	\begin{equation*}
		U_{0,\max}:=\{ u\in U_{I,\sigma,\max} \mid (I,\sigma)\in \mathcal{W}_d, \delta_d+2\delta_{d+d'+1}\in K_u \}.
	\end{equation*}
	The summand corresponding to $u$ in \eqref{eq:top_zeta} has a simple pole at $s=0$ if and only if $u\in U_{0,\max}$. 
	By the reasoning earlier in this proof, if $u\in U_{0,\max}$ then $u\in U_{I,\sigma,\max}$ with $w_\sigma=\bfz^{d'} \bfo^{d'}$. Also $I=[d-1]$ and $J_\sigma=[d'-1]$, because otherwise $U_{I,\sigma,\max}$ is empty. It follows using~\Cref{lem:GMCnooverlap} that if $u\in U_{0,\max}\cap U_{I,\sigma}$, then
	\begin{equation*}
		\MC_{I,\sigma}=\binom{d}{[d-1]} \binom{d'}{[d'-1]} = d! d'!.
	\end{equation*}
	The residue of $\zeta^{\topo}_{\mff_{2,d}}(s)$ at $s=0$ is therefore
	\begin{equation*}
		\lim\limits_{s\to 0} s\zeta^{\topo}_{\mff_{2,d}}(s)
		= d!d'! \lim\limits_{s\to 0} s
		\sum_{\substack{ (I,\sigma)\in \mathcal{W}_d \\ w_\sigma=\bfz^{d'} \bfo^{d'} }}
		\sum_{u\in U_{I,\sigma,\max}} 
		\frac{|D_{\overline{K_u}}|}{\prod_{\alpha\in \textup{CFE}(K_u)}(b_{\sigma}(\alpha)s-a_{\sigma}(\alpha))}.
	\end{equation*}
	Recall that 
	\begin{equation*}
		\frac{|D_{\overline{K_u}}|}{\prod_{\alpha\in \textup{CFE}(K_u)}(b_{\sigma}(\alpha)s-a_{\sigma}(\alpha))}
		= \lim\limits_{\varq\to 1} (\varq-1)^{d+d'} 
		\chi_\sigma( \overline{K_u} (\Xtup, \Ytup, 1) )|_{t\to \varq^{-s}}
	\end{equation*}
	and $\bigcup_{u\in U_{I,\sigma}}
        \overline{K_u}=I_{E_{\sigma},A_{I,\sigma},C_{I,\sigma}}$ where
        the union is disjoint. Thus
	\begin{equation*}
		\lim\limits_{s\to 0} s\zeta^{\topo}_{\mff_{2,d}}(s)
		= d!d'! \lim\limits_{s\to 0} s \lim\limits_{\varq\to 1} (\varq-1)^{d+d'}
		\sum_{\substack{ (I,\sigma)\in \mathcal{W}_d \\ w_\sigma=\bfz^{d'} \bfo^{d'} }}
		\chi_\no( I_{E_{\no},A_{I,J_{\sigma}},C_{I,J_{\sigma}}} (\Xtup, \Ytup, 1) )|_{t\to \varq^{-s}}.
	\end{equation*}
	Recall from~\Cref{rem:decomposition_mcK} that $\bigcup_{(I,\sigma)\in \mathcal{W}_d, w_\sigma=\bfz^{d'} \bfo^{d'}} I_{E_{\no},A_{I,J_{\sigma}},C_{I,J_{\sigma}}} = E_{\no}$ and this union is disjoint.
	Therefore the residue can be written as
	\begin{equation*}
		\lim\limits_{s\to 0} s\zeta^{\topo}_{\mff_{2,d}}(s)
		= d!d'! \lim\limits_{s\to 0} s \lim\limits_{\varq\to 1} (\varq-1)^{d+d'}
		\chi_\no( E_{\no} (\Xtup, \Ytup, 1) )|_{t\to \varq^{-s}}.
	\end{equation*}
	
	Recall the subset $E_0$ of $E_{\no}$ from~\Cref{dfn:mcC0} and
        consider the triangulation $\Gamma=\{ K_u \mid u\in U_{\no}
        \}$ of $E_{\no}$ in~\Cref{pro:special_tria}. Let $U_0:=\{ u\in
        U_{\no} \mid K_u\subseteq E_0\}$ and $U_0^c:=\{ u\in U_{\no}
        \mid K_u\not\subseteq E_0\}$. Then $E_0=\bigcup_{u\in U_{0}}
        \overline{K}_u$ and this union is disjoint.  Thus the residue
        is
	\begin{equation*}
		\lim\limits_{s\to 0} s\zeta^{\topo}_{\mff_{2,d}}(s)
		= \left.d!d'! \lim\limits_{s\to 0} s \lim\limits_{\varq\to 1} (\varq-1)^{d+d'}
		\chi_\no\left( E_{0} (\Xtup, \Ytup, 1) + \sum_{u\in U_0^c}  \overline{K_u} (\Xtup, \Ytup, 1) \right) \right|_{t\to \varq^{-s}}.
	\end{equation*}
	The triangulation $\Gamma=\{ K_u \mid u\in U_{\no} \}$ was constructed so that the $K_u$ for $u\in U_0^c$ do not contain $\delta_d+2\delta_{d+d'+1}$. 
	Therefore by the same reasoning as in the proof of~\Cref{thm:top_pole.simple}, the summands $\chi_\no( \overline{K_u} (\Xtup, \Ytup, 1) )$ do not contribute to the residue.
	For every $\alpha\in \textup{CFE}(E_0)$, let $a_{\no}(\alpha)$ and $b_{\no}(\alpha)$ be
	the respectively non-negative and positive integers such that
	$\chi_{\no}( (\Xtup,\Ytup,1)^{\alpha_{i}})=(1-\varq^{a_{\no}(\alpha)}t^{b_{\no}(\alpha)})$.
	As $E_0$ is simplicial (see~\Cref{rem:mcC0simplicial}) we may use
	\Cref{theorem_generating_function_simplicial_cones} to deduce
	\begin{equation*}
		\lim\limits_{\varq\to 1} (\varq-1)^{d+d'} \chi_\no( E_{0} (\Xtup, \Ytup, 1) |_{t\to \varq^{-s}}
		=\frac{\lvert D_{E_0} \rvert}{\prod_{\alpha\in \textup{CFE}(E_0)}(b_{\no}(\alpha)s-a_{\no}(\alpha))}.
	\end{equation*}
	Thus the residue is
	\begin{align} \label{eq:residue2}
		\lim\limits_{s\to 0} s\zeta^{\topo}_{\mff_{2,d}}(s)
		&= d!d'! \lim\limits_{s\to 0} s 
		\frac{\lvert D_{E_0} \rvert}{\prod_{\alpha\in \textup{CFE}(E_0)}(b_{\no}(\alpha)s-a_{\no}(\alpha))}
		\nonumber\\
		&=(-1)^{D-1} d!d'! \frac{\lvert D_{E_0} \rvert}{b_{\no}(\delta_d+2\delta_{d+d'+1})\prod_{\alpha\in \textup{CFE}(E_0)\backslash \{\delta_d+2\delta_{d+d'+1}\} }a_{\no}(\alpha)}.
	\end{align}
The completely fundamental elements $\alpha$ of $E_0$ and the
corresponding data $a_{\no}(\alpha)$ and $b_{\no}(\alpha)$ are listed
in~\Cref{tab:CFEE0}.
\begin{table}
	\begin{tabular}{l|lll}
		$\alpha\in \textup{CFE}(E_0)$ & 
		$\chi_{\no}( (\Xtup,\Ytup,1)^{\alpha})$
		& $a_{\no}(\alpha)$ & $b_{\no}(\alpha)$ \\ \hline
		$\delta_i$ with $i\in [d-2]$ & $(1-\varq^{i(d-i)}t^{i})$ & $i(d-i)$ & $i$ \\
		$\delta_{d-1}+\delta_{d+d'+1}$ & $(1-\varq^{d-1}t^{d-1})$ & $d-1$ & $d-1$ \\
		$\delta_d+2\delta_{d+k}$ with $k\in [d']$ & $(1-\varq^{2k(d+d'-k)}t^{d+2k})$ & $2k(d+d'-k)$ & $d+2k$ \\
		$\delta_d+2\delta_{d+d'+1}$ & $(1-t^{d})$ & $0$ & $d$
	\end{tabular}
	\caption{The completely fundamental elements of $E_0$ and the corresponding $a_{\no}(\alpha)$ and $b_{\no}(\alpha)$.}
	\label{tab:CFEE0}
\end{table}
Using this data, the denominator in \eqref{eq:residue2} is
\begin{equation*}
	d \left( \prod_{i\in[d-2]}i(d-i)\right)
	\left( d-1 \right)
	\left( \prod_{k\in[d']} 2k(d+d'-k) \right) = 2^{d'}(d)! (d')! (d+d'-1)!. 
\end{equation*}

Lastly, we determine $\lvert D_{E_0} \rvert$.
To do this, we need to count the number of elements $x\in E_0$
that can be written as a $\Q$-linear combination of the completely fundamental elements
of $E_0$ with coefficients in $[0,1)$.
Since $\delta_i$ is the only completely fundamental element of $E_0$
with support containing $\{i\}$ for $i\in [d-2]$, 
we deduce that the coefficient of $\delta_i$ needs to be zero for $i\in [d-2]$. 
Similarly, $\delta_{d-1}+\delta_{d+d'+1}$ is the only 
completely fundamental element with support containing $\{d-1\}$, 
thus the coefficient of $\delta_{d-1}+\delta_{d+d'+1}$ is zero as well. 
Also, $\delta_d+2\delta_{i+1}$ is the only 
completely fundamental element with support containing $\{ i+1 \}$ for $i\in d-1+[d']$,
thus its coefficient lies in $\{0,1/2\}$.
Since the coefficient of $\delta_{d-1}+\delta_{d+d'+1}$ is zero, 
$\delta_d+2\delta_{d+d'+1}$ is the only remaining 
completely fundamental element with support containing $\{ d+d'+1 \}$.
Therefore its coefficient lies in $\{0,1/2\}$ as well.
Thus we find that 
\begin{equation*}
	D_{E_0}=E_0\cap \left\{ \sum_{i\in d-1+[d'+1] } a_{i} (\delta_d+2\delta_{i+1}) \middle| a_i\in \{0,1/2\} \right\}.
\end{equation*}
Obviously $\sum_{i\in d-1+[d'+1] } a_{i} (\delta_d+2\delta_{i+1})$ lies in $\N_0^{d+d'+1}$ if and only if $\sum_{i\in d-1+[d'+1] } a_{i}\in \N$, in other words when an even number of $a_i$ are non-zero.
Moreover, in that case, it also lies in $E_0$, 
thus we find that $\lvert D_{E_0} \rvert = 2^{d'}$.

Substituting this data in \eqref{eq:residue2} allows us to conclude
\begin{equation*} 
	\lim\limits_{s\to 0} s\zeta^{\topo}_{\mff_{2,d}}(s)
	=(-1)^D d!d'! \frac{2^{d'}}{2^{d'}(d)! (d')! (d+d'-1)!}
	=\frac{ (-1)^D }{ (D-1)! }. \qedhere
\end{equation*}
\end{proof}

\section{Explicit computations}\label{sec:expl-form}

We record (aspects of) explicit computations of the $\mfp$-adic,
reduced, and topological zeta functions.  The full results are
available at
\href{https://doi.org/10.5281/zenodo.7966735}{10.5281/zenodo.7966735}.
We start by collecting the well-known formulas for
$d=2,3$. \label{subsec:dleq3}
\begin{pro}[$d=2$] \label{pro:d2}
\begin{align*}
\zeta_{\mff_{2,2}}(\varq,t) &
=\frac{1-\varq^3t^3}{(1-\varq^3t^2)(1-\varq^2t^2)(1-t)(1-\varq t)}, \\ 
\zeta_{\mff_{2,2}}^{\topo}(s)
&=\frac{3}{2(2s - 3)(s - 1)s}, \\ 
\zeta_{\mff_{2,2}}^{\red}(t) &
=\frac{t^2 + t + 1}{(1-t^2)^2(1-t)}.
\end{align*}
\end{pro}

\begin{proof} 
	The $\mfp$-adic formula was given in \cite[Prop.~8.1]{GSS/88}, the others follow immediately. 
\end{proof}

\begin{pro}[$d=3$] \label{pro:d3}
\begin{equation*}
	\zeta_{\mff_{2,3}}(\varq,t) = \frac{(1-\varq^8t^4) W_{2,3}(\varq,t)}
	{(1-t)(1-\varq t)(1-\varq^2t)(1-\varq^4t^2)(1-\varq^5t^2)(1-\varq^6t^2)(1-\varq^6t^3)(1-\varq^7t^3)},
\end{equation*}
where $W_{2,3}(X,Y)$ is
\begin{align*}
 	&1 + X^3Y^2 + X^4 Y^2 + X^5 Y^2 - X^4 Y^3 - X^5 Y^3
	- X^6 Y^3 - X^7 Y^4 - X^9 Y^4 \nonumber \\
	& - X^{10} Y^5 - X^{11} Y^5 - X^{12} Y^5
	+ X^{11} Y^6 + X^{12} Y^6 + X^{13} Y^6 + X^{16} Y^8.
\end{align*}
Furthermore
\begin{align*}
	\zeta_{\mff_{2,3}}^{\topo}(s) 
	&= \frac{25s^2 - 94s + 84}{3(3s - 7)(3s - 8)(2s - 5)(s - 1)(s - 2)^2(s - 3)s},\\
	\zeta_{\mff_{2,3}}^{\red}(t) 
	&= \frac{t^8 + 2t^7 + 7t^6 + 9t^5 + 12t^4 + 9t^3 + 7t^2 + 2t + 1}{(1-t^3)^3(1-t^2)^2(1-t)}.
\end{align*}
\end{pro}

\begin{proof}
  The $\mfp$-adic formula was given in \cite[Thm.~24]{Taylor/01}, the
  others follow immediately.
\end{proof}

\label{subsec:d=4}
In \cite{Barvinok1994}, an algorithm is presented to write the
generating function enumerating integral points of a convex pointed
polyhedral cone in a closed form.  This algorithm is implemented in
the software package {\tt LattE} \cite{LattE}, which can be accessed
in {\tt SageMath} \cite{SageMath93} through the package {\tt Zeta}
\cite{Zeta041}.  By this route, we were able to
implement~\Cref{thm:main} and recover the explicit expressions for
$\zeta_{\mff_{2,d}}(\varq,t)$ for $n=2,3$ in Propositions \ref{pro:d2}
and~\ref{pro:d3}. Moreover, we were also able to obtain an explicit
expression for $\zeta_{\mff_{2,4}}(\varq,t)$, which was not known
before.

\begin{thm}[$d=4$, $\mfp$-adic]\label{thm:d=4.pad} There is an explicitly determined polynomial
  $\Psi_{2,4}(X,\allowbreak Y)\allowbreak \in\Z[X,Y]$ of degrees $335$ in $X$ and $88$ in $Y$
  such that
\begin{equation}\label{equ:f24.para}
\zeta_{\mff_{2,4}(\lri)}(\varq,t)=\frac{\Phi_{2,4}(\varq,t)}{\Psi_{2,4}(\varq,t)},
\end{equation}
where $\Psi_{2,4}(\varq,t)$ is
\begin{align*}
&(1-\varq^{27}t^7)(1-\varq^{25}t^7)(1-\varq^{25}t^6)(1-\varq^{28}t^7)(1-\varq^{22}t^5)^2(1-\varq^{21}t^5)(1-\varq^{17}t^4)\nonumber\\
&(1-\varq^{15}t^4)(1-\varq^{13}t^4)(1-\varq^{26}t^6)(1-\varq^{13}t^3)(1-\varq^{11}t^3)(1-\varq^{18}t^4)(1-\varq^9t^2)\\
&(1-\varq^{12}t^3)(1-\varq^{24}t^6)(1-\varq^{16}t^4)(1-\varq^{14}t^4)(1-\varq^9t^3)(1-\varq^{12}t^4)(1-qt)(1-t).\nonumber
\end{align*}

\end{thm}

\begin{cor}[$d=4$, reduced]\label{cor:d=4.red}
	\begin{equation*}
		\zeta_{\mff_{2,4}}^{\red}(t) =
		\frac{\Phi^{\red}_{2,4}(t)}{(1-t)^2(1-t^3)^4(1-t^4)^4},
	\end{equation*}
where $\Phi^{\red}_{2,4}(t)$ is 
\begin{align*}
& t^{20} + 2t^{19} + 15t^{18} + 30t^{17} +
87t^{16} + 156t^{15} + 284t^{14} + 414t^{13} + 562t^{12} +
658t^{11} \nonumber \\ 
&+ 703t^{10} + 658t^9 + 562t^8 + 414t^7 + 284t^6 + 156t^5 +
87t^4 + 30t^3 + 15t^2 + 2t + 1.
\end{align*}    
\end{cor}

\begin{proof} Substitute $\varq=1$ in \eqref{equ:f24.para}. Alternatively,
  explicate \cite[Prop.~4.1]{Evseev/09}.\end{proof}

\begin{thm}[$d=4$, topological]\label{thm:d=4.top}
  \begin{align*}
  	\Phi^\topo_{2,4}(s)/(168\,\zeta_{\mff_{2,4}}^\topo(s)) 
  	=&(7s - 25)(7s - 27)(6s - 25)(5s - 21)(5s - 22)^2 \nonumber\\
  &(4s - 13)(4s - 15)(4s - 17)(3s - 11)(3s - 13)^2\\ 
    &(2s - 7)(2s - 9)^2(s - 1)(s - 3)^2(s - 4)^4s, \nonumber
    \end{align*}
where $ \Phi^\topo_{2,4}(s)$ is
 \begin{align*}
 &21078036000s^{13} - 1040066363064s^{12} + 23656166485364s^{11}  \nonumber\\
 & - 328379597912246s^{10} + 3103756047141233s^9 - 21092307321737791s^8  \nonumber \\
 & + 106022910302150804s^7 - 399106101276334990s^6 + 1125038325014124489s^5  \\
 & - 2345400850582061927s^4 + 3514612915281294714s^3   \nonumber \\
 & - 3584726815997417886s^2 + 2230351512292203300s -639268261271640000. \nonumber
 \end{align*}
\end{thm}

The computation of the $\mfp$-adic zeta function for $d=5$ is
currently out of our reach, owing to memory limitations. We are,
however, able to compute the reduced zeta function
$\zeta_{\mff_{2,5}}^{\red}(s)$ (using Evseev's method;
\textcolor{black}{note that \cite[Prop.~4.1]{Evseev/09} applicable
  according to \cite[Exa.~4.1]{Evseev/09})} and the topological zeta
function $\zeta_{\mff_{2,5}}^{\topo}(s)$ (using our method).

\begin{thm}[$d=5$, reduced; \cite{Evseev/09}]
	\begin{equation*}
		 \zeta_{\mff_{2,5}}^{\red}(t) =
		\frac{\Phi^{\red}_{2,5}(t)}{(1-t^5)^5(1-t^3 )^5(1-t^4)^4(1-t)},
	\end{equation*}
where $\Phi^{\red}_{2,5}(t)$ is
\begin{align*}
& t^{42} + 4t^{41} + 30t^{40} + 115t^{39} +
431t^{38} + 1330t^{37} + 3709t^{36} + 9185t^{35} + 20876t^{34} \nonumber \\
&+
43410t^{33} + 83737t^{32} + 150127t^{31} + 252056t^{30} +
397040t^{29} + 589457t^{28} \nonumber \\
&+ 826057t^{27} + 1095916t^{26} +
1377780t^{25} + 1644507t^{24} + 1864452t^{23} + 2010117t^{22} \nonumber \\
&+ 2060784t^{21}
+ 2010117t^{20} + 1864452t^{19} + 1644507t^{18} +
1377780t^{17} + 1095916t^{16} \\
& + 826057t^{15} + 589457t^{14} +
397040t^{13} + 252056t^{12} + 150127t^{11} + 83737t^{10} \nonumber\\
& + 43410t^{9} + 20876t^8 + 9185t^7 + 3709t^6 + 1330t^5 + 431t^4 + 115t^3 +
30t^2 + 4t + 1. \nonumber
\end{align*}    
\end{thm}

\begin{thm}[$d=5$, topological]\label{thm:d=5.top}
	\begin{equation*}
		\zeta_{\mff_{2,5}}^{\topo}(s) =
		\frac{\Phi^{\topo}_{2,5}(s)}{\Psi^{\topo}_{2,5}(s)},
	\end{equation*}
	where
	$\Phi^{\topo}_{2,5}(s)\in\Z[s]$ is an explicitly determined
	irreducible polynomial of degree $71$ and $\Psi^{\topo}_{2,5}(s)$ is
	\begin{align*}
		&(38s - 225)(37s - 223)(35s - 216)(31s - 199)(31s - 200)(29s - 189)(29s - 190) \nonumber \\
		&(26s - 165)(25s - 153)(25s - 161)(25s - 166)(23s - 151)(23s - 153)(22s - 141) \nonumber \\
		&(22s - 145)(21s - 130)(20s - 131)(19s - 112)(19s - 122)(17s - 93)(17s - 108) \nonumber \\
		&(17s - 112)(17s - 113)(15s - 89)(14s - 85)(13s - 70)(13s - 81)(13s - 82) \nonumber \\
		&(13s - 88)(12s - 77)(11s - 71)(11s - 72)(10s - 63)^2(9s - 44)(9s - 46)(9s - 47)  \\
		&(9s - 55)(9s - 58)^2(9s - 59)(8s - 45)(8s - 51)(8s - 53)^2(7s - 41)(7s - 43)^2 \nonumber \\
		&(7s - 46)^2(6s - 37)(5s - 21)(5s - 22)(5s - 23)(5s - 24)(5s - 31)(5s - 32) \nonumber \\
		&(5s - 33)^2(4s - 21)(4s - 23)^3(4s - 25)(3s - 14)(3s - 16)(3s - 17)(3s - 19)^2 \nonumber \\
		&(3s - 20)^2(2s - 11)^2(2s - 13)^3(s - 1)(s - 2)(s - 3)(s - 4)^2(s - 6)^4s. \nonumber
	\end{align*}
\end{thm}

For $d=6$, the computation of both the $\mfp$-adic and the topological
zeta function is currently out of our reach. We record the explicit
formula for the reduced zeta function, computed using Evseev's method.

\begin{thm}[$d=6$, reduced; \cite{Evseev/09}]
\begin{equation*}
	\zeta_{\mff_{2,6}}^{\red}(t) =
	\frac{\Phi^{\red}_{2,6}(t)}{(1-t^61)^6(1-t^5 )^6(1-t^4)^6(1-t)^{3}},
\end{equation*}
where $\Phi^{\red}_{2,6}(t)$ is
\begin{align*}
  & t^{72} + 3t^{71} + 36t^{70} + 145t^{69} + 669t^{68} + 2562t^{67} + 8649t^{66} + 27045t^{65} + 77670t^{64} \\
  & + 206735t^{63}+ 515748t^{62} + 1211748t^{61} + 2692110t^{60} + 5682609t^{59}  \nonumber \\
  & + 11436687t^{58} + 22007442t^{57} + 40598238t^{56} + 71961840t^{55} + 122797673t^{54}  \nonumber \\
  & + 202076190t^{53} + 321171642t^{52} + 493662867t^{51} + 734688480t^{50} + 1059758436t^{49} \nonumber \\
  & + 1482992565t^{48} + 2014885665t^{47} + 2659813131t^{46} + 3413604248t^{45} + 4261613451t^{44}  \nonumber \\
  & + 5177738109t^{43} +6124749888t^{42} + 7056165426t^{41} + 7919643378t^{40} + 8661618634t^{39}  \nonumber \\
  & + 9232638888t^{38} +9592688376t^{37} + 9715718352t^{36} + 9592688376t^{35} + 9232638888t^{34} \nonumber \\
  & + 8661618634t^{33} + 7919643378t^{32} + 7056165426t^{31} + 6124749888t^{30} + 5177738109t^{29} \nonumber \\
  & + 4261613451t^{28} + 3413604248t^{27} + 2659813131t^{26} +  2014885665t^{25} + 1482992565t^{24} \nonumber \\
  & + 1059758436t^{23} + 734688480t^{22} + 493662867t^{21} + 321171642t^{20} + 202076190t^{19} \nonumber \\
  & + 122797673t^{18} + 71961840t^{17} + 40598238t^{16} + 22007442t^{15} + 11436687t^{14} \nonumber \\
  & + 5682609t^{13} + 2692110t^{12} + 1211748t^{11} + 515748t^{10} + 206735t^9  \nonumber \\
  & + 77670t^8 + 27045t^7 + 8649t^6 + 2562t^5 + 669t^4 + 145t^3 + 36t^2 + 3t + 1.\nonumber 
\end{align*}
\end{thm}

Our computations of $\mfp$-adic zeta functions allow us to confirm the second
part of \cite[Conjecture~IV ($\mathfrak{P}$-adic form)]{Rossmann/15} for
small values of $d$.

\begin{cor}
	\label{cor:pad.zero}
	For all $d\in\{2,3,4\}$, the following holds:
	\begin{equation*}
		\left.\frac{\zeta_{\mff_{2,d}(\lri)}(s)}{\zeta_{\lri^{D}}(s)}\right|_{s=0}=1,
	\end{equation*}
\end{cor}
The conjecture's first part holds for all $d$ and all but a finite
number of $\cardres$, see~\Cref{thm:simple.pole}.

\begin{acknowledgements}
  This work forms part of the second author's doctoral dissertation,
  supervised by the third author.  We would like to thank Tobias Rossmann for
  pointing out to us how to use \texttt{Zeta} \cite{Zeta041} to efficiently
  write large sums of rational functions of a specific form on a common
  denominator. \textcolor{black}{We are grateful to an anonymous referee, whose
    comments and corrections added great value to the paper.} This work was
  partly funded by the Deutsche Forschungsgemeinschaft (DFG, German Research
  Foundation) — SFB-TRR 358/1 2023 — 491392403.
\end{acknowledgements}

\def\cprime{$'$}
\providecommand{\bysame}{\leavevmode\hbox to3em{\hrulefill}\thinspace}
\providecommand{\MR}{\relax\ifhmode\unskip\space\fi MR }
\providecommand{\MRhref}[2]{%
  \href{http://www.ams.org/mathscinet-getitem?mr=#1}{#2}
}
\providecommand{\href}[2]{#2}

\end{document}